\documentclass[reqno,11pt]{amsart}
\usepackage{amssymb}
\usepackage[letterpaper,margin=1in]{geometry}
\usepackage[showonlyrefs]{mathtools}  
\usepackage{latexsym}
\usepackage[hidelinks]{hyperref}
\usepackage{amsmath}
\usepackage{amscd}
\usepackage{cite}
\usepackage{enumitem}
\usepackage{color}
\usepackage{amsfonts}
\usepackage{graphicx}
\usepackage{mathrsfs}
\usepackage{setspace}
\newcommand{\eps}{{\varepsilon}}
\theoremstyle{plain}
\newtheorem{theorem}{Theorem}
\newtheorem{proposition}[theorem]{Proposition}
\newtheorem{lemma}[theorem]{Lemma}
\newtheorem{corollary}[theorem]{Corollary}

\theoremstyle{definition}
\newtheorem{definition}[theorem]{Definition}
\newtheorem{remark}[theorem]{Remark}

\newcommand{\qtq}[1]{\quad\text{#1}\quad}
\numberwithin{equation}{section}
\numberwithin{theorem}{section}

\def\ge{\geqslant}
\def\le{\leqslant}
\def\geq{\geqslant}
\def\leq{\leqslant}

\def\R{\mathbb{R}}
\def\C{\mathbb{C}}
\def\eps{\varepsilon}

\allowdisplaybreaks[2]

\begin{document}

\title[GWP and scattering for mass-critical Hartree equation]{Global well-posedness and scattering for mass-critical Hartree equation}

\author[Z. Ma]{Zuyu Ma}
\address{Zuyu Ma, Graduate School of China Academy of Engineering Physics, Beijing 100088, China}
\email{mazuyu23@gscaep.ac.cn}

\author[C. Miao]{Changxing Miao}
\address{School of Mathematics and Physics, University of Science and Technology Beijing, Beijing 100083, China}
\email{miao\_changxing@ustb.edu.cn, miao\_changxing@iapcm.ac.cn}

\author[M. Rosenzweig]{Matthew Rosenzweig}
\address{Matthew Rosenzweig, Department of Mathematical Sciences, Carnegie Mellon University, Pittsburgh, PA, USA}
\email{mrosenz2@andrew.cmu.edu}

\author[J. Zheng]{Jiqiang Zheng}
\address{Jiqiang Zheng, Institute of Applied Physics and Computational Mathematics and National Key Laboratory of Computational Physics, Beijing 100088, China}
\email{zheng\_jiqiang@iapcm.ac.cn}

\maketitle 
\begin{abstract}
In this paper, we prove that the defocusing mass-critical Hartree equation is global well-posed and scatters for any initial data $u_0\in L^2$, and prove the analogous result in the focusing case, provided the mass of $u_0$ is strictly less than that of the ground state $Q$. This result removes the radial restriction from \cite{MXZ-JMPA}, thereby proving the scattering conjecture for the Hartree equation at the scaling-critical regularity. 
\end{abstract}

\section{Introduction}
We consider the mass-critical Hartree equation:
\begin{equation}\label{NLH}
\begin{cases}
i\partial_{t}u+\Delta u=\mu F(u),\\
u(0,x)=u_{0}(x)\in L^2(\R^{d}),
\end{cases}
\end{equation}
Throughout this paper, $d\geq3$. Here $F(u):=(V(x)\ast|u|^2)u$, $V=|x|^{-2}$ and $\mu\in\{\pm1\}$. When $\mu=1$, we say that the equation \eqref{NLH} is defocusing, and when $\mu=-1$, we say that the equation is focusing.

The Hartree equation arises in the study of quantum many-body systems with a large number of particles. More precisely, equation \eqref{NLH} is the mean field limit of a system of $N$ bosons in d-space with two-body interactions governed by the pair potential (see \cite{arxiv1}) 
\[
V(x) := \mu / |x|^2.
\]
The solution to the Hartree equation \eqref{NLH} is then an effective description of the dynamics of a single boson in a certain scaling limit as the number of particles $N$ tends to infinity. To our knowledge, a rigorous derivation of equation \eqref{NLH} from the underlying $N$-body Schr\"odinger problem is lacking due to the singularity of $V$, but we refer the reader to \cite{Ammari,Knowles} for derivations in the case of less singular potentials.

For $\lambda > $0, it is an easy computation that the solution to equation \eqref{NLH} is invariant under the scaling transformation
\[
u(t,x) \mapsto u_{\lambda}(t,x) := \lambda^{d/2}u(\lambda^2 t, \lambda x),
\]
which we hereafter refer to as the Hartree scaling. Additionally, classical solutions to the Hartree equation conserve mass, momentum, and energy, respectively given by the functionals

\[
M(u(t)) := \int_{\mathbb{R}^d} |u(t,x)|^2 dx,
\]
\[
P(u(t)) := 2 \int_{\mathbb{R}^d} \text{Im}\{\bar{u}\nabla u\}(t,x)dx,
\]
\[
E(u(t)) := \frac{1}{2} \int_{\mathbb{R}^d} |\nabla u(t,x)|^2 dx + \frac{\mu}{4} \int_{\mathbb{R}^d} |u(t,x)|^2 (V * |u|^2)(t,x)dx.
\]
Since the Hartree scaling preserves the $L_x^2(\mathbb{R}^d)$  norm of  $u(t)$, and $L^2$ is the regularity of the mass functional, we refer to the Hartree equation \eqref{NLH} as being  $L^2$ - or \emph{mass-critical}. In this article, we are interested in solutions at this critical regularity.

In order to discuss the local Cauchy theory for equation \eqref{NLH}, we must clarify our notion of a solution.  
We work exclusively with the class of strong (also called mild) solutions defined below. Thus, use of the word ``solution'' in the sequel should be understood, unless stated otherwise, as ``strong solution.''
\begin{definition}[Solution]
    We say that a function  $u : I \times \mathbb{R}^d \to \mathbb{C}$  on a time interval $I \subset \mathbb{R}$  is a \emph{strong solution} to equation \eqref{NLH} if it belongs to the function space  $C^0_{t} L_x^2(K \times \mathbb{R}^d) \cap L^{\frac{2(d+2)}{d}}_{t,x}(K \times \mathbb{R}^d)$ for any compact interval $K\subset I$, and obeys the Duhamel formula
\[
u(t_1) = e^{i(t_1-t_0)\Delta} u(t_0) - i\mu \int_{t_0}^{t_1} e^{i(t_1-t)\Delta} F(u(t)) \, dt, \quad t_0, t_1 \in I.
\]
Here, the notation $e^{it\Delta}$  denotes the Fourier multiplier with symbol $ e^{-it|\xi|^2} $. The interval  $I$  is the \emph{lifespan} of the solution, and we say that a solution has \emph{maximal lifespan} if  $I$  is not properly contained in an interval $J \subset \mathbb{R}$. We say that a solution is \emph{global} if  $I = \mathbb{R}$.
\end{definition}

If a solution does not have finite $L_{t,x}^{\frac{2(d+2)}{d}}$  norm on its lifespan, then it is said to blow up, a notion we make precise with the next definition. 
    \begin{definition}[Blowup]
We say that a solution $u$ to \eqref{NLH} blows up forward in time if there exists $t_0\in I$ such that
\[
\int_{t_0}^{\sup I}\int_{\mathbb{R}^d}|u(t,x)|^{\frac{2(d+2)}d}dxdt = +\infty,
\]
and blows up backward in time if
\[
\int_{\inf I}^{t_0}\int_{\mathbb{R}^d}|u(t,x)|^{\frac{2(d+2)}d}dxdt = +\infty.
\]
\end{definition}

The last notion we introduce before proceeding to the local well-posedness theory is that of scattering, which refers to global solutions that asymptotically evolve like solutions to the free Schr\"odinger equation.
\begin{definition}[Scattering]
    Let $ u : I \times \mathbb{R}^d \to \mathbb{C} $ be a solution to \eqref{NLH}. Given an asymptotic state $ u_+ \in L^2(\mathbb{R}^d) $, we say that $ u $ scatters forward in time to $ e^{it\Delta} u_+ $, if $ \sup I = \infty $ and $ \lim\limits_{t \to \infty} \|u(t) - e^{it\Delta} u_+\|_{L_x^2(\mathbb{R}^d)} = 0 $. Similarly, given $ u_- \in L^2(\mathbb{R}^d) $, we say that $ u $ scatters backward in time to $ e^{it\Delta} u_- $, if $ \inf I = -\infty $ and $ \lim\limits_{t \to -\infty} \|u(t) - e^{it\Delta} u_-\|_{L_x^2(\mathbb{R}^d)} = 0 $.
\end{definition}

To our knowledge, the mathematical study of the Cauchy problem for \eqref{NLH} with $L^2(\mathbb{R}^d)$  initial data was initiated by Miao, Xu and Zhao \cite{MXZ-JPDE}, who used the fixed-point method of Cazenave and Weissler \cite{Cazenave} for the NLS to establish local well-posedness for arbitrary initial data in $L^2$, as well as small-data global well-posedness and scattering. However, although their definition of a strong solution differs slightly from ours, it is easy to show through the standard local theory argument that the following local well-posedness theory holds.
\begin{theorem}[Local well-posedness, \cite{MXZ-JPDE}] \label{thm:local-well-posedness}
Fix  $\mu \in \{ \pm 1 \}$. For $u_0 \in L^2(\mathbb{R}^d)$, there exists a unique maximal-lifespan solution  $u : I \times \mathbb{R}^d \to \mathbb{C}$ to \eqref{NLH}. Moreover, the following assertions hold.
\begin{enumerate}
    \item There exists  $m_0>0$, such that if  $\|u_0\|_{L^2(\mathbb{R}^d)} < m_0$, then $u$ is global, \( \|u\|_{L_{t,x}^{\frac{2(d+2)}{d}}(\mathbb{R} \times \mathbb{R}^d)} \lesssim 1 \), and $u$ scatters both forward and backward in time.
    
    \item $I$ is an open interval containing $0$.
    
    \item If $\sup I < \infty$ (resp. $\inf I > -\infty$), then $u$ blows up forward (resp. backward) in time.
    
    \item The solution map $u_0$ from  $L^2(\mathbb{R}^d)$  to  $C_{t,\text{loc}}^0 L_x^2(I \times \mathbb{R}^d) \cap L_{t,x}^{\frac{2(d+2)}{d}}(I \times \mathbb{R}^d)$ is uniformly continuous on compact time intervals for bounded subsets of initial data.
    
    \item If $\sup I = \infty$ (resp. $\inf I = -\infty$ ) and $u$ does not blow up forward (resp. backward) in time, then $u$ scatters forward (resp. backward) in time.
\end{enumerate}
\end{theorem}
\begin{remark}
    The proof is essentially the same as in \cite{MXZ-JPDE}. One can redefine the metric  $d(u,v)$ on the space $Y_{T,\rho}^0$ in \cite[Theorem 3.1]{MXZ-JPDE} as 
   \[
   \left\{ u \in C([0,T],L^2(\mathbb{R}^n)) \cap L_{t,x}^{\frac{2(d+2)}{d}}([0,T]\times\mathbb{R}^d) \; \middle| \; \|u\|_{L_t^\infty L^2\cap L_{t,x}^{\frac{2(d+2)}{d}}([0,T]\times\R^d)} \leq \rho \right\}
   \]
   and
   \[
   \|u-v\|_{L_t^\infty L^2\cap L_{t,x}^{\frac{2(d+2)}{d}}([0,T]\times\R^d)}
   \]
   respectively. 
\end{remark}
Having reviewed the local well-posedness for the Hartree equation, we proceed to discuss possible obstructions to global well-posedness and scattering of solutions with initial data of arbitrarily large mass. From (1) of Theorem \ref{thm:local-well-posedness}, we know that initial data with mass below the threshold $m_0^2$ yields a global solution with finite $L_{t,x}^{\frac{2(d+2)}{d}}$ norm and which scatters both forward and backward in time independent of the value of $\mu$. However, in the focusing case, there is a mass threshold such that finite-time blow up may occur for solutions with mass above that threshold. Indeed, analogous to the nonlinear Schr\"odinger equation, one has the following \emph{virial identity} for the Hartree equation (for instance, see Proposition 3.1 in \cite{MiaoXuZhaoblowup})
\[
\frac{d^2}{dt^2} V(t) = 16E(u(0)),
\]
where $V$ denotes the variance of the solution defined by
\[
V(t) := \int_{\mathbb{R}^d} |x|^2 |u(t,x)|^2 dx, \quad \langle x \rangle u \in L^2(\mathbb{R}^d).
\]

When $\mu = 1$, the energy functional $E$ is obviously positive definite, and one can use this fact to obtain a global solution in $C_t^0 H_x^1(\mathbb{R} \times \mathbb{R}^d)$ for initial data in $H^1(\mathbb{R}^d)$. When $\mu = -1$, the energy functional may be negative, in which case one can use the virial identity together with the standard Zakharov/Glassey argument to show finite-time blow-up. However, the virial identity alone does not determine the long-time dynamical behavior of the solution (i.e., whether it scatters). In the defocusing case, Miao, Xu and Zhao \cite{MXZ-JMPA} proved that any radial solution is globally well-posed and scatters in $L^2$. For the general  data, Chae, Cho and Lee \cite{CCS} used the I-method developed by the I-team \cite{Iteam} to prove that solutions are globally well-posed  in $H^{s}$ with $\frac{4(d-2)}{7d-8}<s<1$. In view of these results, it is natural to expect the solutions to equation \eqref{NLH} to be global and scatter for arbitrary initial data in $L^2$.

In the focusing case, there are known counterexamples to global well-posedness and scattering for \eqref{NLH}. Let $Q$ be the positive radial solution to the elliptic equation
\begin{equation}\label{grounstate}
    \Delta Q - Q + \left(|x|^{-2}*|Q|^2\right)Q = 0. 
\end{equation}
Existence of such solution (known as \emph{ground state}) was proved in \cite{MiaoXuZhaoblowup,Moroz}\footnote{However, to our knowledge, the uniqueness of positive, radial solutions in $H^1(\mathbb{R}^d)$ to equation \eqref{grounstate} is open, except in dimension four (cf. \cite{Krieger})}. Then $e^{it}Q$ is the solution to \eqref{NLH} that blows up both forward and backward in time. Moreover, by applying the pseudoconformal transformation to $e^{it}Q$, we obtain the solution
\[
\frac{1}{t^{\frac{d}{2}}}e^{i\frac{|x|^2-4}{4t}}Q\big(\tfrac{x}{t}\big)
\]
that has the same mass of that of ground state $Q$ and blows up in finite time. It is conjectured that the ground state is the minimal mass threshold to global well-posedness and scattering in the focusing case. To the best of the authors' knowledge, this conjecture has only been verified for radial initial data in Miao, Xu, and Zhao\cite{MXZ-JMPA}.

In this paper, we verify the scattering conjecture for equation \eqref{NLH} in both the defocusing and focusing cases. We state our result as follows:
\begin{theorem}\label{main}
Let $u_0\in L^2(\R^d)$. If $\mu=1$, then the solution to \eqref{NLH} is global and scatters in both time directions. If $\mu=-1$ and $M(u_0)<M(Q)$, then the solution to \eqref{NLH}  is global and scatters in both time directions.
\end{theorem}

\begin{remark}
    By a standard argument (for instance, see the lecture note of Killip--Vi\c{s}an \cite{KV}), the regularity is preserved: if $u_0\in H^s(\mathbb{R}^d)$ for some $s>0$, then $u(t)\in H^s(\mathbb{R}^d)$ for all $t\in\R$. In particular, our main result recovers the global well-posedness result of Chae, Cho and Lee \cite{CCS} in the defocusing case.
\end{remark}
\begin{remark}
    The global spacetime norm (e.g., the $L_{t,x}^{\frac{2(d+2)}{d}}$ norm) is bounded by a constant depending only on the mass: 
    \[
    \|u\|_{L_{t,x}^{\frac{2(d+2)}{d}}(\R\times\R^d)}\leq C(M(u)),
    \]
where $C(\cdot)$ is a nondecreasing function with $C(0)=0$, defined on $[0,\infty)$ for $\mu=1$ and on $[0,M(Q))$ for $\mu=-1$. This is because we adopt the Kenig--Merle concentration--compactness roadmap (cf.~\cite{KM-Invent}) to prove the main theorem; see the detailed discussion below.
\end{remark}

\subsection{Outline of the proof.}
To prove our main result,  we use the formulation of the Kenig-Merle concentration compactness/rigidity road map \cite{KM-Invent} for the mass-critical NLS \cite{TVZ-Forum} and Dodson's long-time Strichartz technique developed in \cite{Dodson-JAMS,Dodson-AJM,Dodson-Duke}. Compared with the mass-critical NLS, the main difficulty in our proof lies in the many challenges posed by the non-local nonlinear (as we will discuss in details later). This work, as well as the third author's previous work \cite{Rosenzweig}, are part of broader effort of making the existing mass-critical global well-posedness and scattering program for nonlinear Schr\"odinger equations, which has its antecedents in the pioneering work of Bourgain \cite{Bourgain}, more robust to non-local nonlinearities.

To facilitate the following discussion, we first give the following  definition.

\begin{definition}[Symmetry group $G$]
For a phase $\theta\in\R/2\pi\mathbb{Z},$ position $x_0\in\R^d,$ frequency $\xi_0\in\R^d$ and scaling parameter $\lambda>0$, we define the unitary transformation $g_{\theta,\xi_0,x_0,\lambda}: L^2(\R^d)\to L^2(\R^d)$ by
\begin{equation}\label{equ:gtrandef}
g_{\theta,\xi_0,x_0,\lambda}f(x):=\lambda^{-\frac{d}{2}}e^{i\theta}e^{ix\cdot\xi_0}f\big(\tfrac{x-x_0}{\lambda}\big).
\end{equation}
We let $G$ denote the collection of such transformations. One may check that $G$ is a group with identity $g_{0,0,0,1}$, inverse $g_{\theta,\xi_0,x_0,\lambda}^{-1}=g_{-\theta-x_0\cdot\xi_0,-\lambda\xi_0,-x_0/\lambda,\lambda^{-1}},$ and group law
\begin{equation}\label{equ:gtrslaw}
g_{\theta,\xi_0,x_0,\lambda}g_{\theta',\xi_0',x_0',\lambda'}
=g_{\theta+\theta'-x_0\cdot\xi_0'/\lambda,\xi_0+\xi_0'/\lambda,x_0+\lambda x_0',\lambda\lambda'}.
\end{equation}
$G$ induces a group action on $L^2(\R^d)$, and we denote the quotient of the action (i.e. the space of $G$-orbits $Gf:=\{gf: g\in G\}$ for $f\in L^2(\R^d)$) by $L^2(\R^d)/G$, which we endow with the quotient metric topology, which is complete. For $g_{\theta,\xi_0,x_0,\lambda}\in G,$ we define the action $T_{g_{\theta,\xi_0,x_0,\lambda}}$ on spacetime function $u:I\times\R^d\to\C$ by 
\begin{align}\label{equ:Tgthedef}
T_{g_{\theta,\xi_0,x_0,\lambda}}:&\;\lambda^2I\times\R^d\to\C,\quad \lambda^2I:=\{\lambda^2t:\;t\in I\}\\\label{equ:Tgtxdef}
\big(T_{g_{\theta,\xi_0,x_0,\lambda}}u\big)(t,x):=&\lambda^{-\frac{d}{2}}
e^{i\theta}e^{ix\cdot\xi_0}e^{-it|\xi_0|^2}u\big(\tfrac{t}{\lambda^2},\tfrac{x-x_0-2\xi_0t}{\lambda}\big),
\end{align}
which may be expressed more succinctly as 
\begin{equation}\label{equ:Tganodef}
\big(T_{g_{\theta,\xi_0,x_0,\lambda}}u\big)(t)=g_{\theta-t|\xi_0|^2,\xi_0,x_0+2\xi_0t,\lambda}\Big(u\big(\tfrac{t}{\lambda^2}\big)\Big).
\end{equation}
\end{definition}
\begin{definition}[Almost periodic modulo symmetries (APMS)] \label{def:apms}
We say that $ u \in C_{t,\text{loc}}^0 L_x^2(I \times \mathbb{R}^d) $ is \emph{almost periodic modulo $ G $} if there exists a spatial center function $ x : I \to \mathbb{R}^d $, a frequency center function $ \xi : I \to \mathbb{R}^d $, a frequency scale function $ N : I \to [0, \infty) $, and a compactness modulus function $ C : (0, \infty) \to [0, \infty) $, such that for every $ \eta > 0 $, we have the spatial and frequency localization estimates
\[
\int_{|x-x(t)| \geq C(\eta)/N(t)} |u(t,x)|^2 \, dx + \int_{|\xi-\xi(t)| \geq C(\eta)N(t)} |\hat{u}(t,\xi)|^2 \, d\xi \leq \eta^2, \quad \forall t \in I. 
\]
\end{definition}

The proof of Theorem \ref{main} proceeds by contradiction. Assuming that Theorem \ref{main} fails, there must exist a special solution $u$ that is almost periodic modulo the group $G$, which is generated by the symmetries of phase rotation, spatial translation, $L^2$ scaling, and Galilean transformation. By this we mean that its orbit is precompact in the quotient space defined by the group action of $G$ on $L^2(\mathbb{R}^d)$. Following well-known nomenclature, we call such a solution as minimal mass blow-up solution. A more precise statement in this regard is as follows.

\begin{theorem}[Reduction to almost-periodic solutions] \label{thm:reduction_ap}
Suppose Theorem \ref{main} fails for $\mu=1$. Then  there exists an almost-periodic maximal-lifespan solution $u: I \times \R^d \rightarrow \mathbb{C}$ to the Hartree equation \eqref{NLH} that blows up forward and backward in time. 
Similarly, suppose Theorem \ref{main} fails for $\mu=-1$, then there also exists a maximal-lifespan solution $u: I \times \R^d \rightarrow \mathbb{C}$ to the Hartree equation \eqref{NLH} that blows up forward and backward in time, and $M(u)<M(Q)$.
\end{theorem}

Next, by following the standard argument (e.g. \cite{Dodson,KV,TVZ-Forum}), we can establish the existence of a subclass of minimal blow-up solutions. We consider these special solutions exclusively in this work.

\begin{theorem}
Suppose Theorem \ref{main} fails for $\mu=1$. Then  there exists an almost-periodic maximal-lifespan solution $u: I \times \R^d \rightarrow \mathbb{C}$ to the Hartree equation \eqref{NLH} that blows up forward  in time, and satisfies the following conditions:
\begin{itemize}
\item[\textup{(i)}]  $[0, \infty) \subset I$;
    
\item[\textup{(ii)}]  $N(t) \leq 1$ for all $t \in[0, \infty)$ and $x(0) = \xi(0) = 0$;
    
\item[\textup{(iii)}] $N(t), \xi(t) \in C_{\mathrm{loc}}^1([0, \infty))$ and satisfy the pointwise derivative bounds
\begin{equation*}
\left|N'(t)\right| + \left|\xi'(t)\right| \lesssim_u N(t)^3, \quad t \in [0, \infty).
\end{equation*}
\end{itemize}
Similarly, suppose Theorem \ref{main} fails for $\mu=-1$, then there also exists a maximal-lifespan solution $u: I \times \R^d \rightarrow \mathbb{C}$ to the Hartree equation \eqref{NLH} that blows up forward and backward in time, $M(u)<M(Q)$ and satisfies \textup{(i)-(iii)}. We call such solutions as admissible minimal mass blow-up solutions.
\end{theorem}

We now want to show that the solution $ u \equiv 0 $, since by mass conservation, this fact would contradict that $ M(u) > 0 $. Using the dichotomy introduced by Dodson in the works \cite{Dodson-JAMS,Dodson-AJM,Dodson-Duke}, there are two scenarios for blow-up which we need to preclude. The first is the \emph{rapid frequency cascade scenario}
\[
\int_0^\infty N(t)^3 \, dt < \infty,
\]
and the second is the \emph{quasi-soliton scenario}
\[
\int_0^\infty N(t)^3 \, dt = \infty. 
\]

\subsubsection{Long-time Strichartz estimate} \label{sec:intro_out_LTSE}

The next step is to prove a long-time Strichartz estimate (see Theorem \ref{longtime} below) satisfied by admissible blow-up solutions to the Hartree equation \eqref{NLH}, which we use to preclude both the rapid frequency cascade and quasi-soliton scenarios. Roughly speaking, such an estimate does not hold for arbitrary solutions of equation \eqref{NLH} but is valid only for \emph{admissible minimal mass blow-up solutions}; hence, it is not an a prior estimate. Estimates of this type have been shown to have significant potential in other critical equations (for a detailed discussion, see \cite{Dodson} and the references therein). Specifically for the equation \eqref{NLH} under our study, we need to obtain appropriate control over the high frequencies of $u$. At a high level, we structure our estimate using the norms introduced by Dodson in \cite{Dodson-JAMS}. The proof relies on time-dependent Littlewood-Paley theory adapted to the frequency center $ \xi(t) $ and scale $ N(t) $, as well as specialized bilinear Strichartz estimates. Compared to \cite{Dodson-JAMS}, the main difficulties lie not only in the failure of the endpoint Hardy-Littlewood-Sobolev inequality 
\[
\big\||x|^{-2}\ast u\big\|_{L^{\frac{d}{2}}}\lesssim \|u\|_{L^1}
\]
but also in dealing with bilinear estimates of the following non-local form
\[
(|x|^{-2}\ast|P_{\leq N}|^2)P_{\geq M},\qtq{} N\ll M.
\]
 Similar issues appeared in the work of the third author and were resolved by using the interaction Morawetz estimate, inspired by \cite{Dodson-Duke}. In our case, given the smoothness of the nonlinear term, we introduce a stronger $ L_t^{2}L_x^{\frac{2d}{d-2},2} $ norm, employ an induction on frequency approach, and utilize the \emph{double frequency decomposition method} introduced by  Chae, Cho and  Lee in \cite{CCS} (see also the third author’s work \cite{Rosenzweig}) to obtain a stronger estimate of the solution in terms of the $L_t^{2}L_x^{\frac{2d}{d-2},2} $ norm. Compared to \cite{Rosenzweig}, the proof is significantly simplified. However, this method still cannot simplify the argument in \cite{Rosenzweig}, as the endpoint Strichartz estimate fails in $2d$. In fact, this is precisely the reason why Dodson \cite{Dodson-Duke} employed the interaction Morawetz method of Planchon and Vega \cite{Planchon} to prove the long-time Strichartz estimate. Finally, as a direct consequence of the long‑time Strichartz estimate, we prove a crucial Morawetz commutator estimate, which plays a key role in the subsequent rigidity part of our argument. In contrast to the proof in \cite{Dodson-JAMS} which relies on fractional derivative rules, we make heavy use of the Hardy–Littlewood–Sobolev inequality and frequency decomposition techniques, taking advantage of the good smoothness of the nonlinearity.

\subsubsection{Rigidity.}
Recall that we have two scenarios for blowup to preclude: the rapid frequency cascade scenario  and the quasi-soliton scenario. To preclude the rapid frequency cascade scenario, we use the idea of \cite{Dodson-JAMS} to combine our long-time Strichartz estimate with the additional regularity argument from the previous mass-critical NLS works of Tao, Vi\c{s}an, and Zhang \cite{TVZ-Duke} and Killip, Tao, and Vi\c{s}an \cite{KTV}.
\begin{lemma}[$H_x^3$ regularity]
    If $u$ is an admissible blow-up solution such that $\int_0^\infty N(t)^3 \, dt = K < \infty$, then $u \in C_t^0 H_x^3([0, \infty) \times \mathbb{R}^d)$ and satisfies the estimate
\[
\sup_{0 \leq t < \infty} \|u(t)\|_{H_x^3(\mathbb{R}^d)} \lesssim K^3.
\]
\end{lemma}

This additional regularity implies that the energy of a Galilean transformation $v$ of $u$, which is again a solution to equation \eqref{NLH}, tends to zero as time tends to $\infty$. Conservation of energy then implies that $v$ is identically zero. Since the mass functional is Galilean invariant, it then follows that $u$ is identically zero, which is a contradiction. This precludes the rapid frequency cascade scenario.

To preclude the quasi-soliton scenario, inspired by the approach of \cite{Dodson-Adv} for the focusing mass-critical NLS, we prove a frequency-localized interaction Morawetz ``type'' estimate for admissible blow-up solutions to equation \eqref{NLH} under the specific assumption that $\int_0^\infty N(t)^3 \, dt = \infty$. On an interval $[0, T]$, we show that our Morawetz functional $M(t)$ is bounded from below $\int_0^T N(t)^3 \, dt = K$ and bounded from above by a quantity which is $o(K)$. Since $K$ may be taken arbitrarily large (by taking $T$ arbitrarily large) in the quasi-soliton scenario, we obtain a contradiction. The analysis is much more delicate than that for the mass-critical NLS, as to extract the terms to which we then combine and apply the sharp Gagliardo-Nirenberg inequality for the Hartree nonlinearity. In particular, for the Hartree nonlinearity $F(u) = (|x|^{-2} * |u|^2)u$, the spatial cut-off does \emph{not} commute with the nonlinearity: $\chi F(u) \neq F(\chi u)$ for a characteristic function $\chi$. This necessitates a careful decomposition of the nonlinear term into high‑low frequency interactions and leads to additional commutator error terms. Moreover, in the Morawetz identity, the potential term $\int \nabla a \cdot \{F(v),v\}_P dx$ does \emph{not} admit a simple total differential expression. Consequently, one cannot directly apply integration by parts to extract a sign‑definite quantity. To overcome this, we perform a Taylor expansion of the weight function $a_R(x-y)$. However, this also introduces a large number of error terms. To solve this problem, we introduce a \emph{family of auxiliary weight functions} $\zeta_R, \phi_R, \psi_R$ depending on a small parameter $\eta_2$ and construct new Morawetz multipliers based on them. Through careful analysis (see Section \ref{quasisec} for more details), we then show that these error terms are negligible after a balancing of the floating parameters. We emphasize that this estimate is not an a priori estimate for sufficiently regular solutions to equation \eqref{NLH}. Frequency truncation of the solution introduces error terms which we can control with our long-time Strichartz estimate using an argument of \cite{Dodson-JAMS}, similar in spirit to the ``almost Morawetz'' estimates often used in conjunction with the I-method \cite{Iteam} and induction on energy \cite{Iteam2}. Moreover, we also find that the Lorentz norm plays a simplifying role in obtaining some certain estimates.

Finally, we point out that as a byproduct of the above argument, we can also obtain additional information about the admissible minimal mass blow-up solution  with mass $M(u) = M(Q)$ in the focusing case, as observed in \cite{Dodson-SIAM,Dodson-NA,F}. This is a crucial step in obtaining the characterization of the minimal mass blow-up solution at the $L^2$ level in an ongoing project.

The structure of this paper is as follows: In Section \ref{preli}, we present the necessary notations and lemmas required for the article. In Section \ref{sec3}, we further reduce the proof of Theorem \ref{main} to the exclusion of admissible minimal mass blow-up solutions. In Section \ref{sec4}, using the tools developed in previous sections, we establish the long-time Strichartz estimates. Subsequently, in Sections \ref{cascadesec} and \ref{quasisec}, we employ these estimates to rule out the rapid cascade scenario and the quasi-soliton scenario, respectively. Finally, in the Appendix \ref{app}, we further discuss some properties of admissible minimal mass blow-up solutions with mass $M(u)=M(Q)$ in the focusing case.

\subsection*{Acknowledgements} 
  C. Miao was supported by the National Key RR\&D program
of China: No.2021YFA1002500 and NSFC Grants 12026407 and 12531005. M. Rosenzweig was supported by NSF grants DMS-2441170, DMS-2345533, DMS-2342349.  J. Zheng was supported by  National key R\&D program of China: 2021YFA1002500 and  NSFC Grant 12271051.

\section{Preliminaries}\label{preli}
In this section, we  collect some  analysis tools and prove some technical lemmas needed throughout the paper. 

\subsection{Notation} We begin by setting up some notation. For nonnegative quantities $X$ and $Y$, we write $X\lesssim Y$ to denote the estimate $X\leq C Y$ for some $C>0$. If $X\lesssim Y\lesssim X$, we will write $X\sim Y$. Dependence on various parameters will be indicated by subscripts, e.g. $X\lesssim_u Y$ indicates $X\leq CY$ for some $C=C(u)$. 

We will use the expression $O(X)$ to denote a finite linear combination of terms that resemble $X$ up to Littlewood--Paley projections, complex conjugation, and/or maximal functions. We will use the expression $X\pm$ to denote $X\pm\eps$ for some small $\eps>0$. 

For a spacetime slab $I\times\R^d$, we write $L_t^q L_x^r(I\times\R^d)$ for the Banach space of functions $u:I\times\R^d\to\C$ equipped with the norm
\[
\|u\|_{L_t^{q}L_x^r(I\times\R^d)}:=\bigg(\int_I \|u(t)\|_{L_x^r(\R^d)}^q\,dt\bigg)^{1/q},
\]
with the usual adjustments if $q$ or $r$ is infinity. When $q=r$, we abbreviate $L_t^qL_x^q=L_{t,x}^q$. We also abbreviate $\|f\|_{L_x^r(\R^d)}$ to $\|f\|_{L_x^r}.$ For $1\leq r\leq\infty$, we use $r'$ to denote the dual exponent to $r$, i.e. the solution to $\frac{1}{r}+\frac{1}{r'}=1.$

The Fourier transform on $\mathbb{R}^d$ is given by
\begin{equation*}
\aligned \widehat{f}(\xi)= \mathcal{F} f(\xi)= \tfrac{1}{(2\pi)^{\frac{d}{2}}}\int_{\mathbb{R}^d}e^{- ix\cdot \xi}f(x)\,dx.
\endaligned
\end{equation*}
We then define the fractional differentiation operator $|\nabla|^s=\mathcal{F}^{-1}|\xi|^s\mathcal{F}$, with the corresponding homogeneous Sobolev norm 
\[
\|f\|_{\dot{H}_x^s(\R^d)}:=\||\nabla|^s f\|_{L_x^2(\R^d)}.
\]

Let $\varphi\in C_0^\infty(\R^d)$ be a real-valued, non-negative, even, and radially-decreasing function such that
\begin{align*}
\varphi(\xi) =
\begin{cases}
1, \ |\xi| \leq 1, \\
0, \  |\xi| \ge 2,
\end{cases}
\end{align*}
and define $\chi(\xi)=\varphi(\xi)-\varphi(2\xi)$. For a  number $N\in 2^{\mathbb{Z}}$, we define the Littlewood-Paley projectors $P_{\leq N}, P_{N}, P_{\geq N}$ via
\begin{align*}
\widehat{P_Nf}(\xi):=\chi\big(\tfrac{\xi}{N} \big)\widehat{f}(\xi),  \quad \widehat{P_{\leq N}f}(\xi):=\varphi\big(\tfrac{\xi}{N} \big)\widehat{f}(\xi),\quad \widehat{P_{\geq N}f}(\xi):=\Big(1-\varphi\big(\tfrac{\xi}{N} \big)\Big)\widehat{f}(\xi).
\end{align*}

We denote the free Schr\"odinger propagator by $e^{it\Delta}$.  It is most naturally defined as the Fourier multiplier operator with symbol $e^{-it|\xi|^2}$, but also has a physical-space representation given by convolution with a complex Gaussian. This operator satisfies well-known estimates known as Strichartz estimates, which we state as follows.

First, we call a pair of exponents $(q,r)$ \emph{Schr\"odinger admissible}  if $2\leq q,r\leq\infty$ and $\tfrac{2}{q}+\tfrac{d}{r}=\frac{d}{2}.$ For a spacetime slab $I\times\R^d$, we define the Strichartz norm 
\[
\|u\|_{S^0(I)}:=\sup\big\{\|u\|_{L_t^{q}L_x^r(I\times\R^d)}:(q,r)\text{ is Schr\"odinger admissible}\big\}.
\]
We denote $S^0(I)$ to be the closure of all test functions under this norm and write $N^0(I)$ for the dual of
$S^0(I)$.

\subsection{Ground state and profile decomposition}
We next give the following basic property of ground state:
\begin{proposition}[Existence of ground state,\cite{MiaoXuZhaoblowup}]
     There exists a positive, radial function $Q \in H^1(\R^d)$  which solves the ground state equation \eqref{grounstate}. Moreover, $Q$ has the variational characterization
\[
J(Q) = \min_{u \in H^1(\R^d)} J(u), \quad J(u) := \frac{\|u\|_{L^2(\R^d)}^2 \|\nabla u\|_{L^2(\R^d)}^2}{\int_{\R^d} |u(x)|^2 (V * |u|^2)(x) dx}.
\]
Moreover, J(Q) = $\frac{\|Q\|_{L^2(\R^d)}^2}{2}$, so that
\[
\int_{\R^d} |u(x)|^2 (V * |u|^2)(x) dx \leq \frac{2 \|u\|_{L^2(\R^d)}^2 \|\nabla u\|_{L^2(\R^d)}^2}{\|Q\|_{L^2(\R^d)}^2}, \quad \forall u \in H^1(\R^d).
\]
\end{proposition}
Next, we record $L^2$ profile decomposition:
\begin{proposition}[Linear profile decomposition in $L_x^2 ( \mathbb{R}^d)$,\cite{MV}] \label{pro3.9v23} Let $d\geq3$. Let $\{u_n\}_{n\ge 1}$ 
be a bounded sequence in $L_x^2 (\mathbb{R}^d)$. Then after passing to a subsequence if necessary,
there exist $J^* \in \{0,1, \cdots\} \cup \{\infty\}$,
functions $\phi^{j}$ in $L_x^2(\mathbb{R}^d)$ and mutually orthogonal frames $(\lambda_n^j, t_n^j, x_n^j,\xi_n^j )_{n\ge 1} \subseteq (0,\infty) \times \mathbb{R} \times \mathbb{R}^d \times \mathbb{R}^d$,
which means
\begin{align*}
&\frac{\lambda_n^j}{\lambda_n^k} + \frac{\lambda_n^k}{\lambda_n^j}
+ \lambda_n^j \lambda_n^{k} |\xi_n^j - \xi_n^{k}|^2 \\
&\quad+ \frac{\left|x_n^j-x_n^k-2t_n^j(\lambda_n^j)^2(\xi_n^j-\xi_n^k)\right|^2}{\lambda_n^j\lambda_n^k}
+ \frac{|(\lambda_n^j)^2 t_n^j - (\lambda_n^k)^2 t_n^k |}{\lambda_n^j \lambda_n^k}
\to \infty, \text{ as } n \to \infty, \text{ for } j\ne k,
\end{align*}
and for every $J \le J^*$, a sequence $r_n^J \in L_x^2  (\mathbb{R}^d)$, such that
\begin{align*}
u_n(x)  = \sum\limits_{j=1}^J \frac1{(\lambda_n^j)^{\frac{d}{2}}} e^{ix\cdot\xi_n^j} \left(e^{it_n^j\Delta_{\R^d}} \phi^j\right)\left(\frac{x-x_n^j}{\lambda_n^j} \right)  +r_n^J(x),
\end{align*}
furthermore, 
\begin{align*}
& \lim\limits_{n\to \infty} \left(\|u_n\|_{L_x^2 }^2 - \sum\limits_{j=1}^J \left\|   \frac1{(\lambda_n^j)^{\frac{d}{2}}} e^{ix\cdot\xi_n^j} \left(e^{it_n^j\Delta_{\R^d  }} \phi^j\right)\left(\frac{x-x_n^j}{\lambda_n^j}\right)
\right\|_{ L_x^2}^2 - \| r_n^J\|^2_{L_x^2}  \right)  = 0, \notag \\
& (\lambda_n^j)^{\frac{d}{2}}   e^{-it_n^j \Delta_{\R^d  }}\left(e^{-i(\lambda_n^j x + x_n^j)\cdot \xi_n^j} r_n^J\left(\lambda_n^j x+ x_n^j \right)\right)  \rightharpoonup 0 \text{ in } L_x^2 ,  \text{ as  } n\to \infty,\text{ for each }  j\le J,\notag\\
&  \limsup\limits_{n\to \infty} \|e^{it \Delta_{\R^d    }} r_n^{J} \|_{L_{t,x}^{\frac{2(d+2)}{d}}(\mathbb{R}\times \mathbb{R}^d)}   \to 0, \text{ as } J\to J^*.
\end{align*}
\end{proposition}

\subsection{Atom space $U_{\Delta}^2$}
\begin{definition}[$U_\Delta^2(I\times\R^d)$ space, \cite{Dodson-AJM,Dodson-Duke,Dodson,HHK}]  We define $U_\Delta^2(\R\times\R^d)$ as an atomic space with atoms  $v^\lambda(t,x)$ given by 
\begin{align*}
v^\lambda(t,x)=\sum_{k=0}^{M}\chi_{[t_k,t_{k+1}]}(t)e^{it\Delta} v_k^\lambda(x),
\quad\sum_{k=0}^M\left\| v_k^\lambda(x) \right\|_{L_x^2(\R^d)}^2=1.
\end{align*}
In the summation above, $M$ can be finite or infinite. If $M$ is finite, we further assume $t_0=-\infty$ and $t_{M+1}=+\infty$. The norm of $U_\Delta^2 (\R\times\R^d)$ is defined as 
\begin{align*}
\| v \|_{U_\Delta^2(\R\times\R^d)}
=\inf\left\{\sum_{\lambda\in\mathbb{Z}} |c_\lambda|:  v =\sum_{\lambda\in\mathbb{Z}}c_\lambda  v^\lambda, \hspace{1ex} v^\lambda\mbox{ is an atom}\right\},
\end{align*}
where the infimum is taken over all atom decompositions.  For any  interval $I\subseteq\R$, we can define the local version as
\begin{align*}
\| v\|_{U_\Delta^2 \left(I\times\R^d\right)}=\| v \chi_I(t)\|_{U_\Delta^2 \left(\R\times\R^d\right)}.
\end{align*}
\end{definition}
\begin{proposition}[Strichartz estimates, \cite{Dodson-AJM,Dodson-Duke,Dodson,HHK}]
Let $I$ be an interval, then
\begin{equation}
\|u\|_{U_{\Delta}^2(I\times\R^d)}\lesssim \min_{t\in I}\|u(t)\|_{L^2}+\|(i\partial_t+\Delta)u\|_{L_t^1L_x^2(I\times\R^d)}.
\end{equation}
\end{proposition}
By local theory, the above lemma implies that
\begin{equation}\label{localboundU2}
\|u\|_{L_{t,x}^{\frac{2(d+2)}{d}}(I\times\R^d)}=1\Rightarrow\|u\|_{U_{\Delta}^2(I\times\R^d)}\lesssim_{M(u)}1.  
\end{equation}

\subsection{Bilinear estimates}
In this subsection, we prove some bilinear estimates. We need the following bilinear Strichartz estimates from \cite{KV}.

\begin{lemma}[Bilinear Strichartz estimate, \cite{KV}]\label{prop:bs_B} 
Let $u_{0},v_{0}\in L^{2}(\R^{d})$. If $u_{0}$ and $v_{0}$ have Fourier supports in the annuli $|\xi-\xi_0|\sim N$ and $|\xi-\xi_0|\sim M$, with $M\ll N$, respectively, then 
\begin{equation}
\|(e^{it\Delta}u_{0})(e^{it\Delta}v_{0})\|_{L_{t,x}^2(\R\times\R^{d})} \lesssim \left(\frac{M^{d-1}}{N}\right)^{1/2}\|u_{0}\|_{L^{2}(\R^{d})}\|v_{0}\|_{L^{2}(\R^{d})}.
\end{equation}
\end{lemma}

By Galilean invariance of the free Schr\"{o}dinger equation, we can also prove the following cube-localized version of the bilinear estimate.
\begin{lemma}\label{prop:cube_BS} 
Let $0<c_{1}<c_{2}$, and let $\delta \in (0,1)$. Let $\xi_{0}, \xi_1\in\R^{d}$, and let $Q_{\xi_{0}}$ be a cube centered at $\xi_{0}$ of side length $l(Q_{\xi_{0}})=M$. Let $u_{0},v_{0}\in L^{2}(\R^{d})$. 
If $u_{0}$ has Fourier support in the annulus $A(\xi_1,c_{1}N,c_{2}N)$ and $v_{0}$ has Fourier support in the cube $Q_{\xi_{0}} \subset B(\xi_1,(1-\delta)c_{1}N)$, where $M\ll N$, then 
\begin{equation}
\|(e^{it\Delta}u_{0})(e^{it\Delta}v_{0})\|_{L_{t,x}^{2}(\R\times\R^{d})}\lesssim_{c_{1},c_{2},\delta} \left(\frac{M^{d-1}}{N}\right)^{1/2}\|u_{0}\|_{L^{2}(\R^{d})}\|v_{0}\|_{L^{2}(\R^{d})}.
\end{equation}
\end{lemma}
\begin{proof}
 Observe that by translation invariance of the Lebesgue measure and Galilean invariance of the free Schr\"{o}dinger equation, we can write
\begin{align}
\|(e^{it\Delta}u_{0})(e^{it\Delta}v_{0})\|_{L_{t,x}^{2}(\R\times\R^{d})} &= \|\left({T_{g_{0,-\xi_{0},0,1}}(e^{it\Delta}u_{0})}\right)\left({T_{g_{0,-\xi_{0},0,1}}(e^{it\Delta}v_{0})}\right)\|_{L_{t,x}^{2}(\R\times\R^{d})} \nonumber\\
&= \|(e^{it\Delta}u_{0,\xi_{0}})(e^{it\Delta}v_{0,\xi_{0}})\|_{L_{t,x}^{2}(\R\times\R^{d})},
\end{align}
where $u_{0,\xi_{0}} := g_{0,-\xi_{0},0,1}u_{0}$ and $v_{0,\xi_{0}} := g_{0,-\xi_{0},0,1}v_{0}$. Observe that $v_{0,\xi_{0}}$ has Fourier support in the cube $Q_{0} := [-\frac{M}{2},\frac{M}{2}]^{d}$. Now if $\xi\in \text{supp}(\hat{u}_{0,\xi_{0}})$, then $\xi+\xi_{0}\in A(\xi_1,c_{1}N,c_{2}N)$. Since $Q_{\xi_{0}} \subset B(\xi_1,(1-\delta)c_{1}N)$, in particular, $|\xi_{0}-\xi_1| \leq (1-\delta)c_{1}N$. Hence by the (reverse) triangle inequality,
\begin{equation}
\delta c_{1}N = \left({c_{1}-(1-\delta)c_{1}}\right)N \leq |\xi| \leq \left({c_{2}+(1-\delta)c_{1}}\right)N,
\end{equation}
which implies that $u_{0,\xi_{0}}$ has Fourier support in the annulus $A(0,\delta c_{1}N, (c_{2}+(1-\delta)c_{1})N)$. Applying Proposition \ref{prop:bs_B} with $u_{0,\xi_{0}}$ and $v_{0,\xi_{0}}$, we conclude that
\begin{align*}
\|(e^{it\Delta}u_{0,\xi_{0}})(e^{it\Delta}v_{0,\xi_{0}})\|_{L_{t,x}^{2}(\R\times\R^{d})} &\lesssim_{\delta,c_{1},c_{2}} \left({\frac{M^{d-1}}{N}}\right)^{1/2} \|u_{0,\xi_{0}}\|_{L^{2}(\R^{d})} \|v_{0,\xi_{0}}\|_{L^{2}(\R^{d})}\notag\\
&= \left({\frac{M^{d-1}}{N}}\right)^{1/2} \|u_{0}\|_{L^{2}(\R^{d})} \|v_{0}\|_{L^{2}(\R^{d})},
\end{align*}
as desired.
\end{proof}
As directly application of above two lemmas, we have
\begin{corollary}\label{bilinear}
Let $M\ll N$. Let $J$ be an interval. Let $u,v\in U_{\Delta}^2(J\times\R^d)$. If $u$ and $v$ have Fourier supports in the annuli $|\xi-\xi_1|\sim N$ and $|\xi-\xi_1|\sim M$ respectively, or if $u$ has Fourier support in the annulus $A(\xi_1,c_{1}N,c_{2}N)$ and $v$ has Fourier support in the cube $Q_{\xi_{0}} \subset B(\xi_1,(1-\delta)c_{1}N)$ for some $0<c_{1}<c_{2}$ and  $\delta \in (0,1)$, then
\begin{align}
\|uv\|_{L_{t,x}^2(J\times\R^d)}\lesssim \left({\frac{M^{d-1}}{N}}\right)^{1/2}  \|u\|_{U_{\Delta}^2(J\times\R^d)}\|v\|_{U_{\Delta}^2(J\times\R^d)}.
\end{align}
\end{corollary}
    
Next, we establish a square-type function estimate, which will play an important role in the proof of long-time Strichartz esitmate.
\begin{lemma}\label{Xdeltanorm}
Let $J$ be an interval. Let $\{Q\}$ be a collection of cubes in $\R^{d}$ such that $\|\sum_{Q}1_{Q}\|_{L^{\infty}}<\infty$ (i.e. boundedly overlapping). Let $P_{Q}u:=\mathcal{F}^{-1}(\chi_{Q}\mathcal{F}u)$, then we have
\begin{equation}
\left(\sum_{Q}\|P_{Q}u\|_{U_{\Delta}^{2}(J\times\R^{d})}^{2}\right)^{1/2} \lesssim \|u\|_{U_{\Delta}^{2}(J\times\R^{d})},
\end{equation}
where the implicit constant only depends on $\|\sum_{Q}1_{Q}\|_{L^{\infty}}$.
\end{lemma}
\begin{proof}
Suppose $u=\sum_{j}c_{j}u_{j}$ is a $U_{\Delta}^{2}$ atomic decomposition of $u$, where
\begin{equation}
u_{j} = \sum_{k}1_{[t_{j,k},t_{j,k+1})} e^{it\Delta}u_{j,k}.
\end{equation}
For each cube $Q$, define the function
\begin{equation}
u_{j, Q} := \sum_{k}1_{[t_{j,k},t_{j,k+1})}e^{it\Delta}\frac{P_{Q}u_{j,k}}{(\sum_{k}\|P_{Q}u_{j,k}\|_{L^{2}}^{2})^{1/2}},
\end{equation}
which is a $U_{\Delta}^{2}$ atom. Observe that 
\begin{equation}
u_{Q} := P_{Q}u = \sum_{j}c_{j,Q}u_{j, Q}, \qquad c_{j, Q} := c_{j}\left({\sum_{k}\|P_{Q}u_{j, k}\|_{L^{2}}^{2}}\right)^{1/2}
\end{equation}
is an atomic decomposition for $u_{Q}$. Therefore by Minkowski's inequality,
\begin{align}
\left(\sum_{Q}\|u_{Q}\|_{U_{\Delta}^{2}(J\times\R^{2})}^{2}\right)^{1/2} &\leq \left(\sum_{Q} \left(\sum_{j}|c_{j}|\left(\sum_{k}\|P_{Q}u_{j, k}\|_{L^{2}}^{2}\right)^{1/2}\right)^{2} \right)^{1/2}\nonumber\\
&\leq \sum_{j}|c_{j}| \left(\sum_{Q}\sum_{k}\|P_{Q}u_{j,k}\|_{L^{2}}^{2}\right)^{1/2} \nonumber\\
&= \sum_{j}|c_{j}| \left(\sum_{k}\sum_{Q}\|P_{Q}u_{j, k}\|_{L^{2}}^{2}\right)^{1/2} \nonumber\\
&\lesssim \sum_{j}|c_{j}|\left(\sum_{k}\|u_{j, k}\|_{L^{2}}^{2}\right)^{1/2}\nonumber\\
&= \sum_{j}|c_{j}|,
\end{align}
where we use Plancherel's theorem together with the bounded overlap of the cubes $Q$ to obtain the penultimate inequality. Taking the infimmum of the RHS of the ultimate equality over all $U_{\Delta}^{2}$ atomic decompositions of $u$ completes the proof of the lemma.
\end{proof}

\subsection{Lorentz spaces and  nonlinear estimates}\label{S:2.2}
Let $f$ be a measurable function on $\mathbb{R}^d$. The distribution function of $f$ is defined by
\begin{equation}
d_f(\lambda)=: |\{x\in \mathbb{R}^d : |f(x)|>\lambda\}|, \quad \lambda>0,\notag
\end{equation}
where $|A|$ is the Lebesgue measure of a set $A$ in $\mathbb{R}^d$. The decreasing rearrangement of $f$ is defined by
\begin{equation}
f^*(s)=: \inf \left\{ \lambda>0 : d_f(\lambda)\leq s\right\}, \quad s>0.\notag
\end{equation}

\begin{definition}[Lorentz spaces] 
Let $0<r<\infty$ and $0<\rho\leq \infty$. The Lorentz space $L^{r,\rho}(\mathbb{R}^d)$ is defined by
\begin{equation}
L^{r,\rho}(\mathbb{R}^d)=: \left\{ f \text{ is measurable on } \mathbb{R}^d : \|f\|_{L^{r,\rho}}<\infty\right\}, \notag
\end{equation}
where
\[\|f\|_{L^{r,\rho}}=: \left\{
\begin{array}{cl}
( \frac{\rho}{r} \int_0^\infty (s^{1/r} f^*(s))^\rho \frac{1}{s}ds)^{1/\rho} &\text{ if } \rho <\infty, \\
\sup_{s>0} s^{1/r} f^*(s) &\text{ if } \rho=\infty.
\end{array}\right.\]
\end{definition}

We collect the following basic properties of $L^{r,\rho}(\mathbb{R}^d)$ in the following lemmas.
\begin{lemma}[Properties of Lorentz spaces, \cite{ONeil}] \label{basicp} We have

\begin{itemize}
\item[$(i)$] For $1<r<\infty$, $L^{r,r}(\mathbb{R}^d) \equiv L^r(\mathbb{R}^d)$ and by convention, $L^{\infty,\infty}(\mathbb{R}^d)= L^\infty(\mathbb{R}^d)$.

\item[$(ii)$] For $1<r<\infty$ and $0<\rho_1<\rho_2\leq \infty$, $L^{r,\rho_1}(\mathbb{R}^d)\subset L^{r,\rho_2}(\mathbb{R}^d)$. 
		
\item[$(iii)$] For $1<r<\infty$, $0<\rho \leq \infty$, and $\theta>0$, $\||f|^\theta\|_{L^{r,\rho}} = \|f\|^\theta_{L^{\theta r, \theta \rho}}$.

\item[$(iv)$]  $|x|^{-2} \in L^{\frac{d}{2},\infty}(\mathbb{R}^d)$.
\end{itemize}
\end{lemma}

\begin{lemma}[H\"older's inequality, \cite{ONeil}]  There holds
\begin{itemize}
\item[$(i)$]  Let $1<r, r_1, r_2<\infty$ and $1\leq \rho, \rho_1, \rho_2 \leq \infty$ be such that
\[\frac{1}{r}=\frac{1}{r_1}+\frac{1}{r_2}, \quad \frac{1}{\rho} \leq \frac{1}{\rho_1}+\frac{1}{\rho_2}.
		\]
Then for any $f \in L^{r_1, \rho_1}(\mathbb{R}^d)$ and $g\in L^{r_2, \rho_2}(\mathbb{R}^d)$
\[\|fg\|_{L^{r,\rho}} \lesssim \|f\|_{L^{r_1, \rho_1}} \|g\|_{L^{r_2,\rho_2}}.\]
\item[$(ii)$] Let  $1<r_1,r_2<\infty $ and  $1\le \rho_1,\rho_2\le \infty $  be such that 
\begin{equation}
1=\frac{1}{r_1}+\frac{1}{r_2},\quad 1\le \frac{1}{\rho_1}+\frac{1}{\rho_2}.\notag
\end{equation}
Then  for any $f \in L^{r_1, \rho_1}(\mathbb{R}^d)$ and $g\in L^{r_2, \rho_2}(\mathbb{R}^d)$
\begin{equation}
\|fg\|_{L^1}\lesssim  \|f\|_{L^{r_1, \rho_1}} \|g\|_{L^{r_2,\rho_2}}.\notag
\end{equation}
\end{itemize}
\end{lemma}
\begin{lemma}[{Interpolation in Lorentz spaces}, \cite{Dao}]\label{CZ}
Let  $1<p_0<p<p_1<\infty $,  $r>0$ and $0<\theta<1$ satisfy  $\frac{1}{p}=\frac{1-\theta}{p_0}+\frac{\theta}{p_1}$. For every  $f\in L^{p_0,\infty }(\mathbb{R}^d)\cap L^{p_1,\infty }(\mathbb{R}^d)$, we have
	\begin{equation}
		\|f\|_{L^{p,r}(\mathbb{R} ^d)}\lesssim  \|f\|_{L^{p_0,\infty }(\mathbb{R} ^d)}^{1-\theta} \|f\|_{L^{p_1,\infty }(\mathbb{R} ^d)}^{\theta}.\notag
	\end{equation}
\end{lemma}
\begin{lemma}[Young's inequality for weak type spaces, \cite{Grafakos}]\label{lem:weakYoung}
Let $1 \leq p < \infty$ and $1 < q, r < \infty$ satisfy
\[\frac{1}{q} + 1 = \frac{1}{p} + \frac{1}{r}.\]
Then there exists a constant $C_{p,q,r} > 0$  such that for all $f \in L^p(\R^d)$  and $g \in L^{r,\infty}(\R^d)$  we have
\[\|f*g\|_{L^{q,\infty}(\R^d)} \leq C_{p,q,r}\|g\|_{L^{r,\infty}(\R^d)}\|f\|_{L^p(\R^d)}.
\]
\end{lemma}

\begin{lemma}[Bernstein's inequality in Lorentz space, \cite{LMZ}]\label{L:Bernstein}
Let  $N>0$,  $1<r_1<r_2<\infty $ and  $1\le \rho_1\le\rho_2\le \infty $. Then 
\begin{equation}
\|P_{N}f\|_{L^{r_2,\rho_2}(\mathbb{R} ^d)}\lesssim  N^{d(\frac{1}{r_1}-\frac{1}{r_2})} \|f\|_{L^{r_1,\rho_1}(\mathbb{R} ^d)}.\notag  
\end{equation}
\end{lemma}

\begin{lemma}[Strichartz estimates in Lorentz space, \cite{KT}]\label{P:SZ} We have
\begin{itemize}
\item[$(i)$] Let $(q,r)$ be an admissible pair, then for any $f\in L^2(\mathbb{R}^d)$
\begin{align}  
\|e^{it\Delta }f\|_{L^q_t L^{r,2}_x(\mathbb{R}\times \mathbb{R}^d)} \lesssim \|f\|_{L^2_x(\mathbb{R}^d)}.\notag
\end{align}
		
\item[$(ii)$] Let $(q_1, r_1), (q_2,r_2)$ be two admissible pairs,  and $I\subset \mathbb{R}$ be an interval containing $t_0$. Then for any $F\in L_t^{q_2'}L^{r_2',2}_x(I\times \mathbb{R}^d)$ 
\begin{align} 
\left\|\int_{t_0}^t e^{i(t-\tau)\Delta } F(\tau) d\tau\right\|_{L^{q_1}_tL^{r_1,2}_x(I\times \mathbb{R}^d)} \lesssim \|F\|_{L^{q_2'}_tL^{r_2',2}_x(I\times \mathbb{R}^d)}.\notag
\end{align}
\end{itemize}
\end{lemma}

With the above Lemma in hands, we next prove that 
\begin{lemma}\label{lorennorm}
Let $u$ be a solution to \eqref{NLH}. Let $J$ be an interval such that $\|u\|_{L_{t,x}^{\frac{2(d+2)}{d}}(J\times\R^d)}\leq M$. Then we have
\begin{equation}\label{lorennormes}
\|u\|_{L_t^{\infty}L_x^2\cap L_t^2L_x^{\frac{2d}{d-2},2}(J\times\R^d)}\lesssim_{M(u),M} 1.
\end{equation}
\end{lemma}

\begin{proof}
We observe that by Lemma \ref{P:SZ}, the conservation law of mass, H\"older's inequality, Hardy-Littlewood-Sobolev's inequality, Lemma \ref{basicp} and interpolation, we see that
\begin{align*}
\|u\|_{L_t^{\infty}L_x^2\cap L_t^2L_x^{\frac{2d}{d-2},2}(J\times\R^d)}&\lesssim M(u)^{1/2} +\|F(u)\|_{L_t^{2}L_{x}^{\frac{2d}{d+2},2}(J\times\R^d)}\\
&\lesssim M(u)^{1/2} +\||x|^{-2}\ast|u|^2\|_{L_t^{3}L_{x}^{\frac{3d}{4},\infty}(J\times\R^d)}\|u\|_{L_t^{6}L_x^{\frac{6d}{3d-2},2}(J\times\R^d)}\\
&\lesssim M(u)^{1/2}+\|u\|^2_{L_t^{6}L_x^{\frac{6d}{3d-2}}(J\times\R^d)}\|u\|_{L_t^{6}L_x^{\frac{6d}{3d-2},2}(J\times\R^d)}\\
&\lesssim M(u)^{1/2}+\|u\|^2_{L_t^{6}L_x^{\frac{6d}{3d-2}}(J\times\R^d)}\|u\|_{L_t^{\infty}L_x^2\cap L_t^2L_x^{\frac{2d}{d-2},2}(J\times\R^d)}\\
&\lesssim M(u)^{1/2}+M(u)^{\frac{2(d-1)}{3d}}\|u\|^{\frac{2d+4}{3d}}_{L_{t,x}^{\frac{2(d+2)}{d}}(J\times\R^d)}\|u\|_{L_t^{\infty}L_x^2\cap L_t^2L_x^{\frac{2d}{d-2},2}(J\times\R^d)}.
\end{align*}
Using the fact that $\|u\|_{L_{t,x}^{\frac{2(d+2)}{d}}}\leq M$, we can split the interval $J$ into finitely many subintervals on which the $\|u\|_{L_{t,x}^{\frac{2(d+2)}{d}}}$-norm of $u$ is small and run a standard bootstrap argument to obtain \eqref{lorennormes}.
\end{proof}

\section{Reduction to the almost-periodic solution}\label{sec3}

In this section, we argue by contradiction and suppose that Theorem~\ref{main} fails. By local well-posedness theory, we know that global existence and scattering for sufficiently small initial data, and then we can deduce the existence of a critical threshold size, below which the theorem holds but above which we can find solutions with arbitrarily large scattering size. Using a limiting argument, we can then deduce the existence of minimal counterexamples, that is, blow-up solutions that live exactly at the critical threshold. The key property of these minimal counterexamples is that of almost periodicity modulo the symmetries of the equation. 

\begin{definition}[Almost periodic modulo symmetries (APMS)]
We say that $u \in C_{t, \text { loc }}^0 L_x^2\left(I \times \R^d\right)$ is almost periodic modulo $G$ if there exists a spatial center function $x: I \rightarrow \R^d$, a frequency center function $\xi: I \rightarrow \R^d$, a frequency scale function $N: I \rightarrow[0, \infty)$, and a compactness modulus function $C:(0, \infty) \rightarrow[0, \infty)$, such that for every $\eta>0$, we have the spatial and frequency localization estimates
\begin{equation}\label{localized}
    \int_{|x-x(t)| \geq C(\eta) / N(t)}|u(t, x)|^2 d x+\int_{|\xi-\xi(t)|\geq C(\eta) N(t)}|\hat{u}(t, \xi)|^2 d \xi \leq \eta^2, \quad \forall t \in I.
\end{equation}

\end{definition}

We are now in a position to state precisely the main step in the proof of Theorem \ref{main}.

\begin{theorem}\label{reduction}
Suppose Theorem \ref{main} fails for $\mu=1$. Then  there exists an almost-periodic maximal-lifespan solution $u: I \times \R^d \rightarrow \mathbb{C}$ to the Hartree equation \eqref{NLH} that blows up forward in time, and satisfies the following conditions:
\begin{itemize}
\item[\textup{(i)}]  $[0, \infty) \subset I$;
    
\item[\textup{(ii)}]  $N(t) \leq 1$ for all $t \in[0, \infty)$ and $x(0) = \xi(0) = 0$;
    
\item[\textup{(iii)}] $N(t), \xi(t) \in C_{\mathrm{loc}}^1([0, \infty))$ and satisfy the pointwise derivative bounds
\begin{equation}\label{xxibound}
\left|N'(t)\right| + \left|\xi'(t)\right| \lesssim_u N(t)^3, \quad t \in [0, \infty).
\end{equation}
\end{itemize}
Similarly, suppose Theorem \ref{main} fails for $\mu=-1$, then there also exists a maximal-lifespan solution $u: I \times \R^d \rightarrow \mathbb{C}$ to the Hartree equation \eqref{NLH} that blows up forward  in time, $M(u)<M(Q)$ and satisfies \textup{(i)-(iii)}.
\end{theorem}

The reduction to almost-periodic solutions is now widely regarded as a standard technique in the study of dispersive equations at critical regularity. For the mass-critical case, we refer to Killip-Vi\c{s}an's lecture note \cite{KV} and \cite{MXZ-JMPA} for the  Hartree equation.

The parameters $N(t),x(t),\xi(t)$ of almost-periodic solutions can be shown to obey the following local constancy property (see \cite{KV} for details).

\begin{proposition}[Local constancy of parameters]
Let $u: I \times \R^d \rightarrow \mathbb{C}$ be a nonzero maximal-lifespan solution to \eqref{NLH} as in Theorem \ref{reduction}. Then there exists a $0<\delta=\delta(u) \leq 1$ such that for every $t_0 \in I$,
\[
\Big[t_0-\frac{\delta}{N(t_0)^2}, t_0+\frac{\delta}{N\left(t_0\right)^2}\Big] \subset I
\]
and
\begin{align*}
N(t) & \sim_u N\left(t_0\right) \\
\left|\xi(t)-\xi\left(t_0\right)\right| & \lesssim_u N\left(t_0\right) \\
\left|x(t)-x\left(t_0\right)-2\left(t-t_0\right) \xi\left(t_0\right)\right| & \lesssim_u N\left(t_0\right)^{-1}
\end{align*}
for all $\left|t-t_0\right| \leq \delta N\left(t_0\right)^{-2}$.
\end{proposition}
\begin{proof}
See \cite{Dodson,KV}.
\end{proof}

We also have the following result relating the frequency scale function of an almost-periodic solution to its Strichartz norms (see \cite{Dodson,KV}).

\begin{proposition}[Stricharz norms via $N(t)$]\label{zzzz}
Let $u: I \times \R^d \rightarrow \mathbb{C}$ be a nonzero maximal-lifespan solution to \eqref{NLH} as in Theorem \ref{reduction}. Then for any compact subinterval $J \subset I$,
\[
    \int_J N(t)^2 d t \lesssim_{u,q}\|u\|_{L_t^{q} L_x^{r}\left(J \times \R^d\right)}^q \lesssim_{u,q} 1+\int_J N(t)^2 d t
    \]
for any admissible pair $(q,r)$ with $q<\infty$.
\end{proposition}

Finally, we state the properties of the parameters of almost-periodic solution in the local constant time.

\begin{proposition}\label{zzzz2}
Let $u: I \times \R^d \rightarrow \mathbb{C}$ be a nonzero maximal-lifespan solution to \eqref{NLH} as in Theorem \ref{reduction}. For every compact subinterval $J \subset I$ with $\|u\|_{L_{t, x}^{\frac{2(d+2)}{d}}\left(J \times \R^d\right)} \leq 1$, we have that
\begin{align*}
N\left(t_1\right) & \sim_u N\left(t_2\right), \\
\left|\xi\left(t_1\right)-\xi\left(t_2\right)\right| & \lesssim_u N(J), \\
\left|x\left(t_1\right)-x\left(t_2\right)-2\left(t_1-t_2\right) \xi\left(t_2\right)\right| & \lesssim_u N(J)^{-1},
\end{align*}
for all $t_1,t_2\in J$, where $N(J):=\sup_{t\in J}N(t)$. Moreover,
\[
 |J|\lesssim_u N(J)^{-2},\qquad
 \int_JN(t)^3\,dt\lesssim_u N(J).
\]
If $\|u\|_{L_{t,x}^{\frac{2(d+2)}d}(J\times\R^d)}=1$, then
\[
 |J|\sim_u N(J)^{-2},\qquad
 \int_JN(t)^3\,dt\sim_u N(J).
\]
Consequently, if a compact interval $K\subset I$ is a union of consecutive
intervals $J_l$ with
$\|u\|_{L_{t,x}^{\frac{2(d+2)}d}(J_l\times\R^d)}=1$, then
\begin{equation}\label{NJde}
 \int_KN(t)^3\,dt\sim_u\sum_lN(J_l).
\end{equation}
For a partition with an incomplete final interval, \eqref{NJde} applies to
the union of the complete intervals, while the remaining interval satisfies
the upper bounds above.
\end{proposition}
\begin{proof}
See \cite{Dodson,KV}.
\end{proof}

\section{Long time Strichartz estimate}\label{sec4}
In this section, we aim to establish the long-time Stricharzt estimate. 

Let $u$ be an almost-periodic solution as in Theorem \ref{reduction}. Then by \eqref{xxibound}, there exists some $\eta_1$ such that  
\begin{equation}\label{condi1}
|\xi'(t)| + |N'(t)| \leq \frac{N(t)^3}{\eta_1}. 
\end{equation}

Let $k_0$ be a nonnegative integer and let $[a,b]$ be a compact interval satisfying 
\begin{equation}\label{condi2}
\int_a^b \int_{\R^d} |u(t,x)|^{2(d+2)/d}  dx  dt \leq 2^{k_0}.
\end{equation}

By rescaling, we may further assume that  
\begin{equation}\label{condi3}
\int_a^b N(t)^3  dt = \eta_{\ast} 2^{k_0}
\end{equation}
for some $\eta_{\ast}\ll \eta_1$.
\begin{definition}[Galilean Littlewood–Paley projection]
For $j\in\mathbb Z$, let $P_{2^j}$ and $P_{\leq(\geq) 2^j}$ be the Littlewood–Paley operators, then for any $\xi_0\in\R^d$, we define  
\begin{align*}
P_{\xi_0,2^j}f =& e^{ix\cdot \xi_0} P_{2^j} \left( e^{-ix\cdot \xi_0} f \right)\\
P_{\xi_0,\leq(\geq)2^j}f =& e^{ix\cdot \xi_0} P_{\leq(\geq) 2^j} \left( e^{-ix\cdot \xi_0} f \right).
\end{align*}
We have the indentiy
\begin{equation}\label{equ:GaLP}
P_{\xi_0,2^j}f=P_{|\xi-\xi_0|\simeq 2^j}f.
\end{equation}
\end{definition}

We use the homogeneous decomposition, with $P_{\xi_0,2^j}$ defined as above for every $j\in\mathbb Z$. The smooth high- and low-frequency cutoffs are defined at every positive frequency. For  $1 \leq p \leq \infty$  and $1 \leq q \leq \infty$, define the norm
\[
\|P_{\xi(t),2^j}f\|_{L_t^p L_x^q (I \times \mathbb{R}^d)} = \left\| \|P_{\xi(t),2^j}f(t,x)\|_{L_x^q (\mathbb{R}^d)} \right\|_{L_t^p (I)}. 
\]

\begin{definition}[Long-time Strichartz seminorm] Let $I=[a,b]$ be an interval, then we define the seminorm $X([a,b]\times\R^d)$ as
\[
\|u\|_{X([a,b]\times\R^d)}^2 = \sup_{0 \leq j \leq k_0} 2^{j-k_0} \|P_{\xi(t),\geq 2^j}u\|_{L_t^2 L_x^{\frac{2d}{d-2},2} ([a,b] \times \mathbb{R}^d)}^2,
\]
where $k_0$ is as in \eqref{condi2}.
\end{definition}

\begin{theorem}\label{longtime}
    Let $u$ be a rescaling of an almost-periodic solution from Theorem \ref{reduction}. If $u$ satisfies \eqref{condi1}-\eqref{condi3}, then for any $\eta_{\ast}\ll 1$, we have
\begin{equation}
    \|u\|_{X([a,b]\times\R^d)} \lesssim 1,
\end{equation}
with constant independent of $k_0$ and $\eta_{\ast}$.
\end{theorem}

To prove this theorem, we first divide the interval $[a,b]$ in two different ways.
\begin{definition}[Small intervals]
    Divide $[a,b]$ into at most $2^{k_0}$ consecutive, disjoint intervals $J_l$, such that
\[\int_{J_l} \|u(t)\|_{L_x^{2(d+2)/d}(\mathbb{R}^d)}^{2(d+2)/d} dt = 1.\]
For the last interval, the integral is allowed to be less than one.
\end{definition}

\begin{definition}
    Divide $[a,b]$ into at most $2^{k_0}$ consecutive, disjoint intervals $J_\alpha$ such that
\[\int_{J_\alpha} \left( N(t)^3 + \eta_{\ast} \|u(t)\|_{L_x^{2(d+2)/d}(\mathbb{R}^d)}^{2(d+2)/d} \right) dt = 2\eta_{\ast}. \]
The last interval is allowed to have a smaller integral. We set the remaining $J_\alpha$ to be empty.
For any integer $0 \leq j < k_0$, let
\begin{equation}\label{condi4}
 G_k^j = \bigcup_{\alpha=k2^j}^{(k+1)2^j-1} J_\alpha,    
\end{equation}
and for $ j \geq k_0 $, let
\[ G_k^j = [a, b]. \]
If $G_\alpha^i = [\tilde{a}, \tilde{b}]$ , let $\xi (G_\alpha^i) = \xi (\tilde{a})$ . Define $ \xi (J_l)$ in a similar manner.
\end{definition}
\begin{remark}
$(i)$ Let $N(J_l) = \sup_{t \in J_l} N(t)$. On each complete small interval, \eqref{NJde} gives
\[
\int_{J_l} N(t)^3 \sim \sup_{t \in J_l} N(t) = N(J_l) \sim \inf_{t \in J_l} N(t).
\]

$(ii)$ By \eqref{condi1} and \eqref{condi4}, for all $t \in G_\alpha^i$ , we have
\begin{equation}\label{condi5}
    |\xi (t) - \xi (G_\alpha^i)| \leq \int_{G_\alpha^i} \eta_1^{-1} N(t)^3 dt \leq 2\eta_{\ast} \eta_1^{-1} 2^i,
\end{equation}
so for all  $t \in G_\alpha^i$ ,
\[ \big\{\xi : 2^{i-1} \leq |\xi - \xi (t)| \leq 2^{i+1}\big\} \subset \big\{\xi : 2^{i-2} \leq |\xi - \xi (G_\alpha^i)| \leq 2^{i+2}\big\}.
\]
\end{remark}

Now we are in position to prove Theorem \ref{longtime}

\begin{proof}[Proof of Theorem \ref{longtime}]
For $0\leq j\leq\min\{4,k_0\}$, the local Strichartz estimate on the small intervals gives the claimed bound directly. Thus fix $5\leq j\leq k_0$ and write $[a,b]$ as the union of the nonempty groups among $\{G_k^j\}_{k=0}^{2^{k_0-j}-1}$;  for any $t, t_0 \in G_k^j$, we have by Duhamel's formula
\[ 
P_{\xi (t), \geq 2^j} u(t) = e^{i(t-t_0)\Delta} P_{\xi (t), \geq 2^j} u(t_0) - i\mu \int_{t_0}^t e^{i(t-\tau)\Delta} P_{\xi (t), \geq 2^j} F(u(\tau)) d\tau. 
\]
Then by the appropriate choice of  $t_0^k$  for each $G_k^j$,  Strichartz estimates, and \eqref{condi5},
\begin{align}\notag
&\|P_{\xi (t), \geq 2^j} u(t)\|_{L_t^2 L_x^{\frac{2d}{d-2},2} ([a,b] \times \mathbb{R}^d)}^2\\\notag
\lesssim&\sum_{G_k^j \subset [a,b]} \|P_{\xi (t), \geq 2^j} e^{i(t-t_0^k)\Delta}u(t_0^k)\|_{L_t^2 L_x^{\frac{2d}{d-2},2} (G_k^j \times \mathbb{R}^d)}^2\\\notag
&+ \sum_{G_k^j \subset [a,b]} \left\|P_{\xi (t), \geq 2^j} \int_{t_0^k}^t e^{i(t-\tau)\Delta} F(u(\tau)) d\tau\right\|_{L_t^2 L_x^{\frac{2d}{d-2},2} (G_k^j \times \mathbb{R}^d)}^2\\\notag
\lesssim&2^{k_0-j}\|u\|^2_{L_t^{\infty}L_x^2}+ \sum_{G_k^j \subset [a,b]} \left\|P_{\xi (t), \geq 2^j} \int_{t_0^k}^t e^{i(t-\tau)\Delta} F(u(\tau)) d\tau\right\|_{L_t^2 L_x^{\frac{2d}{d-2},2} (G_k^j \times \mathbb{R}^d)}^2\\\notag
\lesssim& 2^{k_0-j}+\sum_{G_k^j \subset [a,b]} \left\|P_{\xi (t), \geq 2^j} \int_{t_0^k}^t e^{i(t-\tau)\Delta} F(u(\tau)) d\tau\right\|_{L_t^2 L_x^{\frac{2d}{d-2},2} (G_k^j \times \mathbb{R}^d)}^2\\\label{estq0}
\lesssim& 2^{k_0-j}+ \|P_{\xi (\tau), \geq 2^{j-1}} F(u(\tau))\|_{L_t^2 L_x^{\frac{2d}{d+2},2} ([a,b] \times \mathbb{R}^d)}^2.
\end{align}

Next we observe that
\[P_{\xi(\tau),\geq 2^{j-1}}F(u)= P_{\xi(\tau),\geq 2^{j-1}}\left[\big(P_{\xi(\tau),\geq 2^{j-4}}u)(|x|^{-2}\ast|u|^2)+\big(P_{\xi(\tau),\leq 2^{j-4}}u\big)P_{\geq 2^{j-3}}(|x|^{-2}\ast|u|^2)\right].
\]
Therefore, by H\"older's inequality, Lemma \ref{lem:weakYoung} and Bernstein's inequality (Lemma \ref{L:Bernstein}):
\begin{align}\notag
&\|P_{\xi(\tau),\geq 2^{j-1}}F(u(\tau))\|_{L_{\tau}^2L_x^{\frac{2d}{d+2},2}([a,b]\times\R^d)}\\
\lesssim& \big\|(P_{\xi(\tau),\geq 2^{j-4}}u) (|x|^{-2}\ast|P_{\xi(\tau),\geq C(\eta)N(\tau)}u|^2)\big\|_{L_{\tau}^2L_x^{\frac{2d}{d+2},2}([a,b]\times\R^d)} \notag\\
&+\big\|(P_{\xi(\tau),\geq 2^{j-4}}u) (|x|^{-2}\ast|P_{\xi(\tau),\leq C(\eta)N(\tau)}u|^2)\big\|_{L_{\tau}^2L_x^{\frac{2d}{d+2},2}([a,b]\times\R^d)}\notag\\
&+\big\|(P_{\xi(\tau),\leq 2^{j-4}}u)P_{\geq 2^{j-3}}(|x|^{-2}\ast O(P_{\xi(\tau),\geq C(\eta)N(\tau)}uP_{\xi(\tau),\geq 2^{j-4}}u))\big\|_{L_{\tau}^2L_x^{\frac{2d}{d+2},2}([a,b]\times\R^d)}\notag\\
&+\big\|(P_{\xi(\tau),\leq 2^{j-4}}u)P_{\geq 2^{j-3}}(|x|^{-2}\ast O(P_{\xi(\tau),\leq C(\eta)N(\tau)}uP_{\xi(\tau),\geq 2^{j-4}}u))\big\|_{L_{\tau}^2L_x^{\frac{2d}{d+2},2}([a,b]\times\R^d)}\notag\\
\lesssim& \|(P_{\xi(\tau),\geq 2^{j-4}}u\|_{L_{\tau}^2L_x^{\frac{2d}{d-2},2}}\||x|^{-2}\ast |P_{\xi(\tau),\geq C(\eta)N(\tau)}u|^2\|_{L_{\tau}^{\infty}L_x^{\frac{d}{2},\infty}}\notag\\
&+\big\|(P_{\xi(\tau),\geq 2^{j-4}}u) (|x|^{-2}\ast|P_{\xi(\tau),\leq C(\eta)N(\tau)}u|^2)\big\|_{L_{\tau}^2L_x^{\frac{2d}{d+2},2}([a,b]\times\R^d)}\notag\\
&+\|P_{\xi(\tau),\geq C(\eta)N(\tau)}u\|_{L_{\tau}^{\infty}L_x^{2}}\|u\|_{L_{\tau}^{\infty}L_x^2}\|P_{\xi(\tau),\geq 2^{j-4}}u\|_{L_{\tau}^2L_x^{\frac{2d}{d-2},2}}\notag\\
&+\big\|(P_{\xi(\tau),\leq 2^{j-4}}u)P_{\geq 2^{j-3}}(|x|^{-2}\ast O(P_{\xi(\tau),\leq C(\eta)N(\tau)}uP_{\xi(\tau),\geq 2^{j-4}}u))\big\|_{L_{\tau}^2L_x^{\frac{2d}{d+2},2}([a,b]\times\R^d)}\notag\\
\lesssim& (\eta+\eta^2)\|P_{\xi(\tau),\geq 2^{j-4}}u\|_{L_{\tau}^{2}L_x^{\frac{2d}{d-2},2}}\notag\\
&+\big\|(P_{\xi(\tau),\leq 2^{j-4}}u)P_{\geq 2^{j-3}}(|x|^{-2}\ast O(P_{\xi(\tau),\leq C(\eta)N(\tau)}uP_{\xi(\tau),\geq 2^{j-4}}u))\big\|_{L_{\tau}^2L_x^{\frac{2d}{d+2},2}([a,b]\times\R^d)}\notag\\
&+\|P_{\xi(\tau),\geq 2^{j-4}}u (|x|^{-2}\ast|P_{\xi(\tau),\leq C(\eta)N(\tau)}u|^2)\|_{L_{\tau}^2L_x^{\frac{2d}{d+2},2}([a,b]\times\R^d)},\label{estq}
\end{align}
where 
\begin{align*}
&O(P_{\xi(\tau),\leq(\geq) C(\eta)N(\tau)}uP_{\xi(\tau),\geq 2^{j-4}}u)\notag\\
=&P_{\xi(\tau),\leq(\geq) C(\eta)N(\tau)}u\overline{P_{\xi(\tau),\geq 2^{j-4}}u}+\overline{P_{\xi(\tau),\leq(\geq) C(\eta)N(\tau)}u}P_{\xi(\tau),\geq 2^{j-4}}u.
\end{align*}

Next, let $J_l$ be a complete small interval such that
$C(\eta)N(J_l)\ll2^j$. By Proposition \ref{zzzz2}, we have
$N(t)\sim_u N(J_l)$ and
$|\xi(t)-\xi(J_l)|\lesssim_u N(J_l)$ on $J_l$. Taking $C(\eta)$
sufficiently large, set
\[
 \xi_0=\xi(J_l),\qquad M=4C(\eta)N(J_l),\qquad N=2^{j-6}.
\]
The Fourier supports give the reproducing identities
\begin{align*}
 P_{\xi(t),\leq C(\eta)N(t)}
 &=P_{\xi(t),\leq C(\eta)N(t)}P_{\xi_0,\leq M},\\
 P_{\xi(t),\geq2^{j-4}}
 &=P_{\xi(t),\geq2^{j-4}}P_{\xi_0,\geq N}.
\end{align*}
The kernel of the first time-dependent projection satisfies
\begin{align*}
 &\sup_{t\in J_l}\left|
 [C(\eta)N(t)]^d e^{iy\cdot\xi(t)}
       \varphi^\vee(C(\eta)N(t)y)\right|\\
 &\qquad\lesssim_u M^d\langle My\rangle^{-d-1},\qquad
 \int_{\R^d}M^d\langle My\rangle^{-d-1}\,dy\lesssim_d1.
\end{align*}
For the high-frequency projection, we use
$P_{\xi(t),\geq2^{j-4}}=1-P_{\xi(t),\leq2^{j-4}}$; the low-frequency
kernel satisfies the same bound with $M$ replaced by $2^{j-4}$.
Thus Minkowski's inequality, spatial translation invariance and
Corollary \ref{bilinear} give
\begin{align*}
 &\|(P_{\xi(t),\geq2^{j-4}}u)
        (P_{\xi(t),\leq C(\eta)N(t)}u)\|_{L_{t,x}^2(J_l\times\R^d)}\\
 &\lesssim_u\sup_{y,z\in\R^d}
       \|(\tau_zP_{\xi_0,\geq N}u)
         (\tau_yP_{\xi_0,\leq M}u)\|_{L_{t,x}^2(J_l\times\R^d)}\\
 &\lesssim_u \frac{M^{(d-1)/2}}{N^{1/2}}
                  \|u\|_{U_\Delta^2(J_l\times\R^d)}^2.
\end{align*}
Here the high-frequency factor is decomposed into annuli of radii
$2^lN$, $l\geq0$, and the factors $2^{-l/2}$ are summed. The same bound
holds with either factor conjugated. By H\"older's inequality, Sobolev
embedding, Bernstein's inequality and \eqref{localboundU2}, we obtain
\begin{align}\notag
&\big\|(P_{\xi(\tau),\leq 2^{j-4}}u)P_{\geq 2^{j-3}}(|x|^{-2}\ast O(P_{\xi(\tau),\leq C(\eta)N(\tau)}uP_{\xi(\tau),\geq 2^{j-4}}u))\big\|_{L_{\tau}^2L_x^{\frac{2d}{d+2},2}(J_l\times\R^d)}\\\notag
\lesssim&\|u\|_{L_t^\infty L_x^2} \big\|P_{\geq 2^{j-3}}(|x|^{-2}\ast O(P_{\xi(\tau),\leq C(\eta)N(\tau)}uP_{\xi(\tau),\geq 2^{j-4}}u))\big\|_{L_t^2L_x^d(J_l\times\R^d)}\\\notag
\lesssim&\big\||\nabla|^{-\frac{d-2}2}P_{\geq 2^{j-3}} O(P_{\xi(\tau),\leq C(\eta)N(\tau)}uP_{\xi(\tau),\geq 2^{j-4}}u)\big\|_{L_{t,x}^2(J_l\times\R^d)}  \\
\lesssim& 2^{-(d-2)j/2}\|O(P_{\xi(\tau),\leq C(\eta)N(\tau)}uP_{\xi(\tau),\geq 2^{j-4}}u)\|_{L_{t,x}^2(J_l\times\R^d)}\notag\\\notag
\lesssim&2^{-(d-2)j/2} \frac{C(\eta)^{(d-1)/2}N(J_l)^{(d-1)/2}}{2^{(j-4)/2}}\\\label{equ:bilHSBe}
\simeq&\Big(\frac{C(\eta)N(J_l)}{2^j}\Big)^{\frac{d-1}2}.
\end{align}
On the other hand, by Lemma \ref{lem:weakYoung}  and Lemma \ref{lorennorm}, we have
\begin{align}\label{TRIVIAL1}
&\big\|(P_{\xi(\tau),\leq 2^{j-4}}u)P_{\geq 2^{j-3}}(|x|^{-2}\ast O(P_{\xi(\tau),\leq C(\eta)N(\tau)}uP_{\xi(\tau),\geq 2^{j-4}}u))\big\|_{L_{\tau}^2L_x^{\frac{2d}{d+2},2}(J_l\times\R^d)}\notag\\
\lesssim&\|u\|^2_{L_{\tau}^{\infty}L_x^2}\|u\|_{L_{\tau}^{2}L_x^{\frac{2d}{d-2},2}(J_l\times\R^d)}\lesssim 1.
\end{align}
For $C(\eta)N(J_l)\ll2^j$, we combine this with \eqref{equ:bilHSBe}.
For $C(\eta)N(J_l)\gtrsim2^j$, we use \eqref{TRIVIAL1} directly. In both
cases, we obtain
\begin{align}\label{estq1}
&\big\|(P_{\xi(\tau),\leq 2^{j-4}}u)P_{\geq 2^{j-3}}(|x|^{-2}\ast O(P_{\xi(\tau),\leq C(\eta)N(\tau)}uP_{\xi(\tau),\geq 2^{j-4}}u))\big\|_{L_{\tau}^2L_x^{\frac{2d}{d+2},2}(J_l\times\R^d)}\notag\\
\lesssim& \Big(\frac{C(\eta)N(J_l)}{2^j}\Big)^{1/2}.
\end{align}
Inserting \eqref{estq1} into \eqref{estq} and then recalling the definition of $X$-norm, we have
\begin{align}\notag
&\|P_{\xi(\tau),\geq 2^{j-1}}F(u(\tau))\|_{L_{\tau}^2L_x^{\frac{2d}{d+2},2}([a,b]\times\R^d)}\\\notag
\lesssim& (\eta+\eta^2)2^{(k_0-j)/2}\|u\|_{X([a,b])}
+\Big(\sum_{J_l\subset[a,b]}\frac{C(\eta)N(J_l)}{2^j}\Big)^{1/2}\\
&+\big\|(P_{\xi(\tau),\geq 2^{j-4}}u)(|x|^{-2}\ast|P_{\xi(\tau),\leq C(\eta)N(\tau)}u|^2)\big\|_{L_{\tau}^2L_x^{\frac{2d}{d+2},2}([a,b]\times\R^d)} \notag\\
\lesssim& (\eta+\eta^2)2^{(k_0-j)/2}\|u\|_{X([a,b])} +C(\eta)^{1/2}\eta_{\ast}^{1/2}2^{(k_0-j)/2}\notag\\
&+\big\|(P_{\xi(\tau),\geq 2^{j-4}}u)(|x|^{-2}\ast|P_{\xi(\tau),\leq C(\eta)N(\tau)}u|^2)\big\|_{L_{\tau}^2L_x^{\frac{2d}{d+2},2}([a,b]\times\R^d)}.\label{estq2}
\end{align}
It remains to estimate the last term in \eqref{estq2}, which is non-local. To overcome this difficulty,  we use the double frequency decomposition technique, which builds upon an idea of Chae, Cho, and Lee \cite{CCS}.

\begin{lemma}[Bilinear estimate for the non-local term]\label{bilinear2}
Let $M\ll N$ and $J$ be an interval. Then for any $\xi_0\in\R^d$,
\[\left\|(P_{\xi_0, \geq N} u)|\nabla|^{2-d}\left(\left|P_{\xi_0, \leq M} u\right|^{2}\right)\right\|_{L_{t}^{2} L_{x}^{\frac{2d}{d+2},2}\left(J \times \mathbb{R}^{d}\right)}\lesssim \left(\frac{MM(u)} {N}\right)^{1/2}\|u\|^2_{U_{\Delta}^2(J\times\R^d)},\]
with the implicit constant only depends on $d$.
\end{lemma}

For the moment, let us take this lemma for granted and proceed to finish
the proof. On each complete $J_l$ with $C(\eta)N(J_l)\ll2^j$, take
$\xi_0,M,N$ as above. The reproducing identity, Cauchy--Schwarz and the
kernel bound give
\[
 |P_{\xi(t),\leq C(\eta)N(t)}u(t,x)|^2
 \lesssim_u\int_{\R^d}M^d\langle My\rangle^{-d-1}
              |\tau_yP_{\xi_0,\leq M}u(t,x)|^2\,dy.
\]
Since $V\geq0$, we may insert this bound into the convolution with $V$.
Using the high-frequency kernel bound and Minkowski's inequality, followed
by the translated form of Lemma \ref{bilinear2} proved below, we obtain
\begin{align*}
 &\big\|(P_{\xi(t),\geq2^{j-4}}u)
          (V*|P_{\xi(t),\leq C(\eta)N(t)}u|^2)\big\|_
             {L_t^2L_x^{\frac{2d}{d+2},2}(J_l\times\R^d)}\\
 &\lesssim_u\sup_{y,z\in\R^d}
       \big\|(\tau_zP_{\xi_0,\geq N}u)
          (V*|\tau_yP_{\xi_0,\leq M}u|^2)\big\|_
             {L_t^2L_x^{\frac{2d}{d+2},2}(J_l\times\R^d)}\\
 &\lesssim_u\left(\frac MN\right)^{1/2}M(u)^{1/2}
                      \|u\|_{U_\Delta^2(J_l\times\R^d)}^2
 \lesssim_u\left(\frac{C(\eta)N(J_l)}{2^j}\right)^{1/2}.
\end{align*}
If $C(\eta)N(J_l)\gtrsim 2^j$, then similar to \eqref{TRIVIAL1}, we can prove that
\begin{align}\notag
 &\big\|(P_{\xi(\tau),\geq 2^{j-4}}u) (|x|^{-2}\ast|P_{\xi(\tau),\leq C(\eta)N(\tau)}u|^2)\big\|_{L_{\tau}^2L_x^{\frac{2d}{d+2},2}(J_l\times\R^d)}\\
 \lesssim&  1\lesssim\big(\tfrac{C(\eta)N(J_l)}{2^j}\big)^{1/2}.
\end{align}
Therefore,
\begin{align}\label{estq3}
&\big\|(P_{\xi(\tau),\geq 2^{j-4}}u) (|x|^{-2}\ast|P_{\xi(\tau),\leq C(\eta)N(\tau)}u|^2)\big\|_{L_{\tau}^2L_x^{\frac{2d}{d+2},2}([a,b]\times\R^d)}\notag\\
\lesssim&\Big(\sum_{J_l\subset[a,b]}\frac{C(\eta)N(J_l)}{2^j}\Big)^{1/2}\lesssim C(\eta)2^{(k_0-j)/2}.
\end{align}
The possibly incomplete final small interval is estimated directly by the local Strichartz and weak Young inequalities. Its contribution to each squared spacetime estimate is $O(1)$, and $2^{j-k_0}O(1)\lesssim1$ for $j\leq k_0$. All sums involving $N(J_l)$ above are taken over complete intervals; their sum is bounded by a constant times $\int_a^bN(t)^3dt$. In view of \eqref{estq0}, \eqref{estq2} and \eqref{estq3}, choosing $\eta$ sufficiently small closes the argument.
\end{proof}
Now we prove Lemma \ref{bilinear2}.

\begin{proof}[Proof of Lemma \ref{bilinear2}]
We make the homogeneous Littlewood--Paley decomposition
\[
 |\nabla|^{2-d}=\sum_{k\in\mathbb Z}\dot P_{2^k}|\nabla|^{2-d}
              =\sum_{k\in\mathbb Z}|\nabla|_k^{2-d}.
\]
The kernel $\mathcal K_k$ of $|\nabla|_k^{2-d}$ is a Schwartz function.
By homogeneity and integration by parts, we have
\begin{equation}\label{decayk}
 |\mathcal K_k(y)|\lesssim_d2^{2k}\langle2^ky\rangle^{-d-1},\qquad
 \|\mathcal K_k\|_{L^1}\lesssim_d2^{(2-d)k}.
\end{equation}
Since $|P_{\xi_0,\leq M}u|^2$ has Fourier support in $\{|\xi|\leq4M\}$,
it suffices to sum over $2^k\leq4M$.
For each $k\in\mathbb Z$, let $\{Q_a^k\}_{a\in\mathbb Z^d}$ be the
partition of $\mathbb R^d$ into cubes of side length $2^k$. Then
\begin{equation}\label{tri1}
 |P_{\xi_0,\leq M}u|^2
 =\sum_{a,a'\in\mathbb Z^d}
 (P_{\xi_0,\leq M}P_{Q_a^k}u)
 \overline{P_{\xi_0,\leq M}P_{Q_{a'}^k}u}.
\end{equation}
The Fourier support of a product on the right-hand side is contained
in $Q_a^k-Q_{a'}^k$. Consequently,
\[
 |\nabla|_k^{2-d}\big((P_{\xi_0,\leq M}P_{Q_a^k}u)
               \overline{P_{\xi_0,\leq M}P_{Q_{a'}^k}u}\big)\neq0
 \quad\Longrightarrow\quad
 \operatorname{dist}(Q_a^k,Q_{a'}^k)\leq2^{k+10}.
\]
For each $a$, the number of such $a'$ is bounded by a constant depending
only on $d$. We keep the cube summation inside the norm and use the
embedding $L^{\frac{2d}{d+2}}\hookrightarrow L^{\frac{2d}{d+2},2}$,
Minkowski's inequality, spatial H\"older's inequality, and
Cauchy--Schwarz in $a,a'$ to obtain
\begin{align*}
 &\big\|(P_{\xi_0,\geq N}u)|\nabla|_k^{2-d}
                         (|P_{\xi_0,\leq M}u|^2)\big\|_
                         {L_t^2L_x^{\frac{2d}{d+2},2}(J\times\R^d)}\\
 &\lesssim_d\int_{\R^d}|\mathcal K_k(y)|
 \Bigg\|\sum_{\substack{a,a'\in\mathbb Z^d\\
                 \operatorname{dist}(Q_a^k,Q_{a'}^k)\leq2^{k+10}}}
 \big\|(P_{\xi_0,\geq N}u(t))
       (P_{\xi_0,\leq M}P_{Q_a^k}\tau_yu(t))\big\|_{L_x^2}\\
 &\hspace{28ex}\times
       \|P_{\xi_0,\leq M}P_{Q_{a'}^k}\tau_yu(t)\|_{L_x^d}
 \Bigg\|_{L_t^2(J)}\,dy\\
 &\lesssim_d\int_{\R^d}|\mathcal K_k(y)|
 \Bigg\|\left(\sum_a
       \big\|(P_{\xi_0,\geq N}u(t))
       (P_{\xi_0,\leq M}P_{Q_a^k}\tau_yu(t))\big\|_{L_x^2}^{2}
       \right)^{1/2}\\
 &\hspace{20ex}\times
       \left(\sum_{a'}
       \|P_{\xi_0,\leq M}P_{Q_{a'}^k}\tau_yu(t)\|_{L_x^d}^{2}
       \right)^{1/2}\Bigg\|_{L_t^2(J)}\,dy.
\end{align*}
By Bernstein's inequality, Plancherel's theorem and conservation of mass,
\begin{equation}\label{bi0}
 \sup_{t\in J}\left(\sum_a
       \|P_{\xi_0,\leq M}P_{Q_a^k}\tau_yu(t)\|_{L_x^d}^{2}
       \right)^{1/2}
 \lesssim_d2^{k(d-2)/2}M(u)^{1/2}.
\end{equation}
To estimate the other factor, we decompose $P_{\xi_0,\geq N}u$ into
annuli of radii $2^lN$, $l\geq0$, and apply Corollary \ref{bilinear}.
Only cubes meeting $B(\xi_0,2M)$ occur. Taking the separation in $M\ll N$
sufficiently large depending on $d$, their doubles lie in
$B(\xi_0,N/4)$. The factors $2^{-l/2}$ are summable. Hence, by
Lemma \ref{Xdeltanorm} and the translation invariance of $U_\Delta^2$,
\begin{align}\label{bi1}
 &\left(\sum_a\big\|(P_{\xi_0,\geq N}u)
       (P_{\xi_0,\leq M}P_{Q_a^k}\tau_yu)\big\|_{L_{t,x}^2(J\times\R^d)}^2
       \right)^{1/2}\notag\\
 &\lesssim_d\frac{2^{k(d-1)/2}}{N^{1/2}}
       \|u\|_{U_\Delta^2(J\times\R^d)}
       \left(\sum_a\|P_{\xi_0,\leq M}P_{Q_a^k}u\|_
                            {U_\Delta^2(J\times\R^d)}^2\right)^{1/2}\notag\\
 &\lesssim_d\frac{2^{k(d-1)/2}}{N^{1/2}}
                  \|u\|_{U_\Delta^2(J\times\R^d)}^2.
\end{align}
Both bounds are uniform in $y$. Using \eqref{decayk}, \eqref{bi0}
and \eqref{bi1}, we conclude that
\begin{align*}
 &\big\|(P_{\xi_0,\geq N}u)|\nabla|^{2-d}
                (|P_{\xi_0,\leq M}u|^2)\big\|_
                {L_t^2L_x^{\frac{2d}{d+2},2}(J\times\R^d)}\\
 &\lesssim_d\sum_{2^k\leq4M}
       2^{(2-d)k}2^{k(d-2)/2}\frac{2^{k(d-1)/2}}{N^{1/2}}
       M(u)^{1/2}\|u\|_{U_\Delta^2(J\times\R^d)}^2\\
 &\lesssim_d\left(\frac{MM(u)}N\right)^{1/2}
                  \|u\|_{U_\Delta^2(J\times\R^d)}^2,
\end{align*}
as desired.
The estimates \eqref{bi0} and \eqref{bi1} are unchanged if the high-frequency
factor and the low-frequency factor inside the square are translated
independently. Indeed, spatial translations commute with all the projections and
preserve the $L_x^2$ and $U_\Delta^2$ norms. Thus the same argument gives
\[
 \sup_{y,z\in\R^d}
 \big\|(\tau_zP_{\xi_0,\geq N}u)
       |\nabla|^{2-d}(|\tau_yP_{\xi_0,\leq M}u|^2)\big\|_
             {L_t^2L_x^{\frac{2d}{d+2},2}(J\times\R^d)}
 \lesssim_d\left(\frac{MM(u)}N\right)^{1/2}
                  \|u\|_{U_\Delta^2(J\times\R^d)}^2.
\]
\end{proof}

Now let $u$ be a solution as in Theorem \ref{reduction}, and let $[a,b]$ be a compact interval satisfying
\[
 \int_a^bN(t)^3\,dt=K>0.
\]
Choose a nonnegative integer $k_0$ such that
\begin{align*}
 &\max\left\{1,\int_a^b\int_{\R^d}|u(t,x)|^{\frac{2(d+2)}d}\,dx\,dt\right\}
 \leq 2^{k_0}\\
 &\hspace{3ex}<2\max\left\{1,\int_a^b\int_{\R^d}|u(t,x)|^{\frac{2(d+2)}d}\,dx\,dt\right\},
\end{align*}
and let
\[
 \lambda=\frac{\eta_{\ast}2^{k_0}}K,
 \qquad v(t,x)=\lambda^{d/2}u(\lambda^2t,\lambda x).
\]
Then $v$ is an almost-periodic solution with frequency scale function
$N_v(t)=\lambda N(\lambda^2t)$ and frequency center function
$\xi_v(t)=\lambda\xi(\lambda^2t)$. By a change of variables, we have
\begin{gather*}
 \int_{a/\lambda^2}^{b/\lambda^2}\int_{\R^d}|v(t,x)|^{\frac{2(d+2)}d}\,dx\,dt
 =\int_a^b\int_{\R^d}|u(t,x)|^{\frac{2(d+2)}d}\,dx\,dt\leq2^{k_0},\\
 \int_{a/\lambda^2}^{b/\lambda^2}N_v(t)^3\,dt
 =\lambda K=\eta_{\ast}2^{k_0}.
\end{gather*}
Moreover, $v$ satisfies \eqref{condi1}. Therefore, by Theorem \ref{longtime},
\[
 \sup_{0\leq j\leq k_0}2^{(j-k_0)/2}
 \|P_{\xi_v(t),\geq2^j}v\|_{L_t^2L_x^{\frac{2d}{d-2},2}
 ([a/\lambda^2,b/\lambda^2]\times\R^d)}\lesssim1.
\]
Since
\begin{align*}
 &\|P_{\xi_v(t),\geq2^j}v\|_{L_t^2L_x^{\frac{2d}{d-2},2}
 ([a/\lambda^2,b/\lambda^2]\times\R^d)}\\
 &=\|P_{\xi(t),\geq2^j/\lambda}u\|_{L_t^2L_x^{\frac{2d}{d-2},2}([a,b]\times\R^d)}
\end{align*}
and $2^{k_0}/\lambda=\eta_{\ast}^{-1}K$, we obtain the desired bound for
$2^{-k_0}\eta_{\ast}^{-1}K\leq2^j\leq\eta_{\ast}^{-1}K$ by choosing adjacent dyadic cutoffs.
For $2^j<2^{-k_0}\eta_{\ast}^{-1}K$, the local Strichartz estimate on at most
$2^{k_0}$ small intervals gives
\[
 \|P_{\xi(t),\geq2^j}u\|_{L_t^2L_x^{\frac{2d}{d-2},2}([a,b]\times\R^d)}
 \lesssim2^{k_0/2}\leq(2^{-j}\eta_{\ast}^{-1}K)^{1/2}.
\]
Combining these estimates and relabeling the dyadic index, we conclude that
\begin{equation}\label{longtime02}
 \sup_{\substack{j\in\mathbb Z\\2^j\leq\eta_{\ast}^{-1}K}}
 2^{j/2}\eta_{\ast}^{1/2}K^{-1/2}
 \|P_{\xi(t),\geq2^j}u\|_{L_t^2L_x^{\frac{2d}{d-2},2}([a,b]\times\R^d)}
 \lesssim1.
\end{equation}

With the above  estimate in hand,  similar to the case of the mass-critical nonlinear Schr\"odinger equation (cf. \cite{Dodson-JAMS,Dodson-Duke,Dodson-AJM} ), we can derive the following commute operator estimate, which will be used to preclude the existence of quasi-solitary solutions. 
\begin{proposition}[Frequency truncation error estimate]\label{reminder}
    Let $a:[0, \infty) \times \R^d \rightarrow \R^d$ be a map which is odd in the spatial variable (i.e. $a(t, x)=-a(t,-x)$ ) and for which there exists a constant $C_0(a)>0$ such that
    \begin{equation}\label{amapb}
        \max \left\{\|a\|_{L_{t, x}^{\infty}\left([0, \infty) \times \R^d\right)}, \sup _{x \in \R^d}|x|\|\nabla a(x)\|_{L_t^{\infty}([0, \infty))}\right\} \leq C_0(a)
    \end{equation}
Let $u$ be the solution as in Theorem \ref{reduction} satisfying $\int_0^{\infty} N(t)^3dt=\infty$. Then  there exist  constants $\eta_0=\eta_0(u)\ll1$ and $K_0=K_0(u)\gg1$ such that if $\eta_{\ast}\leq \eta_0$ and $\int_0^T N(t)^3dt=K\geq K_0$, then
\begin{align*}
& \left| 2\int_0^{T} \int_{\mathbb{R}^{2d}} a(t, x-y) \cdot \operatorname{Im}\{\bar{w} \mathcal{N}\}(t, y) \operatorname{Im}\left\{\bar{w}\nabla w\right\}(t, x)  dx  dy  dt \right. \\
& \quad +\int_0^{T} \int_{\mathbb{R}^{2d}} a(t, x-y) \cdot|w(t, y)|^2 \operatorname{Re}\left\{\bar{\mathcal{N}}\nabla w\right\}(t, x) \, dx \, dy dt \\
& \quad \left. - \int_0^{T } \int_{\mathbb{R}^{2d}} a(t, x-y) \cdot |w(t, y)|^2 \operatorname{Re}\left\{\bar{w}\nabla \mathcal{N}\right\}(t, x)  dx  dy  dt \right| \\
=& \left| 2\int_0^{T} \int_{\mathbb{R}^{2d}} a(t, x-y) \cdot \operatorname{Im}\{\bar{w} \mathcal{N}\}(t, y) \operatorname{Im}\left\{\bar{w}\left(\nabla-i \xi(t)\right) w\right\}(t, x)  dx  dy  dt \right. \\
& \quad +\int_0^{T} \int_{\mathbb{R}^{2d}} a(t, x-y) \cdot|w(t, y)|^2 \operatorname{Re}\left\{\bar{\mathcal{N}}\left(\nabla-i \xi(t)\right) w\right\}(t, x) \, dx \, dy dt \\
& \quad \left. - \int_0^{T } \int_{\mathbb{R}^{2d}} a(t, x-y) \cdot |w(t, y)|^2 \operatorname{Re}\left\{\bar w\left(\nabla-i \xi(t)\right) \mathcal{N}\right\}(t, x)  dx  dy  dt \right| \\
 :=&\Big|\Gamma_1+\Gamma_2+\Gamma_3\Big|\lesssim_{u,\eta_{\ast}} C_0(a) o(K),
\end{align*}
where $\frac{o(K)}{K}\to0$ $K\to\infty$, where $w:=P_{\leq \eta_{\ast}^{-1}K}u$ and  $\mathcal{N}:=P_{\leq \eta_{\ast}^{-1}K}F(u)-F(P_{\leq \eta_{\ast}^{-1}K}u)$.
\end{proposition}
\begin{proof}
We estimate the terms $\Gamma_1,\Gamma_2,\Gamma_3$ individually.
Throughout this proof, choose $k_0$ such that
$2^{k_0}\leq\eta_{\ast}^{-1}K<2^{k_0+1}$.
By Lemma \ref{basicp} and \eqref{longtime02}, we have
\begin{equation}\label{longtime28}
 2^{(j-k_0)/2}\|P_{\xi(t),\geq2^j}u\|_{L_t^2L_x^{\frac{2d}{d-2}}}
 \lesssim
 2^{(j-k_0)/2}\|P_{\xi(t),\geq2^j}u\|_{L_t^2L_x^{\frac{2d}{d-2},2}}
 \lesssim1,\qquad j\leq k_0.
\end{equation}
Since $\xi(0)=0$, \eqref{condi1} implies
$\sup_{[0,T]}|\xi(t)|\leq\eta_1^{-1}K$.
Taking $\eta_{\ast}$ sufficiently small, we may assume that
$\eta_1^{-1}K\leq2^{k_0-12}$. Hence, by the reproducing formula for smooth Fourier multipliers, we obtain
\begin{equation}\label{longtime29}
 \|P_{\geq2^j}u\|_{L_t^2L_x^{\frac{2d}{d-2},2}}
 \lesssim\|P_{\xi(t),\geq2^{j-2}}u\|_{L_t^2L_x^{\frac{2d}{d-2},2}},
 \qquad k_0-8\leq j\leq k_0.
\end{equation}
Therefore,
\begin{equation}\label{longtime2}
 \|P_{\geq2^j}u\|_{L_t^2L_x^{\frac{2d}{d-2}}}
 \lesssim\|P_{\geq2^j}u\|_{L_t^2L_x^{\frac{2d}{d-2},2}}
 \lesssim1,\qquad k_0-8\leq j\leq k_0.
\end{equation}
From now on, we fix $\eta_{\ast}$ and $K$ as above. Interpolating
\eqref{longtime28} with conservation of mass, we obtain, for every admissible
pair $(q,r)$,
\[
 \|P_{\xi(t),2^j}u\|_{L_t^qL_x^{r,2}}
 \lesssim_u (2^{-j}\eta_{\ast}^{-1}K)^{1/q},
 \qquad 2^j\leq8\eta_{\ast}^{-1}K.
\]
For $j>k_0$, the same bound follows from the high-frequency estimate at
$2^{k_0}$ and the boundedness of the smooth multipliers; only a fixed number
of additional dyadic scales occurs.

\textbf{Contribution of $\Gamma_1$:}
We first perform the near-far frequency decomposition
\[
u=u_{l}+u_{h}, \qquad u_{l} := P_{\leq\frac{\eta_{\ast}^{-1}K}{32}}u,
\]
into the expression $\operatorname{Im}\{\bar{w}\mathcal{N}\}$, and algebraically expand, grouping terms based on the number of factors $u_{l}$ and $u_{h}$ to obtain the decomposition
\[
\operatorname{Im}\{\bar{w}\mathcal{N}\} = F_0 + F_1 + F_2 + F_3 + F_4, 
\]
where $F_j$ consists of $4 - j$ factors $u_{l}$ and $j$ factors $u_{h}$, for $j = 0, \ldots, 4$. As reader may verify, the precise formulae for $F_j$ are given below
\begin{align*}
F_0 :=& \operatorname{Im}\left\{ \overline{u_l} \, P_{\leq \eta_{\ast}^{-1}K}\Bigl(|\nabla|^{2-d}(|u_l|^2) u_l\Bigr) - |u_l|^2 \, |\nabla|^{2-d}(|u_l|^2) \right\}, \\
F_1 :=& {\rm Im}\Big\{ \overline{u_l} \, P_{\leq \eta_{\ast}^{-1}K}\Bigl(|\nabla|^{2-d}(|u_l|^2) u_h\Bigr) 
+ 2\overline{u_l} \, P_{\leq \eta_{\ast}^{-1}K}\Bigl(|\nabla|^{2-d}\bigl(\operatorname{Re}\{u_l \overline{u_h}\}\bigr) u_l\Bigr)\\ 
&\qquad+ \bigl(\overline{P_{\leq \eta_{\ast}^{-1}K} u_h}\bigr) P_{\leq \eta_{\ast}^{-1}K}\Bigl(|\nabla|^{2-d}(|u_l|^2) u_l\Bigr) \Big\} \\
& - \operatorname{Im}\Big\{ \overline{u_l} \, |\nabla|^{2-d}(|u_l|^2) \bigl(P_{\leq \eta_{\ast}^{-1}K} u_h\bigr) 
+ 2\overline{u_l} \, |\nabla|^{2-d}\Bigl(\operatorname{Re}\{u_l \bigl(\overline{P_{\leq \eta_{\ast}^{-1}K} u_h}\bigr)\}\Bigr) u_l\\ 
&\qquad+ \bigl(\overline{P_{\leq \eta_{\ast}^{-1}K} u_h}\bigr) |\nabla|^{2-d}(|u_l|^2) u_l \Big\}, \\
F_2 :=& 2\operatorname{Im}\left\{ \bigl(\overline{P_{\leq \eta_{\ast}^{-1}K} u_l}\bigr) P_{\leq \eta_{\ast}^{-1}K}\Bigl(|\nabla|^{2-d}\bigl(\operatorname{Re}\{u_l \bar{u}_h\}\bigr) u_h\Bigr) \right\} \\
&+ 2\operatorname{Im}\left\{ \bigl(\overline{P_{\leq \eta_{\ast}^{-1}K} u_h}\bigr) P_{\leq \eta_{\ast}^{-1}K}\Bigl(|\nabla|^{2-d}\bigl(\operatorname{Re}\{u_l \bar{u}_h\}\bigr) u_l\Bigr) \right\} \\
& + \operatorname{Im}\Big\{ \bigl(\overline{P_{\leq \eta_{\ast}^{-1}K} u_l}\bigr) P_{\leq \eta_{\ast}^{-1}K}\Bigl(|\nabla|^{2-d}(|u_h|^2) u_l\Bigr) \\
&\qquad+ \bigl(\overline{P_{\leq \eta_{\ast}^{-1}K} u_h}\bigr) P_{\leq \eta_{\ast}^{-1}K}\Bigl(|\nabla|^{2-d}(|u_l|^2) u_h\Bigr) \Big\} \\
& - 2\operatorname{Im}\left\{ \overline{u_l} \, |\nabla|^{2-d}\Bigl(\operatorname{Re}\{u_l \bigl(\overline{P_{\leq \eta_{\ast}^{-1}K} u_h}\bigr)\}\Bigr) \bigl(P_{\leq \eta_{\ast}^{-1}K} u_h\bigr) \right\} \\
&- 2\operatorname{Im}\left\{ \bigl(\overline{P_{\leq \eta_{\ast}^{-1}K} u_h}\bigr) |\nabla|^{2-d}\Bigl(\operatorname{Re}\{u_l \bigl(\overline{P_{\leq \eta_{\ast}^{-1}K} u_h}\bigr)\}\Bigr) u_l \right\} \\
& - \operatorname{Im}\left\{ |u_l|^2 \, |\nabla|^{2-d}\bigl(|P_{\leq \eta_{\ast}^{-1}K} u_h|^2\bigr) + |P_{\leq \eta_{\ast}^{-1}K} u_h|^2 \, |\nabla|^{2-d}(|u_l|^2) \right\}, \\
F_3 :=& \operatorname{Im}\left\{ \overline{u_l} \, P_{\leq \eta_{\ast}^{-1}K}\Bigl(|\nabla|^{2-d}(|u_h|^2) u_h\Bigr) 
+ \bigl(\overline{P_{\leq \eta_{\ast}^{-1}K} u_h}\bigr) P_{\leq \eta_{\ast}^{-1}K}\Bigl(|\nabla|^{2-d}(|u_h|^2) u_l\Bigr) \right\} \\
& + 2\operatorname{Im}\left\{ \bigl(\overline{P_{\leq \eta_{\ast}^{-1}K} u_h}\bigr) P_{\leq \eta_{\ast}^{-1}K}\Bigl(|\nabla|^{2-d}\bigl(\operatorname{Re}\{u_h \bar{u}_l\}\bigr) u_h\Bigr) \right\}, \\
F_4 :=& \operatorname{Im}\Big\{ \bigl(\overline{P_{\leq \eta_{\ast}^{-1}K} u_h}\bigr) P_{\leq \eta_{\ast}^{-1}K}\Bigl(|\nabla|^{2-d}(|u_h|^2) u_h\Bigr) \\
&\qquad- |P_{\leq \eta_{\ast}^{-1}K} u_h|^2 \, |\nabla|^{2-d}\bigl(|P_{\leq \eta_{\ast}^{-1}K} u_h|^2\bigr) \Big\}.
\end{align*}

We now consider the contributions of the $F_j$ to $\Gamma_1$ in groups. Hereafter, the mixed norm notation $L_t^p L_x^q$ implicitly denotes $L_t^p L_x^q([0, T] \times \mathbb{R}^d)$.

\begin{itemize}
\item Since the symbol of $P_{\leq \eta_{\ast}^{-1}K}$ is identically one on the Fourier support of $|\nabla|^{2-d}(|u_{l}|^2)u_{l}$, we see that $F_0 \equiv 0$. Hence, there is no contribution to $\Gamma_1$.

\item For the terms $F_2$, $F_3$, and $F_4$, thanks to H\"older's inequality, Bernstein's inequality, Lemma \ref{basicp}, the weak Young inequality, and the Hardy–Littlewood–Sobolev inequality, we place two $L_t^{2}L_x^{\frac{2d}{d-2},2}$ norms on $u_h$, while the remaining terms are estimated using the $L_t^{\infty}L_x^2$ norm, and then apply \eqref{longtime2} to obtain the bound
\[
\|F_2 + F_3 + F_4\|_{L_{t,x}^1} \lesssim \|u\|^2_{L_t^{\infty}L_x^2} \|u_h\|^2_{L_t^2L_x^{\frac{2d}{d-2},2}}\lesssim 1.
\]
\end{itemize}
Therefore, by the estimate \eqref{amapb} for the map $a$, Fubini-Tonelli, Cauchy-Schwarz, Plancherel's theorem, followed by conservation of mass, we obtain that
\[
\begin{aligned}
&\left| \int_0^{T} \int_{\mathbb{R}^{2d}} a(t, x - y) \cdot (F_2 + F_3 + F_4)(t, y) \operatorname{Im}\{\bar{w}(\nabla - i\xi(t))w\}(t, x) \,dx\,dy\,dt \right| \\
&\leq \|F_2 + F_3 + F_4\|_{L_{t,x}^1} \|a\|_{L_{t,x}^{\infty}} \|\bar{w}(\nabla - i\xi(t))w\|_{L_t^\infty L_x^1} \\
&\lesssim C_0(a)\|(\nabla - i\xi(t))w\|_{L_t^\infty L_x^2}.
\end{aligned}
\]

In order to show that the ultimate line can be made small, as measured by the parameter $\eta$, we first decompose
\[
w = P_{\xi(t), \leq C(\eta)N(t)} w + P_{\xi(t), \geq C(\eta)N(t)} w,
\]
where $C(\cdot)$ denotes the compactness modulus function for $u$. Since $N(t) \leq 1$, by the triangle inequality, the frequency localization property \eqref{localized}, Plancherel's theorem, and mass conservation, we obtain
\begin{align}\label{kineticb}
   \| (\nabla -i\xi(t))w\|_{L_{t}^{\infty}L_{x}^{2}} \notag
\le& \| (\nabla -i\xi(t))P_{\xi(t),\leq C(\eta)N(t)}w\|_{L_{t}^{\infty}L_{x}^{2}} \notag\\
&+\| (\nabla -i\xi(t))P_{\xi(t),\geq C(\eta)N(t)}w\|_{L_{t}^{\infty}L_{x}^{2}} \notag\\
\lesssim& C(\eta)\| N(t)\|_{L_{t}^{\infty}}M(u)^{1/2} + \eta \eta_{\ast}^{-1}KM(u)^{1/2} \notag\\
\lesssim& C(\eta) +\eta \eta_{\ast}^{-1}K. 
\end{align}
Thus we have shown
\begin{align*}
&\left|\int_{0}^{T}\int_{\mathbb{R}^{2d}}a(t,x - y)\cdot (F_{2} + F_{3} + F_{4})(t,y)\operatorname{Im}\{\bar{w} (\nabla -i\xi(t))w\} (t,x)\,dx\,dy\,dt\right| \\
\lesssim& C_0(a)[C(\eta) +\eta \eta_{\ast}^{-1}K].
\end{align*}

 Next, the contribution to $F_1$ with $u_h$ replaced by
$P_{\leq\eta_{\ast}^{-1}K/4}u_h$ vanishes, since all the outer projections
are the identity on the corresponding products. Thus only
$P_{\geq\eta_{\ast}^{-1}K/4}u_h=P_{\geq\eta_{\ast}^{-1}K/4}u$
contributes. We observe that $F_{1}$ has Fourier support in the region $\{\vert \xi \vert \geq \frac{\eta_{\ast}^{-1}K}{16}\}$. Hence the Fourier multiplier $\Delta^{-1}$ is well-defined on the Fourier support of $F_{1}$. Therefore we can integrate by parts in $y$ to obtain
 \begin{align*}
&\int_{0}^{T }\int_{\mathbb{R}^{2d}}\left(\frac{\Delta}{\Delta}F_{1}\right)(t,y)a(t,x - y)\cdot \operatorname{Im}\{\bar{w} (\nabla -i\xi(t))w\} (t,x)\,dx\,dy\,dt \\
=& \int_{0}^{T }\int_{\mathbb{R}^{2d}}\left(\frac{\partial_{k}}{\Delta}F_{1}\right)(t,y)(\partial_{k}a)(t,x - y)\cdot \operatorname{Im}\{\bar{w} (\nabla -i\xi(t))w\} (t,x)\,dx\,dy\,dt, 
 \end{align*}
where the right-hand side is written using Einstein summation. Next, by the $a$-map estimate \eqref{amapb}, Fubini-Tonelli, and H\"older's inequality, we have
\begin{align*}
&\left|\int_{0}^{T}\int_{\mathbb{R}^{2d}}\left(\frac{\partial_{k}}{\Delta}F_{1}\right)(t,y)(\partial_{k}a)(t,x - y)\cdot \operatorname{Im}\{\bar{w} (\nabla -i\xi(t))w\} (t,x)\,dx\,dy\,dt\right| \\
\lesssim C_0(a)& \| P_{\geq \frac{\eta_{\ast}^{-1}K}{32}}\frac{\nabla}{\Delta}F_{1}\|_{L_{t}^{\frac{2(d+2)}{d+4}}L_{x}^{\frac{d(d+2)}{d^2+2d-2}}}\| \tilde{V} *(|w|\,|(\nabla -i\xi(t))w|)\|_{L_{t}^{2(d+2)/d}L_{x}^{d(d+2)/2}},
\end{align*}
where $\tilde{V}(x) \coloneqq |x|^{-1}$. By the Hardy-Littlewood-Sobolev inequality followed by H\"older's inequality,
\begin{align*}
    \||\nabla |^{1-d}(|w|\,|(\nabla -i\xi(t))w|)\|_{L_{t}^{2(d+2)/d}L_{x}^{d(d+2)/2}}
&\lesssim \| w(\nabla -i\xi(t))w\|_{L_{t}^{2(d+2)/d}L_{x}^{\frac{d+2}{d+1}}} \\
&\lesssim \| w\|_{L_{t}^{\infty}L_{x}^{2}}\| (\nabla -i\xi(t))w\|_{L_{t}^{2(d+2)/d}L_{x}^{2(d+2)/d}} \\
&\lesssim \| (\nabla -i\xi(t))w\|_{L_{t}^{2(d+2)/d}L_{x}^{2(d+2)/d}},
\end{align*}
where we also use mass conservation to obtain the last inequality. Using the frequency localized property \eqref{localized}, interpolation estimate, Bernstein's inequality, \eqref{longtime29} and \eqref{longtime2}, we see that
\begin{align}
    &\| (\nabla -i\xi(t))w\|_{L_{t}^{2(d+2)/d}L_{x}^{2(d+2)/d}}\lesssim \sum_{2^j\leq 8\eta_{\ast}^{-1}K} 2^{j}\|P_{\xi(t),2^j}w\|_{L_{t}^{2(d+2)/d}L_{x}^{2(d+2)/d}}\notag\\
    \lesssim&\sum_{2^j\leq4C(\eta)} 2^{j}\|P_{\xi(t),2^j}P_{\xi(t),\leq C(\eta)N(t)}w\|_{L_{t}^{2(d+2)/d}L_{x}^{2(d+2)/d}}\\
    &+\sum_{2^j\leq 8\eta_{\ast}^{-1}K} 2^{j}\|P_{\xi(t),2^j}P_{\xi(t),\geq C(\eta)N(t)}w\|_{L_{t}^{2(d+2)/d}L_{x}^{2(d+2)/d}}\notag\\
    \lesssim& \sum_{2^j\leq4C(\eta)} 2^{j}2^{\frac{d(k_0-j)}{2(d+2)}}(2^{\frac{j-k_0}{2}}\|P_{\xi(t),2^j}u\|_{L_{t}^{2}L_{x}^{\frac{2d}{d-2}}})^{\frac{d}{d+2}}\\
    &+\sum_{2^j\leq 8\eta_{\ast}^{-1}K} 2^{j}2^{\frac{d(k_0-j)}{2(d+2)}}\eta^{\frac{2}{d+2}}(2^{\frac{j-k_0}{2}}\|P_{\xi(t),2^j}u\|_{L_{t}^{2}L_{x}^{\frac{2d}{d-2}}})^{\frac{d}{d+2}}\notag\\
    \lesssim&  [C(\eta)^{\frac{d+4}{2(d+2)}}(\eta_{\ast}^{-1}K)^{\frac{d}{2(d+2)}}+\eta^{\frac{2}{d+2}} \eta_{\ast}^{-1}K]\notag\\
    \lesssim& C(\eta)^{\frac{d+4}{2(d+2)}}(\eta_{\ast}^{-1}K)^{\frac{d}{2(d+2)}}+\eta^{\frac{2}{d+2}} \eta_{\ast}^{-1}K.\label{10q}
\end{align}
By Bernstein's inequality followed by unpacking the definition of $F_{1}$  and applying the triangle inequality,
\begin{align*}
   & \Big\| P_{\geq \frac{\eta_{\ast}^{-1}K}{32}}\frac{\nabla}{\Delta}F_{1}\Big\|_{L_{t}^{\frac{2(d+2)}{d+4}}L_{x}^{\frac{d(d+2)}{d^2+2d-2}}}\\
\lesssim& \eta_{\ast}K^{-1}\| F_{1}\|_{L_{t}^{\frac{2(d+2)}{d+4}}L_{x}^{\frac{d(d+2)}{d^2+2d-2}}} \\
\lesssim& \eta_{\ast}K^{-1}\| \overline{u_{l}} P_{\leq \eta_{\ast}^{-1}K}\big(|\nabla |^{2-d}(|u_{l}|^{2})(P_{\geq\frac{\eta_{\ast}^{-1}K}{4}}u)\big)\|_{L_{t}^{\frac{2(d+2)}{d+4}}L_{x}^{\frac{d(d+2)}{d^2+2d-2}}}\\
& +\eta_{\ast}K^{-1}\| \overline{u_{l}} P_{\leq \eta_{\ast}^{-1}K}\big(|\nabla |^{2-d}(\operatorname{Re}\{u_{l}\overline{P_{\geq\frac{\eta_{\ast}^{-1}K}{4}}u}\})u_{l}\big)\|_{L_{t}^{\frac{2(d+2)}{d+4}}L_{x}^{\frac{d(d+2)}{d^2+2d-2}}} \\
& +\eta_{\ast}K^{-1}\| \overline{u_{l}} |\nabla |^{2-d}(|u_{l}|^{2})(P_{\leq\eta_{\ast}^{-1}K}P_{\geq\eta_{\ast}^{-1}K/4}u)\|_{L_{t}^{\frac{2(d+2)}{d+4}}L_{x}^{\frac{d(d+2)}{d^2+2d-2}}} \\
\coloneqq& \mathrm{Term}_{1} + \mathrm{Term}_{2} + \mathrm{Term}_{3}.
\end{align*}

We observe that
\[
\mathrm{Term}_{1}\lesssim \eta_{\ast}K^{-1}\| P_{\geq \frac{\eta_{\ast}^{-1}K}{4}}u\|_{L_{t}^{2}L_{x}^{\frac{2d}{d-2}}}M(u)^{1/2}\| |\nabla |^{2-d}(|u_{l}|^{2})\|_{L_{t,x}^{d+2}}
\lesssim \eta_{\ast}K^{-1}\| |\nabla |^{2-d}(|u_{l}|^{2})\|_{L_{t,x}^{d+2}},
\]
where we use mass conservation and \eqref{longtime2} to obtain the last inequality. By the Hardy-Littlewood-Sobolev inequality,
\[
\| |\nabla |^{2-d}(|u_{l}|^{2})\|_{L_{t,x}^{d+2}}\lesssim \| u_{l}\|_{L_{t}^{2(d+2)}L_{x}^{\frac{2d(d+2)}{d^2+d-4}}}^{2}. 
\]
Since $\big(2(d+2),\frac{2d(d+2)}{d^2+2d-2}\big)$ is admissible,
interpolation, Bernstein's inequality and \eqref{longtime28} give
\begin{align}
 &\|u_l\|_{L_t^{2(d+2)}L_x^{\frac{2d(d+2)}{d^2+d-4}}}\notag\\
 &\lesssim_u\sum_{\substack{j\in\mathbb Z\\2^j\leq\eta_{\ast}^{-1}K/4}}
       2^{j/2}\|P_{\xi(t),2^j}u_l\|_
                   {L_t^{2(d+2)}L_x^{\frac{2d(d+2)}{d^2+2d-2}}}\notag\\
 &\lesssim_u(\eta_{\ast}^{-1}K)^{\frac1{2(d+2)}}
       \sum_{\substack{j\in\mathbb Z\\2^j\leq\eta_{\ast}^{-1}K/4}}
                  2^{\frac{d+1}{2(d+2)}j}
 \lesssim_u(\eta_{\ast}^{-1}K)^{1/2}.\label{zz0}
\end{align}
Here the summation range follows from the Fourier support of $u_l$ and
$\sup_{[0,T]}|\xi(t)|\leq2^{k_0-12}$.
Hence,
\[
\mathrm{Term}_{1}\lesssim 1,
\]
and by the same argument we also obtain
\[
\mathrm{Term}_{3}\lesssim 1. 
\]

For $\mathrm{Term}_2$, H\"older's inequality and the boundedness of the
smooth projections give
\begin{align*}
 \mathrm{Term}_2
 &\lesssim\eta_{\ast}K^{-1}
 \|u_l\|_{L_t^{2(d+2)}L_x^{\frac{2d(d+2)}{d^2+d-4}}}^{2}\\
 &\qquad\times
 \big\||\nabla|^{2-d}\operatorname{Re}
       (u_l\overline{P_{\geq\eta_{\ast}^{-1}K/4}u})\big\|_{L_t^2L_x^d}.
\end{align*}
By the Hardy--Littlewood--Sobolev inequality and \eqref{longtime2},
\begin{align*}
 &\big\||\nabla|^{2-d}\operatorname{Re}
       (u_l\overline{P_{\geq\eta_{\ast}^{-1}K/4}u})\big\|_{L_t^2L_x^d}\\
 &\lesssim_d\|u_l\|_{L_t^\infty L_x^2}
       \|P_{\geq\eta_{\ast}^{-1}K/4}u\|_
                         {L_t^2L_x^{\frac{2d}{d-2}}}
 \lesssim_u1.
\end{align*}
Together with \eqref{zz0}, this yields
\[
 \mathrm{Term}_2\lesssim_u
 (\eta_{\ast}^{-1}K)^{-1}(\eta_{\ast}^{-1}K)\lesssim_u1,
\]
and hence
\begin{align*}
    \left|\Gamma_1\right|\lesssim C_0(a)[C(\eta)^{\frac{d+4}{2(d+2)}}(\eta_{\ast}^{-1}K)^{\frac{d}{2(d+2)}}+\eta^{\frac{2}{d+2}} \eta_{\ast}^{-1}K].
\end{align*}

\vskip 0.1in

\textbf{Contribution of $\Gamma_2$:} By H\"older's inequality and
conservation of mass,
\begin{align*}
 |\Gamma_2|
 &\leq\|a\|_{L_{t,x}^\infty}\|w\|_{L_t^\infty L_x^2}^2
               \|\mathcal N(\nabla-i\xi(t))w\|_{L_{t,x}^1}\\
 &\lesssim_u C_0(a)\|\mathcal N\|_{L_t^{4/3}L_x^{\frac{2d}{d+1}}}
               \|(\nabla-i\xi(t))w\|_{L_t^4L_x^{\frac{2d}{d-1}}}.
\end{align*}
Mimicking the proof of \eqref{10q}, we can also show that
\begin{equation}\label{zz1}
    \| (\nabla -i\xi(t))w\|_{L_{t}^{\frac{2(d+1)}{d}}L_{x}^{\frac{2(d+1)}{d-1}}}\lesssim C(\eta)^{\frac{d+2}{2(d+1)}}\eta_{\ast}^{-\frac{d}{2(d+1)}}K^{\frac{d}{2(d+1)}}+\eta^{\frac{1}{d+1}}\eta_{\ast}^{-1}K
\end{equation}
and
\begin{equation}\label{zz12}
    \| (\nabla -i\xi(t))w\|_{L_{t}^{4}L_{x}^{\frac{2d}{d-1}}}\lesssim C(\eta)^{3/4}\eta_{\ast}^{-1/4}K^{1/4}+\eta^{1/2}\eta_{\ast}^{-1}K.
\end{equation}
Next, we decompose
\[
\mathcal{N} = \mathcal{N}_0 + \mathcal{N}_1 + \mathcal{N}_2 + \mathcal{N}_3,
\]
where $\mathcal{N}_j$ contains $j$ factors $u_{h}$ and $3 - j$ factors $u_{l}$. Specifically,
\begin{align*}
\mathcal{N}_{0} &:= P_{\leq \eta_{\ast}^{-1}K} (|\nabla|^{2-d}(|u_l|^2) u_l) - |\nabla|^{2-d}(|u_l|^2) u_l,\\
\mathcal{N}_{1} &:= [P_{\leq \eta_{\ast}^{-1}K}, |\nabla|^{2-d}(|u_l|^2)] u_h + 2P_{\leq \eta_{\ast}^{-1}K} (|\nabla|^{2-d}(\operatorname{Re}\{u_l \overline{u_h}\}) u_l) \\
&\qquad - 2|\nabla|^{2-d}(\operatorname{Re}\{u_l (\overline{P_{\leq \eta_{\ast}^{-1}K} u_h})\}) u_l,\\
\mathcal{N}_{2} &:= P_{\leq \eta_{\ast}^{-1}K} (|\nabla|^{2-d}(|u_h|^2) u_l) + 2P_{\leq \eta_{\ast}^{-1}K} (|\nabla|^{2-d}(\operatorname{Re}\{u_h \overline{u_l}\}) u_h) \\
&\qquad - |\nabla|^{2-d}(|P_{\leq \eta_{\ast}^{-1}K} u_h|^2) u_l - 2|\nabla|^{2-d}(\operatorname{Re}\{(P_{\leq \eta_{\ast}^{-1}K} u_h) \overline{u_l}\})(P_{\leq \eta_{\ast}^{-1}K} u_h),\\
\mathcal{N}_{3} &:= P_{\leq \eta_{\ast}^{-1}K} (|\nabla|^{2-d}(|u_h|^2) u_h) - |\nabla|^{2-d}(|P_{\leq \eta_{\ast}^{-1}K} u_h|^2)(P_{\leq \eta_{\ast}^{-1}K} u_h).
\end{align*}

Since the symbol of $P_{\leq \eta_{\ast}^{-1}K}$ is identically one on the Fourier support of $|\nabla|^{2-d}(|u_l|^2)u_l$, we see that $\mathcal{N}_0 \equiv 0$. Hence $\|\mathcal{N}_0\|_{L_t^{4/3}L_x^{\frac{2d}{d+1}}} = 0$.

By the triangle inequality,
\begin{align*}
&\|\mathcal{N}_{1}\|_{L_{t}^{4/3}L_{x}^{\frac{2d}{d+1}}}
\lesssim \| [P_{\leq \eta_{\ast}^{-1}K},|\nabla |^{2-d}(|u_{l}|^{2})] u_{h}\|_{L_{t}^{4/3}L_{x}^{\frac{2d}{d+1}}} \\
&+\| P_{\leq \eta_{\ast}^{-1}K}(|\nabla |^{2-d}(\operatorname{Re}\{u_{l}\overline{u_{h}}\}) u_{l}) - |\nabla |^{2-d}(\operatorname{Re}\{u_{l}(\overline{P_{\leq \eta_{\ast}^{-1}K}u_{h}})\}) u_{l}\|_{L_{t}^{4/3}L_{x}^{\frac{2d}{d+1}}} \\
&=: \mathrm{Term}_{1} + \mathrm{Term}_{2}. 
\end{align*}
By the fundamental theorem of calculus followed by Minkowski's and Hölder's inequalities,
\begin{align*}
   \mathrm{Term}_{1}
&\leq \eta_{\ast}K^{-1}\| |\nabla |^{2-d}(\nabla |u_{l}|^{2})\|_{L_{t}^{8/3}L_{x}^{4d/5}}\| u_{h}\|_{L_{t}^{8/3}L_{x}^{\frac{4d}{2d-3}}} \\
&\lesssim \eta_{\ast}K^{-1}\| \nabla |u_{l}|^{2}\|_{L_{t}^{8/3}L_{x}^{\frac{4d}{4d-3}}}\| u_{h}\|_{L_{t}^{8/3}L_{x}^{\frac{4d}{2d-3}}} \\
&\lesssim \eta_{\ast}K^{-1}\| u_{l}\|_{L_{t}^{\infty}L_{x}^{2}}\| (\nabla -i\xi(t))u_{l}\|_{L_{t}^{8/3}L_{x}^{\frac{4d}{2d-3}}}\| u_{h}\|_{L_{t}^{8/3}L_{x}^{\frac{4d}{2d-3}}} \\
&\lesssim 1, 
\end{align*}
where we used
$\nabla|u_l|^2=2\operatorname{Re}\{(\nabla-i\xi(t))u_l\,\overline{u_l}\}$.
Indeed, interpolation and Bernstein's inequality give
\[
 \|(\nabla-i\xi(t))u_l\|_{L_t^{8/3}L_x^{\frac{4d}{2d-3}}}
 \lesssim_u(\eta_{\ast}^{-1}K)^{3/8}
       \sum_{\substack{j\in\mathbb Z\\2^j\leq\eta_{\ast}^{-1}K/4}}2^{5j/8}
 \lesssim_u\eta_{\ast}^{-1}K,
\]
and $\|u_h\|_{L_t^{8/3}L_x^{\frac{4d}{2d-3}}}\lesssim_u1$ follows from
\eqref{longtime2} and conservation of mass.

For $\mathrm{Term}_2$, split
\[
 u_h=P_{\leq\eta_{\ast}^{-1}K/4}u_h
             +P_{\geq\eta_{\ast}^{-1}K/4}u_h.
\]
The first part has Fourier support in $|\xi|\leq\eta_{\ast}^{-1}K/2$.
Since $u_l$ has Fourier support in $|\xi|\leq\eta_{\ast}^{-1}K/16$,
the corresponding cubic products are supported in
$|\xi|\leq5\eta_{\ast}^{-1}K/8$. Both projections in their difference
are therefore the identity, so this contribution vanishes. Moreover,
$P_{\geq\eta_{\ast}^{-1}K/4}u_h=P_{\geq\eta_{\ast}^{-1}K/4}u$.
Thus, by H\"older's inequality and the boundedness of the smooth projections,
\begin{align*}
 \mathrm{Term}_2
 &\lesssim_u
 \big\||\nabla|^{2-d}\operatorname{Re}
       (u_l\overline{P_{\geq\eta_{\ast}^{-1}K/4}u})\big\|_
                                      {L_t^{4/3}L_x^{2d}}\\
 &\quad+
 \big\||\nabla|^{2-d}\operatorname{Re}
       (u_l\overline{P_{\leq\eta_{\ast}^{-1}K}
                          P_{\geq\eta_{\ast}^{-1}K/4}u})\big\|_
                                      {L_t^{4/3}L_x^{2d}}.
\end{align*}
The products inside $|\nabla|^{2-d}$ are supported in
$|\xi|\geq3\eta_{\ast}^{-1}K/16$. Consequently, Bernstein's inequality
and the Hardy--Littlewood--Sobolev inequality yield
\begin{align*}
 &\big\||\nabla|^{2-d}\operatorname{Re}
       (u_l\overline{P_{\geq\eta_{\ast}^{-1}K/4}u})\big\|_
                                      {L_t^{4/3}L_x^{2d}}\\
 &\lesssim_d(\eta_{\ast}^{-1}K)^{-1/2}
 \big\||\nabla|^{5/2-d}\operatorname{Re}
       (u_l\overline{P_{\geq\eta_{\ast}^{-1}K/4}u})\big\|_
                                      {L_t^{4/3}L_x^{2d}}\\
 &\lesssim_d(\eta_{\ast}^{-1}K)^{-1/2}
       \|u_l\|_{L_t^4L_x^{\frac{2d}{d-2}}}
       \|P_{\geq\eta_{\ast}^{-1}K/4}u\|_
                                      {L_t^2L_x^{\frac{2d}{d-2}}}.
\end{align*}
Here $d-5/2>0$ and
$\frac{d-2}{d}-\frac{d-5/2}{d}=\frac1{2d}$ for every $d\geq3$.
Using the admissible pair $(4,2d/(d-1))$, we obtain
\begin{align*}
 \|u_l\|_{L_t^4L_x^{\frac{2d}{d-2}}}
 &\lesssim_u(\eta_{\ast}^{-1}K)^{1/4}
       \sum_{\substack{j\in\mathbb Z\\2^j\leq\eta_{\ast}^{-1}K/4}}2^{j/4}
 \lesssim_u(\eta_{\ast}^{-1}K)^{1/2}.
\end{align*}
The other norm is bounded by \eqref{longtime2}. The term containing the
additional projection $P_{\leq\eta_{\ast}^{-1}K}$ is estimated in the same
way. Hence $\mathrm{Term}_2\lesssim_u1$, and we conclude that
\[
 \|\mathcal N_1\|_{L_t^{4/3}L_x^{\frac{2d}{d+1}}}\lesssim_u1.
\]

By triangle and H\"older's inequalities together with Bernstein's lemma,
\begin{align*}
    \|\mathcal{N}_{2}\|_{L_{t}^{4/3}L_{x}^{\frac{2d}{d+1}}}
&\lesssim \| |\nabla|^{2-d}(|u_{h}|^{2})\|_{L_{t}^{4/3}L_{x}^{2d}}\| u_{l}\|_{L_{t}^{\infty}L_{x}^{2}}
+ \| |\nabla|^{2-d}(|P_{\leq \eta_{\ast}^{-1}K}u_{h}|^{2})\|_{L_{t}^{4/3}L_{x}^{2d}}\| u_{l}\|_{L_{t}^{\infty}L_{x}^{2}} \\
&\quad +\| |\nabla|^{2-d}(\operatorname{Re}\{u_{h}\overline{u_{l}}\})\|_{L_{t}^{8/3}L_{x}^{\frac{4d}{5}}}\| u_{h}\|_{L_{t}^{8/3}L_{x}^{\frac{4d}{2d-3}}} \\
&\quad +\| |\nabla|^{2-d}(\operatorname{Re}\{(P_{\leq \eta_{\ast}^{-1}K}u_{h})\overline{u_{l}}\})\|_{L_{t}^{8/3}L_{x}^{\frac{4d}{5}}}\| u_{h}\|_{L_{t}^{8/3}L_{x}^{\frac{4d}{2d-3}}}. 
\end{align*}
By Hardy-Littlewood-Sobolev's inequality, mass conservation, \eqref{longtime2} and interpolation,
\[
\begin{aligned}
&\| u_{l}\|_{L^{\infty}L_{x}^{2}}\Big(\| |\nabla|^{2-d}(|u_{h}|^{2})\|_{L_{t}^{4/3}L_{x}^{2d}}
+ \| |\nabla|^{2-d}(|P_{\leq \eta_{\ast}^{-1}K}u_{h}|^{2})\|_{L_{t}^{4/3}L_{x}^{2d}}\Big) \\
\lesssim& M(u)^{1/2}\| u_{h}\|_{L_{t}^{8/3}L_{x}^{\frac{4d}{2d-3}}}^{2} 
\lesssim 1
\end{aligned}
\]
 and
\[
\begin{aligned}
&\| u_{h}\|_{L_{t}^{8/3}L_{x}^{\frac{4d}{2d-3}}}\Big(\| |\nabla|^{2-d}(\operatorname{Re}\{u_{h}\overline{u_{l}}\})\|_{L_{t}^{8/3}L_{x}^{4d/5}}
+ \| |\nabla|^{2-d}(\operatorname{Re}\{(P_{\leq \eta_{\ast}^{-1}K}u_{h})\overline{u_{l}}\})\|_{L_{t}^{8/3}L_{x}^{4d/5}}\Big) \\
\lesssim& \| u_{h}\|_{L_{t}^{8/3}L_{x}^{\frac{4d}{2d-3}}}^{2}\| u_{l}\|_{L_{t}^{\infty}L_{x}^{2}}
\lesssim 1.
\end{aligned}
\]
Therefore,
\[
\|\mathcal{N}_{2}\|_{L_{t}^{4/3}L_{x}^{\frac{2d}{d+1}}}\lesssim 1 .
\]
By the same line of reasoning used to obtain the estimate for $\|\mathcal{N}_{2}\|_{L_{t}^{4/3}L_{x}^{\frac{2d}{d+1}}}$, except now putting one of the factors $u_{h}$ in $L_{t}^{\infty}L_{x}^{2}$, we obtain
\[
\|\mathcal{N}_{3}\|_{L_{t}^{4/3}L_{x}^{\frac{2d}{d+1}}}\lesssim 1. 
\]
Combining the preceding estimates, we obtain
\[
 \|\mathcal N\|_{L_t^{4/3}L_x^{\frac{2d}{d+1}}}\lesssim_u1.
\]
By \eqref{zz1} and \eqref{zz12}, we conclude that
\begin{align*}
   \left| \Gamma_2\right| &\lesssim C_0(a)[C(\eta)^{\frac{d+2}{2(d+1)}}\eta_{\ast}^{-\frac{d}{2(d+1)}}K^{\frac{d}{2(d+1)}}+\eta^{\frac{1}{d+1}}\eta_{\ast}^{-1}K\\
    &+ C(\eta)^{3/4}\eta_{\ast}^{-1/4}K^{1/4}+\eta^{1/2}\eta_{\ast}^{-1}K].
\end{align*}

\vskip 0.1in

\textbf{Contribution of $\Gamma_3$:}
We use the decomposition $\mathcal N=\mathcal N_0+\cdots+\mathcal N_3$
from above and write
\[
 |\Gamma_3|\leq\sum_{j=0}^3\mathrm{Term}_{\mathcal N_j}.
\]
Since $\mathcal N_0=0$, it remains to consider $j=1,2,3$. For $\mathcal N_1$,
the commutativity of Fourier multipliers gives
\begin{align*}
 \mathcal N_1={}&
 [P_{\leq\eta_{\ast}^{-1}K},|\nabla|^{2-d}(|u_l|^2)]u_h\\
 &+2[P_{\leq\eta_{\ast}^{-1}K},u_l]
              |\nabla|^{2-d}\operatorname{Re}(u_l\overline{u_h})\\
 &+|\nabla|^{2-d}\big(
       [P_{\leq\eta_{\ast}^{-1}K},u_l]\overline{u_h}
       +[P_{\leq\eta_{\ast}^{-1}K},\overline{u_l}]u_h\big)u_l.
\end{align*}
For each $j=1,2,3$, since $\xi(t)$ is independent of $x$,
\[
 \operatorname{Re}\{\bar w(\nabla-i\xi(t))\mathcal N_j\}
 =\nabla\operatorname{Re}(\bar w\mathcal N_j)
 -\operatorname{Re}\{\overline{(\nabla-i\xi(t))w}\,\mathcal N_j\}.
\]
Integrating by parts in $x$, we obtain
\begin{align*}
 &\int_0^T\int_{\R^{2d}}a(t,x-y)\cdot|w(t,y)|^2
       \operatorname{Re}\{\bar w(\nabla-i\xi(t))\mathcal N_j\}(t,x)
                                  \,dx\,dy\,dt\\
 ={}&-\int_0^T\int_{\R^{2d}}(\nabla_x\cdot a)(t,x-y)|w(t,y)|^2
       \operatorname{Re}(\bar w\mathcal N_j)(t,x)\,dx\,dy\,dt\\
 &-\int_0^T\int_{\R^{2d}}a(t,x-y)\cdot|w(t,y)|^2
       \operatorname{Re}\{\overline{(\nabla-i\xi(t))w}\,\mathcal N_j\}(t,x)
                                  \,dx\,dy\,dt.
\end{align*}
By the estimates for $\mathcal N_j$ in the treatment of $\Gamma_2$ and
\eqref{zz12}, the absolute value of the second integral is bounded by
\[
 C_0(a)\big[C(\eta)^{3/4}(\eta_{\ast}^{-1}K)^{1/4}
                    +\eta^{1/2}\eta_{\ast}^{-1}K\big].
\]
For the first integral, \eqref{amapb} gives
$|\nabla_x\cdot a(t,x-y)|\lesssim_d C_0(a)|x-y|^{-1}$.
By weak Young's inequality and conservation of mass,
\[
 \|\tilde V*|w|^2\|_{L_t^\infty L_x^{d,\infty}}\lesssim_u1,
 \qquad \tilde V(x)=|x|^{-1}.
\]
Thus Lorentz H\"older's inequality yields
\begin{align*}
 &\left|\int_0^T\int_{\R^{2d}}(\nabla_x\cdot a)(t,x-y)|w(t,y)|^2
       \operatorname{Re}(\bar w\mathcal N_j)(t,x)\,dx\,dy\,dt\right|\\
 &\lesssim_u C_0(a)\|w\mathcal N_j\|_{L_t^1L_x^{\frac d{d-1},1}}.
\end{align*}

We first consider $j=1$. Since $V\in L^{d/2,\infty}$,
\[
 \|V*|u_l|^2\|_{L_x^\infty}
 \lesssim_d\||u_l|^2\|_{L_x^{\frac d{d-2},1}}
 \lesssim_d\|u_l\|_{L_x^{\frac{2d}{d-2},2}}^2.
\]
The same estimate holds for $u_h$. By Cauchy--Schwarz in the convolution
integral,
\[
 \|V*(u_l\overline{u_h})\|_{L_x^\infty}
 \lesssim_d\|u_l\|_{L_x^{\frac{2d}{d-2},2}}
             \|u_h\|_{L_x^{\frac{2d}{d-2},2}}.
\]
Applying these estimates to the original expression for $\mathcal N_1$,
and using the boundedness of the smooth projections, we find
\begin{align*}
 \|w\mathcal N_1\|_{L_t^1L_x^{\frac d{d-1},1}}
 &\lesssim_u\|\mathcal N_1\|_{L_t^1L_x^{\frac{2d}{d-2},2}}\\
 &\lesssim_u\|u_l\|_{L_t^4L_x^{\frac{2d}{d-2},2}}^2
             \|u_h\|_{L_t^2L_x^{\frac{2d}{d-2},2}}.
\end{align*}
For fixed $\eta$, take $K$ sufficiently large that
$C(\eta)\leq\eta_{\ast}^{-1}K$. Splitting $u_l$ at $C(\eta)N(t)$ and
using \eqref{longtime28}, interpolation and Bernstein's inequality, we get
\begin{align*}
 \|P_{\xi(t),\leq C(\eta)N(t)}u_l\|_
                       {L_t^4L_x^{\frac{2d}{d-2},2}}
 &\lesssim_u(\eta_{\ast}^{-1}K)^{1/4}
       \sum_{2^j\leq4C(\eta)}2^{j/4}\\
 &\lesssim_u C(\eta)^{1/4}(\eta_{\ast}^{-1}K)^{1/4},\\
 \|P_{\xi(t),\geq C(\eta)N(t)}u_l\|_
                       {L_t^3L_x^{\frac{6d}{3d-8},2}}
 &\lesssim_u(\eta_{\ast}^{-1}K)^{1/3}
       \sum_{2^j\leq\eta_{\ast}^{-1}K/4}2^{j/3}\\
 &\lesssim_u(\eta_{\ast}^{-1}K)^{2/3}.
\end{align*}
Here $(3,6d/(3d-4))$ is admissible, and the Bernstein factor in the second
estimate is $2^{2j/3}$. By interpolation and \eqref{localized},
\begin{align*}
 &\|P_{\xi(t),\geq C(\eta)N(t)}u_l\|_
                       {L_t^4L_x^{\frac{2d}{d-2},2}}\\
 &\lesssim_d\|P_{\xi(t),\geq C(\eta)N(t)}u_l\|_{L_t^\infty L_x^2}^{1/4}
       \|P_{\xi(t),\geq C(\eta)N(t)}u_l\|_
                       {L_t^3L_x^{\frac{6d}{3d-8},2}}^{3/4}\\
 &\lesssim_u\eta^{1/4}(\eta_{\ast}^{-1}K)^{1/2}.
\end{align*}
Consequently,
\[
 \|w\mathcal N_1\|_{L_t^1L_x^{\frac d{d-1},1}}
 \lesssim_u C(\eta)(\eta_{\ast}^{-1}K)^{1/2}
                   +\eta^{1/2}\eta_{\ast}^{-1}K.
\]

For $j=2,3$, we place two high-frequency factors in
$L_t^2L_x^{\frac{2d}{d-2},2}$ and the remaining factor in
$L_t^\infty L_x^2$. In particular,
\begin{align*}
 \|(V*|u_h|^2)u_l\|_{L_t^1L_x^2}
 &\lesssim_d\|u_h\|_{L_t^2L_x^{\frac{2d}{d-2},2}}^2
               \|u_l\|_{L_t^\infty L_x^2},\\
 \|(V*(u_h\overline{u_l}))u_h\|_{L_t^1L_x^2}
 &\lesssim_d\|V*(u_h\overline{u_l})\|_{L_t^2L_x^d}
               \|u_h\|_{L_t^2L_x^{\frac{2d}{d-2}}}\\
 &\lesssim_d\|u_l\|_{L_t^\infty L_x^2}
               \|u_h\|_{L_t^2L_x^{\frac{2d}{d-2}}}^2.
\end{align*}
The same estimates apply to the projected factors and to the terms with
three factors $u_h$. By \eqref{longtime2}, we obtain
\[
 \|\mathcal N_2\|_{L_t^1L_x^2}+\|\mathcal N_3\|_{L_t^1L_x^2}\lesssim_u1.
\]
On the other hand, Bernstein's inequality and \eqref{localized} give
\begin{align*}
 \|w\|_{L_t^\infty L_x^{\frac{2d}{d-2},2}}
 &\lesssim_u C(\eta)\|N\|_{L_t^\infty}
       +\eta_{\ast}^{-1}K
          \|P_{\xi(t),\geq C(\eta)N(t)}w\|_{L_t^\infty L_x^2}\\
 &\lesssim_u C(\eta)+\eta\eta_{\ast}^{-1}K.
\end{align*}
Therefore, for $j=2,3$,
\[
 \|w\mathcal N_j\|_{L_t^1L_x^{\frac d{d-1},1}}
 \lesssim_d\|w\|_{L_t^\infty L_x^{\frac{2d}{d-2},2}}
             \|\mathcal N_j\|_{L_t^1L_x^2}
 \lesssim_u C(\eta)+\eta\eta_{\ast}^{-1}K.
\]
Collecting these estimates and taking $0<\eta\leq1$ and
$\eta_{\ast}^{-1}K\geq1$, we conclude that
\[
 |\Gamma_3|\lesssim_u C_0(a)
 \big[C(\eta)(\eta_{\ast}^{-1}K)^{1/2}
                 +\eta^{\frac1{d+1}}\eta_{\ast}^{-1}K\big].
\]

Collecting our estimates for $\Gamma_1,\Gamma_2,\Gamma_3$, we have
\begin{align}
    \left|\Gamma_1+\Gamma_2+ \Gamma_3\right| &\lesssim C_0(a)\Bigl[C(\eta)^{\frac{d+2}{2(d+1)}}\eta_{\ast}^{-\frac{d}{2(d+1)}}K^{\frac{d}{2(d+1)}}+\eta^{\frac{1}{d+1}}\eta_{\ast}^{-1}K\notag\\
    &+ C(\eta)^{3/4}\eta_{\ast}^{-1/4}K^{1/4}+\eta^{1/2}\eta_{\ast}^{-1}K
    +C(\eta)^{\frac{d+4}{2(d+2)}}(\eta_{\ast}^{-1}K)^{\frac{d}{2(d+2)}}+\eta^{\frac{2}{d+2}} \eta_{\ast}^{-1}K\notag\\
    &+C(\eta)(\eta_{\ast}^{-1}K)^{1/2}+\eta^{\frac{1}{d+1}}\eta_{\ast}^{-1}K\Bigr].\label{finalmL}
\end{align}
For fixed $\eta_{\ast}$ and $\eta$, divide \eqref{finalmL} by $K$ and
let $K\to\infty$. Since every power of $K$ in the remaining terms is
strictly less than one, we obtain
\[
 \limsup_{K\to\infty}\frac{|\Gamma_1+\Gamma_2+\Gamma_3|}{K}
 \lesssim_u C_0(a)\eta_{\ast}^{-1}\eta^{\frac1{d+1}}.
\]
Letting $\eta\to0$ completes the proof.
\end{proof}


\section{No rapid frequency cascade case}\label{cascadesec}
In this section, we preclude the rapid frequency cascade scenario $\int_{0}^{\infty} N(t)^{3} d t<\infty$ by the additional regularity argument of \cite{Dodson-JAMS,Dodson-AJM,Dodson-Duke}. In particular, we can use the long-time Strichartz estimates to show that the hypothesis $\int_{0}^{\infty} N(t)^{3} d t=K<\infty$ implies that the solution $u \in C_{t}^{0} H_{x}^{3}\left([0, \infty) \times \R^d\right)$. From this additional regularity, we derive a contradiction to conservation of energy for $H^{1}$ solutions to the Hartree equation \eqref{NLH}. We therefore conclude that the rapid frequency cascade scenario does not occur.

\begin{lemma}[$H_{x}^{3}$ regularity]\label{addregu}
 Let $u$ be a solution to \eqref{NLH} as in Theorem \ref{reduction}. If $\int_{0}^{\infty} N(t)^{3}dt=K<\infty$, then $u \in C_{t}^{0} H_{x}^{3}\left([0, \infty) \times \R^d\right)$ and satisfies the estimate
\begin{equation*}
\sup _{t\geq0}\|u(t)\|_{\dot{H}_{x}^3\left(\R^d\right)} \lesssim K^3. 
\end{equation*}
\end{lemma}
\begin{proof}
The proof of \cite{Dodson-JAMS,Dodson-AJM,Dodson-Duke} carries over mutatis mutandis. We therefore omit the details.
\end{proof}
We now use Lemma \ref{addregu} to exclude the rapid frequency cascade case:
\begin{theorem}[No rapid frequency cascade]\label{rapidcascade}
    There does not exist a solution $u$ as in Theorem \ref{reduction}  such that $\int_{0}^{\infty} N(t)^{3} d t=K<\infty$.
\end{theorem}
 
\begin{proof}
We observe that by the fundamental theorem of calculus and the estimate \eqref{xxibound} for $\xi^{\prime}(t)$, it follows that $\xi_{\infty} := \lim\limits_{t \to \infty} \xi(t)$ exists in $\mathbb{R}^d$. Then, by applying a Galilean transformation to $u$, we can set $\xi_{\infty} = 0$. Note that the Galilean transformation also preserves the uniform boundedness in $\dot{H}_x^3$ norm and does not change the frequency scale $N(t)$.

Now the triangle inequality,
\begin{equation*}
\|u(t)\|_{\dot{H}_{x}^{1}\left(\R^d\right)} \leq\left\|P_{\xi(t), \leq C(\eta) N(t)} u(t)\right\|_{\dot{H}_{x}^{1}\left(\R^d\right)}+\left\|P_{\xi(t),\geq C(\eta) N(t)} u(t)\right\|_{\dot{H}_{x}^{1}\left(\R^d\right)}, 
\end{equation*}
for any $\eta \in(0,1)$. Interpolating between $L_{x}^{2}$ and $\dot{H}_{x}^{3}$ and using \eqref{localized}, we have that
\begin{equation*}
\left\|P_{\xi(t), \geq C(\eta) N(t)} u(t)\right\|_{\dot{H}_{x}^{1}\left(\R^d\right)} \lesssim \eta^{1 / 3} K, \quad \forall t \geq 0 
\end{equation*}
Now by using Plancherel's theorem and the translation/modulation symmetry of the Fourier transform (i.e. Galilean invariance of the mass), we have that
\begin{equation*}
\left\|P_{\xi(t), \leq C(\eta) N(t)} u(t)\right\|_{\dot{H}_{x}^{1}\left(\R^d\right)} \lesssim\left(\left|\xi(t)\right|+C(\eta) N(t)\right)\|u(t)\|_{L_{x}^{2}\left(\R^d\right)}=M(u)^{1 / 2}\left(\left|\xi(t)\right|+C(\eta) N(t)\right). 
\end{equation*}
Since $\left|\xi(t)\right|$ and $N(t)$ tend to zero as $t \rightarrow \infty$, we conclude that 
\begin{equation}
\limsup _{t \rightarrow \infty}\|u(t)\|_{\dot{H}_{x}^{1}\left(\R^d\right)} \leq \eta^{1 / 3} K, \quad \forall \eta \in(0,1) \Longrightarrow \limsup _{t \rightarrow \infty}\|u(t)\|_{\dot{H}_{x}^{1}\left(\R^d\right)}=0.
\end{equation}
Finally, by H\"older's inequality, the Hardy-Littlewood-Sobolev inequality and interpolation,
\begin{align*}
\left\|\left(|x|^{-2}\ast |u(t)|^{2}\right)|u(t)|^{2}\right\|_{L_{x}^{1}\left(\mathbb{R}^{d}\right)} \lesssim_{M(u)} \|u(t)\|_{\dot{H}_{x}^{1}\left(\R^d\right)}^{2}.
\end{align*}
Therefore, $E(u(t)) \rightarrow 0$ as $t \rightarrow \infty$, which, by conservation of energy, implies that $E\left(u_{0}\right)=0$. Since the energy functional is positive definite under our assumptions, we conclude that $u_{0} \equiv 0$, which is a contradiction. The proof of Theorem \ref{rapidcascade} is now complete. 
\end{proof}

\section{No quasi-soliton case}\label{quasisec}

In this section, we preclude the quasi-soliton scenario $\int_{0}^{\infty} N(t)^{3} d t=\infty$.

\begin{theorem}\label{quasi}
    There does not exist a solution $u$ as in Theorem \ref{reduction}  such that $\int_{0}^{\infty} N(t)^{3} d t= \infty$.
\end{theorem}

\begin{remark}
    We remark that in the defocusing case, Theorem \ref{quasi} admits an alternative proof, as an interaction Morawetz estimate analogous to that for the NLS is available; hence, this theorem can be established by following a similar line of argument as in \cite{Dodson-JAMS, Dodson-AJM, Dodson-Duke}. For a unified presentation, however, we have chosen to treat both the focusing and defocusing cases simultaneously.
\end{remark}

We begin with some preliminary spatial and frequency localization estimates for almost-periodic solutions to \eqref{NLH}. Throughout this section, we use the notation
$o_A(1)$, where A denotes a positive real parameter, to denote a quantity satisfying
\[
\lim_{A\to\infty}o_A(1)=0.
\]
\begin{lemma}\label{quasisolitonlemma1}
Let $u$ be the solution as in Theorem \ref{reduction}. Then
\begin{equation}\label{truncateu}
 \sup_{t\geq0}\int_{|x-x(t)|\geq R/N(t)}|w(t,x)|^2\,dx=o_R(1),
\end{equation}
where $w:=P_{\leq K_0}u$, uniformly in $K_0\geq1$. Consequently, if
$[0,T]$ is a union of complete small intervals $J_l$, then
\begin{align}\label{truncateu2}
 &\int_0^T\int_{|x-x(t)|\geq R/N(t)}
       \tilde N(t)|w(t,x)|^2(V*|w|^2)(t,x)\,dx\,dt\notag\\
 &\qquad\lesssim_u o_R(1)\int_0^T\tilde N(t)N(t)^2\,dt,
\end{align}
uniformly in $T$ and $K_0\geq1$, where $0\leq\tilde N(t)\leq N(t)$
satisfies $\sup_{t\in J_l}\tilde N(t)\sim\inf_{t\in J_l}\tilde N(t)$ on
each small interval. For arbitrary $T>0$, the right-hand side of
\eqref{truncateu2} is replaced by $o_R(1)(1+\int_0^T\tilde N(t)N(t)^2\,dt)$.
\end{lemma}
\begin{proof}
The kernel of $P_{\leq K_0}$ is $K_0^d\varphi^\vee(K_0z)$. By
Cauchy--Schwarz and the $L^1$ bound for this kernel,
\begin{align*}
 &\int_{|x-x(t)|\geq R/N(t)}|w(t,x)|^2\,dx\\
 &\lesssim_d
 \int_{|z|\leq R/(2N(t))}K_0^d|\varphi^\vee(K_0z)|
       \int_{|x-x(t)|\geq R/N(t)}|u(t,x-z)|^2\,dx\,dz\\
 &\quad+
 \int_{|z|>R/(2N(t))}K_0^d|\varphi^\vee(K_0z)|
       \int_{|x-x(t)|\geq R/N(t)}|u(t,x-z)|^2\,dx\,dz\\
 &=:\operatorname{Term}_1+\operatorname{Term}_2.
\end{align*}
If $|z|\leq R/(2N(t))$ and $|x-x(t)|\geq R/N(t)$, then
$|x-z-x(t)|\geq R/(2N(t))$. Thus \eqref{localized} gives
$\operatorname{Term}_1\lesssim_u o_R(1)$. For the second term, conservation
of mass and the decay of $\varphi^\vee$ yield
\begin{align*}
 \operatorname{Term}_2
 &\lesssim_u\int_{|z|>R/(2N(t))}K_0^d|\varphi^\vee(K_0z)|\,dz\\
 &\lesssim_u\int_{|y|>K_0R/(2N(t))}\langle y\rangle^{-d-1}\,dy
 \lesssim_u\frac{N(t)}{K_0R}\lesssim_u R^{-1},
\end{align*}
since $K_0\geq1$ and $N(t)\leq1$. This proves \eqref{truncateu}.

We next estimate the potential energy on a small interval $J_l$.
By H\"older's inequality and the Hardy--Littlewood--Sobolev inequality,
\begin{align*}
 &\int_{J_l}\int_{|x-x(t)|\geq R/N(t)}
       |w(t,x)|^2(V*|w|^2)(t,x)\,dx\,dt\\
 &\lesssim_d
 \|1_{\{|x-x(t)|\geq R/N(t)\}}w\|_
                   {L_t^4L_x^{\frac{2d}{d-1}}(J_l\times\R^d)}^2
 \|w\|_{L_t^4L_x^{\frac{2d}{d-1}}(J_l\times\R^d)}^2.
\end{align*}
The local Strichartz estimates and the boundedness of $P_{\leq K_0}$ give
\[
 \|w\|_{L_t^4L_x^{\frac{2d}{d-1}}(J_l\times\R^d)}
 +\|w\|_{L_t^2L_x^{\frac{2d}{d-2}}(J_l\times\R^d)}\lesssim_u1.
\]
Interpolating with \eqref{truncateu}, we obtain
\begin{align*}
 &\|1_{\{|x-x(t)|\geq R/N(t)\}}w\|_
                   {L_t^4L_x^{\frac{2d}{d-1}}(J_l\times\R^d)}^2\\
 &\leq
 \|1_{\{|x-x(t)|\geq R/N(t)\}}w\|_{L_t^\infty L_x^2(J_l\times\R^d)}
 \|w\|_{L_t^2L_x^{\frac{2d}{d-2}}(J_l\times\R^d)}
 \lesssim_u o_R(1).
\end{align*}
Hence the potential energy integral on each small interval is $o_R(1)$,
uniformly in $K_0$. If $J_l$ is complete, Proposition \ref{zzzz2} and the
local comparability of $\tilde N$ imply
\[
 \sup_{t\in J_l}\tilde N(t)
 \lesssim_u\int_{J_l}\tilde N(t)N(t)^2\,dt.
\]
Multiplying the local estimate by $\sup_{J_l}\tilde N$ and summing over
the complete intervals proves \eqref{truncateu2}. For an incomplete final
interval, the same local estimate gives an additional $o_R(1)$, since
$\tilde N(t)\leq N(t)\leq1$.
\end{proof}

\begin{lemma}\label{quasisolitonlemma2}
Let $u$ be the solution as in Theorem \ref{reduction} that satisfies $\int_{0}^{\infty} N(t)^{3} d t= \infty$, then
    \begin{equation}\label{vvvc}
        \int_0^T\int_{|x-x(t)|\leq\frac{R}{N(t)}}\tilde{N}(t) |P_{\leq \eta_{\ast}^{-1}K}u(t,x)|^{\frac{2(d+2)}{d}} dx dt\gtrsim \int_0^T\tilde{N}(t)N(t)^2dt
    \end{equation}
    if $K=\int_0^TN(t)^3dt\gg1$, $R\gg1$ and $\eta_{\ast}\ll 1$, where  $\tilde{N}(t)\leq N(t)$ satisfies $\max_{t\in J_l}\tilde{N}(t)\sim \min_{t\in J_l}\tilde{N}(t)$ on any small interval $J_l\subset [0,T]$.
 \end{lemma} 
    
    \begin{proof}
        By choosing $\eta_{\ast}$ sufficiently small, for any small $\eta>0$, we see that 
        \[
       |\xi|\geq \eta_{\ast}^{-1}K\Longrightarrow |\xi-\xi(t)|\geq C(\eta)N(t),\quad \forall t\in [0,T],
        \]
        if $K\geq K_0(\eta)$. Therefore, \eqref{vvvc} comes from \eqref{localized}, and the fact that  
        \[
        \int_{J_l} N(t)^2 dt\lesssim\int_{J_l}\int_{\R^d} |u(t,x)|^{\frac{2(d+2)}{d}}dxdt\lesssim  1+\int_{J_l} N(t)^2 dt.
        \]
    \end{proof}

We are now in the position to prove Theorem \ref{quasi}. To highlight the main ideas of the proof and the technical route, the structure of our proof parallels that in \cite{Dodson-Adv}: we first demonstrate it in the simplified case where $N(t)\equiv1$, and then extend it to the general case.
\subsection{$N(t)\equiv1$}
We take $[0,T]$ to be a union of complete small intervals; then $K=T$.
Let  $\chi: \mathbb{R} \rightarrow[0, \infty)$ be an even bump function satisfying $\|\chi\|_{L^{1}(\mathbb{R})}=1$ and $\operatorname{supp}(\chi) \subset(-1,1)$. Given a constant $\eta_2\ll1$, we define a function $\zeta_{R}: \mathbb{R} \rightarrow[0,1]$ by the formula
\begin{equation}\label{equ:zetaRdef}
\zeta_{R}(r):=\frac{1}{\eta_{2} R} \int_{\mathbb{R}} 1_{\left[-\left(1+2 \eta_{2}\right) R,\left(1+2 \eta_{2}\right) R\right]}(r-s) \chi\left(\frac{s}{\eta_{2} R}\right) ds, 
\end{equation}
where $1_A(r)$ denotes the characteristic function of the set $A.$
 Observe that $\zeta_{R}$ is nonnegative, decreasing, and satisfies
\[
\zeta_{R}(r)=\left\{\begin{array}{ll}
1, & |r| \leq R \\
0, & |r| \geq\left(1+4 \eta_{2}\right) R.
\end{array} \right.
\]
Moreover, by Young's inequality, $\zeta_{R}$ satisfies the derivative estimates
\begin{equation}\label{mor1}
\left|\zeta_{R}^{(n)}(r)\right| \lesssim_n \frac{1}{\left(\eta_{2} R\right)^{n}}, \quad \forall n \in \mathbb{N}_{0} 
\end{equation}

Next, define a compactly supported function $\varphi_{R}: \mathbb{R} \rightarrow[0, \infty)$ by the formula
\begin{equation}\label{equ:varphiRdef}
\varphi_{R}(r):=\frac{1}{\left\|\zeta_{R}(|\cdot|)\right\|_{L^{1}\left(\mathbb{R}^{d}\right)}} \int_{\mathbb{S}^{d-1}} \int_{\mathbb{R}^{d}} \zeta_{R}(|r \omega-z|) \zeta_{R}(|z|) d z d \omega, 
\end{equation}
where $d \omega$ denotes the unit normalized surface measure on the sphere $\mathbb{S}^{d-1}$. Observe that by rotation invariance of the measure and the isometry property of rotations, for any $x, y \in \mathbb{R}^{d}$, we have that
\begin{align*}
\varphi_{R}(|x-y|) & =\frac{1}{\left\|\zeta_{R}(|\cdot|)\right\|_{L^{1}}} \int_{\mathbb{S}^{d-1}} \int_{\mathbb{R}^{d}} \zeta_{R}\big(\big||x-y| \omega-z\big|\big) \zeta_{R}(|z|) d z d \omega \\
& =\frac{1}{\left\|\zeta_{R}(|\cdot|)\right\|_{L^{1}}} \int_{\mathbb{S}^{d-1}} \int_{\mathbb{R}^{d}} \zeta_{R}\left(\left|\mathcal{O}_{\omega}(x-y)-z\right|\right) \zeta_{R}(|z|) d z d \omega \\
& =\frac{1}{\left\|\zeta_{R}(|\cdot|)\right\|_{L^{1}}} \int_{\mathbb{S}^{d-1}} \int_{\mathbb{R}^{d}} \zeta_{R}\left(\left|x-y-\mathcal{O}_{\omega}^{*} z\right|\right) \zeta_{R}\left(\left|\mathcal{O}_{\omega}^{*} z\right|\right) d z d \omega \\
& =\frac{1}{\left\|\zeta_{R}(|\cdot|)\right\|_{L^{1}}} \int_{\mathbb{R}^{d}} \zeta_{R}(|x-y-z|) \zeta_{R}(|z|) d z .
\end{align*}

Moreover, by translation invariance of the Lebesgue measure,
\begin{equation*}
\varphi_{R}(|x-y|)=\left\|\zeta_{R}(|\cdot|)\right\|_{L^{1}}^{-1} \int_{\mathbb{R}^{d}} \zeta_{R}(|x-z|) \zeta_{R}(|y-z|) d z. 
\end{equation*}

Additionally, $\varphi_{R}$ is decreasing and $\varphi_{R}$ satisfies the derivative estimates
\begin{equation*}
\operatorname{supp}\left(\varphi_{R}\right) \subset\left[-\left(2+8 \eta_{2}\right) R,\left(2+8 \eta_{2}\right) R\right] \quad \text { and } \quad\left|\varphi_{R}^{(n)}(r)\right| \lesssim_{n} \frac{1}{\left(\eta_{2} R\right)^{n}}, \forall n \in \mathbb{N} . 
\end{equation*}

We define a ``smoothed out" indicator function $\tilde{1}_{[0,1]}$ of the unit interval to be a bump function which satisfies $0 \leq \tilde{1}_{[0,1]} \leq 1$ and
\begin{equation}\label{def1}
    \tilde{1}_{[0,1]}(r)=\left\{\begin{array}{ll}
1, & r \leq 1 \\
0, & r \geq 5 / 4
\end{array} .\right.
\end{equation}

For general $a \in(0, \infty)$, we define $\tilde{1}_{[0, a]}:=\tilde{1}_{[0,1]}(\cdot / a)$ and $\tilde{1}_{(a, b]}:=\tilde{1}_{[0, b]}-\tilde{1}_{[0, a]}$, for $0<a<b<\infty$. We define $\tilde{1}_{(a, \infty)}:=1-\tilde{1}_{[0, a]}$. In particular,
\begin{gather}\label{mor2}
    \left|\tilde{1}_{[0, a]}^{(n)}(r)\right| \lesssim_{n} a^{-n} 1_{[a, 5 a / 4]}(r), \quad\left|\tilde{1}_{(a, b]}^{(n)}(r)\right| \lesssim_{n} \max \left\{a^{-n}, b^{-n}\right\} 1_{[a, 5 b / 4]}(r), \notag\\
    \left|\tilde{1}_{(a, \infty)}^{(n)}(r)\right| \lesssim_{n} a^{-n} 1_{[a, 5 a / 4]}(r).
\end{gather}

Here $n\geq1$, and the cutoff functions take values in $[0,1]$.

For later use in the estimation of error terms, we decompose the function $\varphi_{R}$ into a sum of three pieces $\varphi_{R, 1}$, $\varphi_{R, 2}$, and $\varphi_{R, 3}$, respectively defined by the formulae
\begin{align}
& \varphi_{R, 1}(r):=\left\|\zeta_{R}(|\cdot|)\right\|_{L^{1}}^{-1}\left(\int_{\mathbb{S}^{d-1}} \int_{|z| \leq\left(1-4 \eta_{2}\right) R} \zeta_{R}(|r \omega-z|) \zeta_{R}(|z|) d z d \omega\right) \tilde{1}_{\left[0, \eta_{2} R\right]}(r), \label{varphi1} \\
& \varphi_{R, 2}(r):=\left\|\zeta_{R}(|\cdot|)\right\|_{L^{1}}^{-1}\left(\int_{\mathbb{S}^{d-1}} \int_{|z|>\left(1-4 \eta_{2}\right) R} \zeta_{R}(|r \omega-z|) \zeta_{R}(|z|) d z d \omega\right) \tilde{1}_{\left[0, \eta_{2} R\right]}(r), \label{varphi2}\\
& \varphi_{R, 3}(r):=\left\|\zeta_{R}(|\cdot|)\right\|_{L^{1}}^{-1}\left(\int_{\mathbb{S}^{d-1}} \int_{\R^d} \zeta_{R}(|r \omega-z|) \zeta_{R}(|z|) d z d \omega\right) \tilde{1}_{\left( \eta_{2} R, \infty\right)}(r) .\label{varphi3}
\end{align}
We record some properties of the $\varphi_{R, j}$ functions in the next lemma.
\begin{lemma}[$\varphi_{R, j}$ properties]\label{lemmavarphi}
The functions $\varphi_{R, 1}, \varphi_{R, 2}$, and $\varphi_{R, 3}$ satisfy the following properties:
\begin{itemize}
\item[$(i)$] There exists a constant $C_{\eta_{2}} \sim 1$ such that
$$\varphi_{R, 1}(r)=C_{\eta_{2}} \tilde{1}_{\left[0, \eta_{2} R\right]}(r),$$
and for any $n \in \mathbb{N}$,
$$\left|\varphi_{R, 1}^{(n)}(r)\right| \lesssim_n \frac{1}{\left(\eta_{2} R\right)^{n}} 1_{\left[\eta_{2} R / 2,2 \eta_{2} R\right]}(r).$$

\item[$(ii)$] We have that
\begin{equation}\label{varphi22}
 \varphi_{R, 2}(r)=\left\|\zeta_{R}(|\cdot|)\right\|_{L^{1}}^{-1}\left(\int_{\mathbb{S}^{d-1}} \int_{\left(1-4 \eta_{2}\right) R<|z| \leq\left(1+5 \eta_{2}\right) R} \zeta_{R}(|r \omega-z|) \zeta_{R}(|z|) d z d \omega\right) \tilde{1}_{\left[0, \eta_{2} R\right]}(r)    
\end{equation}
and for every $n \in \mathbb{N}_{0}$.
\begin{equation}\label{varphi2est}
\left|\varphi_{R,2}^{(n)}(r)\right| \lesssim_n \frac{\eta_{2}}{\left(\eta_{2} R\right)^{n}} 1_{\left[0,5 \eta_{2} R / 4\right]}(r).
\end{equation}

\item[$(iii)$] We have that
\begin{equation}\label{varphi32}
\quad \varphi_{R, 3}(r)=\left\|\zeta_{R}(|\cdot|)\right\|_{L^{1}}^{-1}\left(\int_{\mathbb{S}^{d-1}} \int_{|z| \leq\left(1+4 \eta_{2}\right) R} \zeta_{R}(|r \omega-z|) \zeta_{R}(|z|) d z d \omega\right) \tilde{1}_{\left(\eta_{2} R,\infty\right)}(r),
\end{equation}
and for every $n \in \mathbb{N}_{0}$,
\begin{equation}\label{varphi3est}
\left|\varphi_{R, 3}^{(n)}(r)\right| \lesssim_{n}\left(\eta_{2} R\right)^{-n} 1_{\left[\eta_{2} R,\left(2+8 \eta_{2}\right) R\right]}(r).
\end{equation}
\end{itemize}
\end{lemma}

\begin{proof}

(i) Since $\zeta_{R} \equiv 1$ on the interval $[0, R]$ and $\operatorname{supp}\left(\tilde{1}_{[0, \eta_{2} R]}\right) \subset[0,5 \eta_{2} R / 4]$, we see from the triangle inequality that
\begin{equation*}
|r \omega-z| \leq|r|+|z| \leq(1-2 \eta_{2}) R, \quad \forall|r| \leq 2 \eta_{2} R, \quad|z| \leq (1-4 \eta_{2}) R.
\end{equation*}
Hence the integrand
\begin{equation*}
\zeta_{R}(|r \omega-z|) \zeta_{R}(|z|) 
\end{equation*}
is identically one when both $|r| \leq 2 \eta_{2} R$ and $|z| \leq(1-4 \eta_{2}) R$. Therefore, for $r\leq\eta_2R$,
\begin{equation*}
\varphi_{R, 1}(r)=\|\zeta_{R}(|\cdot|)\|_{L^{1}}^{-1} \int_{\mathbb{S}^{d-1}} \int_{|z| \leq(1-4 \eta_{2}) R} d z d \omega=\frac{C_{1}((1-4 \eta_{2}) R)^{d}}{\|\zeta_{R}(|\cdot|)\|_{L^{1}}}=\frac{C_{1}(1-4 \eta_{2})^{d}}{C_{2}},
\end{equation*}
where we obtain the ultimate equality by using that $\|\zeta_{R}(|\cdot|)\|_{L^{1}(\mathbb{R}^{d})}=C_{2} R^{d}$, for some $C_{2}>0$, by scaling invariance of Lebesgue measure. The derivative estimates for $\varphi_{R, 1}$ follow from the Leibnitz rule and the bounds \eqref{mor1} and \eqref{mor2}.

(ii) If $|z|>(1+5 \eta_{2}) R$, then $\zeta_{R}(|z|)=0$, which implies that the integrand in \eqref{varphi2} is zero, yielding the desired identity \eqref{varphi22}. To obtain the derivative bounds, set $\tilde{\zeta}_{R}(x):=\zeta_{R}(|x|)$. Then by the chain rule,
\[
|\nabla^{n} \tilde{\zeta}_{R}(x)| \lesssim_n(\eta_{2} R)^{-n} 1_{A(0, R,(1+4 \eta_{2}) R)}(x), \quad \forall x \in \mathbb{R}^{d}, n \in \mathbb{N}.
\]
So by the Leibnitz rule and \eqref{mor2}, we see that
\begin{align*}
&|\varphi_{R, 2}^{(n)}(r)| \\
\lesssim& \sum_{\alpha+\beta=n} \frac{C_{\alpha \beta}| \tilde{1}_{[0, \eta_{2} R]}^{(\beta)}(r)|}{\|\zeta_{R}(|\cdot|)\|_{L^{1}(\mathbb{R}^{d})}}\left(\int_{\mathbb{S}^{d-1}} \int_{(1-4 \eta_{2}) R \leq|z| \leq(1+5 \eta_{2}) R}|\nabla^{\alpha} \tilde{z}_{R}(r \omega-z)||\zeta_{R}(|z|)| d z d \omega\right) \\
 \lesssim&_n \|\zeta_{R}(|\cdot|)\|_{L^{1}(\mathbb{R}^{d})}^{-1}(\eta_{2} R)^{-n}\left(\int_{(1-4 \eta_{2}) R \leq|z| \leq(1+5 \eta_{2}) R} d z\right) 1_{[0,5 \eta_{2}R / 4]}(r) \\
 \lesssim& \frac{\eta_{2} R^{d}}{(\eta_{2} R)^{n}\|\zeta_{R}(|\cdot|)\|_{L^{1}(\mathbb{R}^{d})}} 1_{[0,5 \eta_{2}R / 4]}(r) \\
\lesssim& \frac{\eta_{2}}{(\eta_{2} R)^{n}} 1_{[0,5 \eta_{2}R / 4]}(r)
\end{align*}
where the ultimate equality follows from $\|\zeta_{R}(|\cdot|)\|_{L^{1}} \sim R^{d}$.

 (iii) If $|z| \geq(1+4 \eta_{2}) R$, then $\zeta_{R}(z)=0$, implying the integrand in \eqref{varphi3} is zero. Now observe from the reverse triangle inequality that
\begin{equation*}
|r \omega-z| \geq(2+8 \eta_{2}) R-(1+4 \eta_{2}) R=(1+4 \eta_{2}) R, \quad \forall|z| \leq(1+4 \eta_{2}) R,|r| \geq(2+8 \eta_{2}) R,
\end{equation*}
which implies that $\zeta_{R}(|r \omega-z|)=0$ and therefore the integrand in \eqref{varphi3} is zero. The identity \eqref{varphi32} follows immediately. The derivative estimates follow by the same argument as in the proof of assertion (i).
\end{proof}

Next, we define a  spatially averaged version of the function $\varphi$. More precisely, define the function $\psi_{R}:[0, \infty) \rightarrow[0, \infty)$ by the formula
\begin{equation}\label{equ:psiRdef}
\psi_{R}( r):=\frac{1}{r} \int_{0}^{r} \varphi_{R}\left( s\right) d s.
\end{equation}

\begin{remark}We record the observation that
\begin{equation}\label{equ:idenpsi}
r\left(\partial_{r} \psi_{R}\right)(r)+\psi_{R}(r)=\frac{d}{d r}\left(r \psi_{R}(r)\right)=\varphi_{R}\left(r\right) 
\end{equation}
\end{remark}
As with $\varphi$, we decompose $\psi_{R}$ into three terms:
\begin{align}
 & \psi_{R, 1}(r):=\frac{1}{r} \int_{0}^{r} \varphi_{R, 1}\left( s\right) d s \label{psi1}\\
&\psi_{R,  2}(r) :=\frac{1}{r} \int_{0}^{r} \varphi_{R, 2}\left( s\right) d s\label{psi2}\\
& \psi_{R,  3}(r)
:=\frac{1}{r} \int_{0}^{r} \varphi_{R, 3}\left( s\right) d s.\label{psi3}  
\end{align}
We record some properties of the function $\psi_{R, j}$ with the next lemma.
\begin{lemma}[$\psi_{R,  j}$ properties]\label{lemmapsi}
The functions $\psi_{R, 1}, \psi_{R, 2}$, and $\psi_{R, 3}$ satisfy the following properties:
   \begin{itemize}
    \item[$(i)$] We have that
    \begin{equation}\label{psi12}
        \psi_{R,  1}(r)= \begin{cases}C_{\eta_{2}}, & r \leq \eta_{2}R\\ \frac{C_{\eta_{2}} }{r } \int_{0}^{\frac{5 \eta_{2} R}{4}} \tilde{1}_{[0,\eta_2R]}(s) d s, & r \geq \frac{5 \eta_{2} R}{4}\end{cases}
    \end{equation}
In particular, for every $n \in \mathbb{N}, \operatorname{supp}\left(\psi_{R,  1}^{(n)}\right) \subset\left[\eta_{2} R, \infty\right)$ and
    \begin{equation}\label{psi1est}
        \left|\psi_{R,1}^{(n)}(r)\right| \lesssim_{n} \frac{\eta_{2} R}{r^{n+1}} 1_{\left[\eta_{2} R, \infty\right)}(r).
    \end{equation}
 \item[$(ii)$] We have that
 \begin{equation}\label{psi22}
    \psi_{R, 2}(r)=\frac{1}{r} \int_{0}^{5 \eta_{2} R / 4} \varphi_{R, 2}(s) d s, \quad \forall r \geq \frac{5 \eta_{2} R}{4 }, 
 \end{equation}
and for every $n \in \mathbb{N}_{0}$,
\begin{equation}\label{psi2est}
\left|\psi_{R,2}^{(n)}(r)\right| \lesssim_{n} \eta_{2} \min \left\{\left(\frac{1}{\eta_{2} R}\right)^{n}, \frac{\eta_{2} R}{ r^{n+1}}\right\}. 
\end{equation}
\item[$(iii)$] We have that
\begin{equation}\label{psi32}
    \psi_{R, 3}(r)= \begin{cases}0, & r \leq \eta_{2} R \\ \frac{1}{r} \int_{0}^{\left(2+8 \eta_{2}\right) R} \varphi_{R, 3}(s) d s, & r \geq \left(2+8 \eta_{2}\right) R\end{cases}
\end{equation}
and for every $n \in \mathbb{N}_{0}$,
\begin{equation}\label{psi3est}
\left|\psi_{R,3}^{(n)}(r)\right| \lesssim_{n,d}
\begin{cases}
0, & 0\leq r\leq\eta_2R,\\
(\eta_2R)^{-n}, & \eta_2R<r\leq(2+8\eta_2)R,\\
Rr^{-n-1}, & r>(2+8\eta_2)R.
\end{cases}
\end{equation}
   \end{itemize} 
\end{lemma}

\begin{proof}
(i) We know from Lemma \ref{lemmavarphi}(i) that $\varphi_{R, 1}(r)=C_{\eta_{2}}$ for $r \leq \eta_{2} R$. 
\begin{equation*}
\psi_{R,  1}(r)=\frac{1}{r} \int_{0}^{r } C_{\eta_{2}} d s=C_{\eta_{2}}, \quad r \leq \eta_{2}R.
\end{equation*}
Similarly, we know that $\operatorname{supp}\left(\varphi_{R, 1}\right) \subset\left[0,5 \eta_{2} R / 4\right]$,
the second assertion of \eqref{psi12} follows. For the derivative estimate \eqref{psi1est}, the support of $\psi_{R, 1}^{(n)}$ means we only need to consider $r \geq \eta_{2} R$. If
$$\eta_{2} R <r<\frac{5 \eta_{2} R}{4},$$
then by the Leibnitz rule and the fundamental theorem of calculus,
\begin{align*}
\left| \psi_{R,  1}^{(n)}(r)\right| & \leq \sum_{\alpha+\beta=n} C_{\alpha \beta} \frac{1}{r^{\alpha+1}}\left|\partial_{r}^{\beta}\left(\int_{0}^{(\cdot) } \varphi_{R, 1}(s) d s\right)(r)\right| \\
& \lesssim_{n} \left(\frac{r }{r^{n+1}}+\sum_{\alpha=0}^{n-1} \frac{1}{r^{\alpha+1}\left(\eta_{2} R\right)^{\beta-1}}\right) \\
& \lesssim_{n} \frac{\eta_{2} R}{r^{n+1} }, 
\end{align*}
where we use Lemma \ref{lemmavarphi}(i) to obtain the penultimate inequality and $r \sim \eta_{2} R$ to obtain the ultimate inequality. If $r \geq 5 \eta_{2} R/4$, then proceeding similarly now using \eqref{psi12}, we obtain the same estimate. 

(ii) Since $\operatorname{supp}\left(\varphi_{R, 2}\right) \subset\left[0,5 \eta_{2} R / 4\right]$, the identity \eqref{psi22} follows by the same reasoning as in the preceding case. Now using Lemma \ref{lemmavarphi}(ii) we see that for every $n \in \mathbb{N}_{0}$,
\begin{equation*}
\left|\psi_{R, 2}^{(n)}(r)\right| \lesssim_{n} \frac{1}{r^{n+1}} \int_{0}^{5 \eta_{2} R / 4} \varphi_{R, 2}(s) d s \lesssim_n \frac{\eta_{2}^{2} R}{r^{n+1}}, \quad \forall r \geq \frac{5 \eta_{2} R}{4 } .
\end{equation*}
Now suppose $r \leq 5 \eta_{2} R/4$. Then for $N \gg n$, we can use Taylor's theorem to write
\begin{equation*}
\varphi_{R, 2}(s)=\sum_{\alpha=0}^{N} \frac{\varphi_{R, 2}^{(\alpha)}(0)}{\alpha!} s^{\alpha}+\int_{0}^{s} \frac{\varphi_{R, 2}^{(N+1)}(\tau)}{N!}(s-\tau)^{N} d \tau, \quad \forall 0 \leq s \leq r \leq \frac{5 \eta_{2} R}{4} . 
\end{equation*}
Now by the definition of $\psi_{R, 2}$, we see that
\begin{equation*}
\psi_{R,  2}(r)=\sum_{\alpha=0}^{N} \frac{\varphi_{R, 2}^{(\alpha)}(0)}{(\alpha+1)!}\left(r\right)^{\alpha}+\frac{1}{r} \int_{0}^{r}\left(\int_{0}^{s} \frac{\varphi_{R, 2}^{(N+1)}(\tau)}{N!}(s-\tau)^{N} d \tau\right) d s .
\end{equation*}
By Lemma \ref{lemmavarphi}(ii),
\begin{align*}
\left|\sum_{\alpha=n}^{N} \frac{\varphi_{R, 2}^{(\alpha)}(0)}{(\alpha+1-n)!} r^{\alpha-n}\right| & \lesssim_{n, N} \sum_{\alpha=n}^{N} \frac{\eta_{2}}{\left(\eta_{2} R\right)^{\alpha}} r^{\alpha-n} \\
& \lesssim_{n} \sum_{\alpha=n}^{N} \frac{\eta_{2}}{\left(\eta_{2} R\right)^{\alpha}}\left(\eta_2R\right)^{\alpha-n} \\
& \lesssim_{n, N} \eta_{2}\left(\eta_{2} R\right)^{-n} .
\end{align*}
For the remainder term, set
\begin{equation*}
\rho_{R, N}(s):=\int_{0}^{s} \frac{\varphi_{R, 2}^{(N+1)}(\tau)}{N!}(s-\tau)^{N} d \tau.
\end{equation*}
It is straightforward to check from the chain rule and Lemma \ref{lemmavarphi}(ii) that
\begin{equation}\label{phoest}
\left|\rho_{R, N}^{(j)}(s)\right| \lesssim_{j, N} \frac{\eta_{2}}{\left(\eta_{2} R\right)^{N+1}} s^{N+1-j}, \quad \forall 0 \leq j \ll N.
\end{equation}
Hence by the chain rule and fundamental theorem of calculus,
\begin{align}
    \left|\left(\frac{1}{(\cdot)} \int_{0}^{(\cdot)} \rho_{R, N}(s) d s\right)^{(n)}(r)\right| &\lesssim_n \frac{1}{r^{n+1}}\left|\int_{0}^{r} \rho_{R, N}(s) d s\right|\notag\\
    &+\sum_{\alpha=0}^{n-1} \frac{1}{r^{\alpha+1}}\left|\rho_{R, N}^{(n-1-\alpha)}\left(r\right)\right|. 
\end{align}
Using \eqref{phoest} and recalling that $r \leq 5\eta_{2}R/4$, we see that
\begin{align*}
\frac{1}{r^{n+1}}\left|\int_{0}^{r} \rho_{R, N}(s) d s\right| & \lesssim_{N} \frac{1}{r^{n+1}} \frac{\eta_{2}}{\left(\eta_{2} R\right)^{N+1}}\left(r\right)^{N+2} \\
& \lesssim_{n, N} \frac{\eta_{2}}{\left(\eta_{2} R\right)^{N+1}}\left(\eta_{2} R\right)^{N+1-n} \\
& =\eta_2\left(\eta_{2} R\right)^{-n} 
\end{align*}
and similarly,
\begin{align*}
& \sum_{\alpha=0}^{n-1} \frac{1}{r^{\alpha+1}}\left|\rho_{R, N}^{(n-1-\alpha)}\left(r\right)\right| \\
 \lesssim&_{N, n} \sum_{\alpha=0}^{n-1} \frac{1}{r^{\alpha+1}}\frac{\eta_{2}}{\left(\eta_{2} R\right)^{N+1}}\left(r\right)^{N+\alpha-n+2} \\
 \lesssim&_{n, N} \eta_2\left(\eta_{2}R\right)^{-n} . 
\end{align*}
This completes the proof of \eqref{psi2est}.

 (iii) By Lemma \ref{lemmavarphi}(iii), we have
$\operatorname{supp}(\varphi_{R,3})\subset[\eta_2R,(2+8\eta_2)R]$,
which implies \eqref{psi32}. For $\eta_2R<r\leq(2+8\eta_2)R$, we use
\[
 \psi_{R,3}^{(n)}(r)=\int_0^1s^n\varphi_{R,3}^{(n)}(sr)\,ds
\]
and \eqref{varphi3est} to obtain the second bound in \eqref{psi3est}.
For $r>(2+8\eta_2)R$, direct differentiation gives
\[
 \psi_{R,3}^{(n)}(r)=\frac{(-1)^nn!}{r^{n+1}}
 \int_0^{(2+8\eta_2)R}\varphi_{R,3}(s)\,ds.
\]
Since the integral is bounded by a constant times $R$, the proof of \eqref{psi3est} is complete.
\end{proof}

Next, we define the vector-valued map $a_{R}$.
\begin{definition}[$a_{R}$ map]\label{defa}
   We define the  vector-valued map $a_{R}:[0, \infty) \times \mathbb{R}^{d} \rightarrow \mathbb{R}^{d}$ by the formula
\begin{equation*}
a_{R}(t, z):=\psi_{R}(|z|)  z, \quad \forall(t, z) \in[0, \infty) \times \mathbb{R}^{d} .
\end{equation*}
We denote the $j^{\text {th }}$ component, for $j \in\{1,\cdots,d\}$, of $a_{R}$ by $a_{R}^{j}$. 
\end{definition}

\begin{remark}[Uniform boundedness of $a_R$]
By direct verification, for all $\eta_2\leq1$ and $R\geq1$, the inequality  
\begin{equation}\label{arbound}
 \max \left\{ \|a_R\|_{L_{t, x}^{\infty}\left([0, \infty) \times \mathbb{R}^d\right)},\; \sup_{x \in \mathbb{R}^d} |x| \|\nabla a_R(x)\|_{L_t^{\infty}([0, \infty))} \right\} \leq C_0  R 
\end{equation}
holds uniformly for some positive constant $C_0$.
\end{remark}

\begin{remark}[Divergence of $a_{R}$]\label{dera}
From Definition \eqref{defa}, \eqref{equ:idenpsi} and the chain rule, we observe that the spatial divergence of $a_{R}$ is given by the formula
    \begin{align*}
\left(\nabla_{x} \cdot a_{R}\right)( z) & =d \psi_{R}(|z|) +\left(\partial_{r} \psi_{R}\right)(|z|) |z| \\
& =(d-1) \psi_{R}(|z|)+\varphi_{R}\left(|z|\right).
\end{align*}
\end{remark}

We next record the derivative estimates for $a_R$. By \eqref{equ:psiRdef},
\[
 \psi_R(r)=\int_0^1\varphi_R(sr)\,ds,\qquad
 \psi_R'(r)=\int_0^1s\varphi_R'(sr)\,ds,\qquad
 \psi_R''(r)=\int_0^1s^2\varphi_R''(sr)\,ds.
\]
Differentiating \eqref{equ:idenpsi}, we obtain
\[
 r\psi_R''(r)=\varphi_R'(r)-2\psi_R'(r).
\]
It follows from the derivative estimates for $\varphi_R$ that
\[
 |\psi_R'(r)|+|r\psi_R''(r)|\lesssim_d\frac1{\eta_2R}.
\]
Moreover, since $\varphi_R$ is smooth and even, $\varphi_R'(0)=0$. Hence
\[
 \frac{\varphi_R'(r)}r=\int_0^1\varphi_R''(sr)\,ds,\qquad
 \frac{\psi_R'(r)}r=\int_0^1s^2\frac{\varphi_R'(sr)}{sr}\,ds,
\]
where the quotients at the origin are defined by continuity. Consequently,
\[
 |\varphi_R''(r)|+\left|\frac{\varphi_R'(r)}r\right|
 +|\psi_R''(r)|+\left|\frac{\psi_R'(r)}r\right|
 \lesssim_d\frac1{(\eta_2R)^2}.
\]
By direct differentiation of $a_R(t,x)=x\psi_R(|x|)$ and Remark \ref{dera},
\[
 \|\nabla_x^2a_R\|_{L^\infty}\lesssim_d\frac1{\eta_2R},\qquad
 \|\Delta_x\nabla_x\cdot a_R\|_{L^\infty}
 \lesssim_d\frac1{(\eta_2R)^2}.
\]
Indeed, the second spatial derivatives of $a_R$ are bounded by
$|\psi_R'(r)|+|r\psi_R''(r)|$, up to a constant depending on $d$, while
\[
 \Delta_x\nabla_x\cdot a_R
 =\varphi_R''+\frac{d-1}{r}\varphi_R'
 +(d-1)\psi_R''+\frac{(d-1)^2}{r}\psi_R'.
\]
Finally, the monotonicity and definition of $\varphi_R$ imply
$0\leq\varphi_R(r)\leq\psi_R(r)\leq1$. For $r\leq\eta_2R$,
\[
 \varphi_{R,1}(r)=\psi_{R,1}(r)=C_{\eta_2},\qquad
 \varphi_{R,3}(r)=\psi_{R,3}(r)=0.
\]
Thus \eqref{varphi2est} and \eqref{psi2est} give
$|\psi_R(r)-\varphi_R(r)|\lesssim_d\eta_2$ on this interval. Combining
this with the uniform bounds for $r>\eta_2R$, we conclude that
\[
 0\leq\psi_R(r)-\varphi_R(r)
 \lesssim_d\eta_2+1_{[\eta_2R,\infty)}(r).
\]

Now we have all the necessary ingredients to define our interaction Morawetz functional $M_{R}$.

\begin{definition}[Interaction Morawetz functional] Let $w:=P_{\leq \eta_{\ast}^{-1} K} u$. Suppose that $\int_{0}^{T} N(t)^{3} d t=K$. We define $M_{R}:[0, \infty) \rightarrow \mathbb{R}$ by the formula
\begin{equation*}
M_{R}(t):=2 \int_{\mathbb{R}^{2d}} a_{R}(t, x-y) \cdot|w(t, y)|^{2} \operatorname{Im}\{\bar{w} \nabla w\}(t, x) d x d y, \quad \forall t \in[0, \infty).
\end{equation*}
\end{definition}

Since $w=P_{\leq\eta_{\ast}^{-1}K}u$, we have
\[
 (i\partial_t+\Delta)w=\mu F(w)+\mu\mathcal N.
\]
By direct computation,
\[
 \partial_t|w|^2=-2\partial_k\operatorname{Im}(\bar w\partial_kw)
 +2\mu\operatorname{Im}(\bar w\mathcal N)
\]
and
\begin{align*}
 \partial_t\operatorname{Im}(\bar w\partial_jw)
 ={}&-2\partial_k\operatorname{Re}(\overline{\partial_kw}\partial_jw)
 +\tfrac12\partial_j\Delta|w|^2\\
 &-\mu|w|^2\partial_j(V*|w|^2)
 +\mu\operatorname{Re}(\overline{\mathcal N}\partial_jw-\bar w\partial_j\mathcal N).
\end{align*}
Therefore, differentiating $M_R$ and integrating in time, we obtain
\begin{align}
&M_{R}(T)-M_{R}(0)\notag\\
=&  -4 \int_{0}^{T} \int_{\R^{2d}} a_{R}^{j}(t, x-y)|w(t, y)|^{2} \partial_{k} \operatorname{Re}\left\{\overline{\partial_{k} w} \partial_{j} w\right\}(t, x) d x d y d t \label{mo1}\\
& -4 \int_{0}^{T} \int_{\R^{2d}} a_{R}^{j}(t, x-y) \partial_k\operatorname{Im}\left\{\bar{w} \partial_{k} w\right\}(t, y) \operatorname{Im}\left\{\bar{w} \partial_{j} w\right\}(t, x) d x d y d t \label{mo2} \\
& -2\mu\int_{0}^{T} \int_{\R^{2d}} a_{R}(t, x-y) \cdot|w(t, y)|^{2}\left(\left(\nabla V *|w|^{2}\right)|w|^{2}\right)(t, x) d x d y d t \label{mo3}\\
& +\int_{0}^{T} \int_{\R^{2d}} a_{R}^{j}(t, x-y)|w(t, y)|^{2} \partial_{j} \partial_{k}^{2}\left(|w|^{2}\right)(t, x) d x d y d t \label{mo4}\\
& +4\mu \int_{0}^{T} \int_{\R^{2d}} a_{R}(t, x-y) \cdot \operatorname{Im}\{\bar{w} \mathcal{N}\}(t, y) \operatorname{Im}\{\bar{w} \nabla w\}(t, x) d x d y d t \label{mo6}\\
& +2\mu\int_{0}^{T} \int_{\R^{2d}} a_{R}(t, x-y) \cdot|w(t, y)|^{2} \operatorname{Re}\{\bar{\mathcal{N}}  \nabla w\}(t, x) d x d y d t \label{mo7}\\
& -2\mu \int_{0}^{T} \int_{\R^{2d}} a_{R}(t, x-y) \cdot|w(t, y)|^{2} \operatorname{Re}\{\bar{w} \nabla \mathcal{N}\}(t, x) d x d y d t,\label{mo8}
\end{align}
where $\mathcal{N}:=P_{\leq \eta_{\ast}^{-1}K}F(u)-F(P_{\leq \eta_{\ast}^{-1}K}u)$. Recalling \eqref{kineticb} and \eqref{arbound}, by H\"older's inequality, we see that
\begin{equation}\label{F-est1}
    \sup_{0\leq t\leq T}|M_{R}(t)|\lesssim_u Ro(K)
\end{equation}
as $K\to\infty$.

We next treat  terms \eqref{mo1}-\eqref{mo8}.

$\bullet$ Estimate of \eqref{mo6}+\eqref{mo7}+\eqref{mo8}: By Proposition \ref{reminder} and \eqref{arbound}, 
    \begin{equation}\label{F-est2}
        |\eqref{mo6}+\eqref{mo7}+\eqref{mo8}|=Ro(K)\qtq{uniformly for all}0<\eta_2\ll1.
    \end{equation}

$\bullet$ Estimate of \eqref{mo4}: Thrice integrating by parts in the variable $x$, we obtain that
\begin{equation*}
\eqref{mo4}=-\int_{0}^{T} \int_{\mathbb{R}^{2d}}\left(\Delta_{x} \nabla_{x} \cdot a_{R}\right)(t, x-y)|w(t, y)|^{2}|w(t, x)|^{2} d x d y d t.
\end{equation*}
Now using the identity for the divergence $\nabla_{x} \cdot a_{R}$ given in Remark \ref{dera}, we see that
\begin{equation*}
\left(\Delta_{x} \nabla_{x} \cdot a_{R}\right)(t, x-y)= \Delta_{x}\left(\varphi_{R}\left(|x-y|\right)+(d-1) \psi_{R}(|x-y|)\right), 
\end{equation*}
so that the preceding derivative estimates give
\begin{equation*}
\left|\Delta_{x}\left(\varphi_{R}\left(|x-y|\right)+(d-1) \psi_{R}(|x-y|)\right)\right| \lesssim \frac{1}{\left(\eta_{2} R\right)^{2}}, \quad \forall(t, x, y) \in[0, \infty) \times \mathbb{R}^{d} \times \mathbb{R}^{d} . 
\end{equation*}
So by conservation of mass and the Fubini-Tonelli theorem, we conclude that
\begin{equation}\label{F-est3}
   |\eqref{mo4}|\lesssim  \frac{K}{\left(\eta_{2} R\right)^{2}} . 
\end{equation}

 $\bullet$ Estimate of \eqref{mo1}+\eqref{mo2}:  we find that
\begin{align*}
\eqref{mo1}+\eqref{mo2}= & 4 \int_{0}^{T} \int_{\mathbb{R}^{2d}}\left(\partial_{k} a_{R}^{j}\right)(t, x-y)|w(t, y)|^{2} \operatorname{Re}\left\{\overline{\partial_{k} w} \partial_{j} w\right\}(t, x) d x d y d t \\
& -4 \int_{0}^{T} \int_{\mathbb{R}^{2d}}\left(\partial_{k} a_{R}^{j}\right)(t, x-y) \operatorname{Im}\left\{\bar{w} \partial_{k} w\right\}(t, y) \operatorname{Im}\left\{\bar{w} \partial_{j} w\right\}\left(t, x) dx d y d t\right. \\
= & : \operatorname{Term}_1+\operatorname{Term}_2 
\end{align*}

We introduce the following notation for the radial component of the gradient centered at the point $y \in \mathbb{R}^{d}$ :
\begin{align*}
& \nabla_{r a d, y} f(x):=\left(\nabla f(x) \cdot \frac{(x-y)}{|x-y|}\right) \frac{(x-y)}{|x-y|},  \\
& \nabla_{r a d, y}^{\perp} f(x):=\nabla f(x)-\nabla_{r a d, y} f(x) . 
\end{align*}
Using the calculus identity
\begin{equation}\label{lllk}
\left(\partial_{k} a_{R}^{j}\right)(t, z)=\left(\left(\partial_{r} \psi_{R}\right)(|z|) \frac{z_{j} z_{k}}{|z|}+\delta_{j k} \psi_{R}(|z|)\right), \quad \forall(t, z) \in[0, \infty) \times \mathbb{R}^{d}, 
\end{equation}
we obtain by direct computation that
\begin{align}\notag
\operatorname{Term}_{1}= & 4 \int_{0}^{T} \int_{\mathbb{R}^{2d}} \varphi_{R}\left(|x-y|\right)|\nabla w(t, x)|^{2}|w(t, y)|^{2} d x d y d t \\\notag
& +4 \int_{0}^{T} \int_{\mathbb{R}^{2d}} \left(\psi_{R}(|x-y|)-\varphi_{R}\left(|x-y|\right)\right)\left|\nabla_{r a d, y}^{\perp} w(t, x)\right|^{2}|w(t, y)|^{2} d x d y d t \\
= & \operatorname{Term}_{1,1}+\operatorname{Term}_{1,2}
\end{align}
and
\begin{align}\notag
\operatorname{Term}_{2}= & -4 \int_{0}^{T} \int_{\mathbb{R}^{2d}} \varphi_{R}\left(|x-y|\right) \operatorname{Im}\{\bar{w} \nabla w\}(t, y) \cdot \operatorname{Im}\{\bar{w} \nabla w\}(t, x) d x d y d t \\\notag
& -4 \int_{0}^{T} \int_{\mathbb{R}^{2d}} \left(\psi_{R}(|x-y|)-\varphi_{R}\left(|x-y|\right)\right)\notag\\
&\qquad\times \operatorname{Im}\left\{\bar{w} \nabla_{rad, x}^{\perp} w\right\}(t, y) \cdot \operatorname{Im}\left\{\bar{w} \nabla_{rad, y}^{\perp} w\right\}(t, x) d x d y d t\\
&:=\operatorname{Term}_{2,1}+\operatorname{Term}_{2,2}.
\end{align}
We first claim the nonnegativity condition $\operatorname{Term}_{1,2}+\operatorname{Term}_{2,2} \geq 0$, and hence the sum can be harmlessly discarded. Indeed, since the function $\varphi_{R}$ is nonincreasing, unpacking the definition of $\psi_{R}$, we see that
\begin{equation*}
\psi_{R}(|x-y|)-\varphi_{R}\left(|x-y|\right) \geq 0.
\end{equation*}
So by Cauchy-Schwarz and symmetry under swapping $x$ and $y$, we obtain the estimate
\begin{equation*}
\left|\operatorname{Term}_{2,2}\right| \leq 4 \int_{0}^{T} \int_{\mathbb{R}^{2d}} \left(\psi_{R}(|x-y|)-\varphi_{R}\left(|x-y|\right)\right)|w(t, y)|^{2}\left|\nabla_{r a d, y}^{\perp} w(t, x)\right|^{2} d x d y d t, 
\end{equation*}
which implies the claim.

To analyze the sum $\operatorname{Term}_{1,1}+\operatorname{Term}_{2,1}$, we use a Galilean invariance trick from \cite{Dodson-Adv}. More precisely, we define a map $\xi:[0, \infty) \times \mathbb{R}^{d} \rightarrow \mathbb{R}^{d}$
\begin{equation*}
\xi(t, z):=
\begin{cases}
0, \quad \text{ if }\ \int_{\mathbb{R}^{d}} \zeta_{R}\left(\left| x-z\right|\right)|w(t, x)|^{2} d x=0,\\
    \frac{\int_{\mathbb{R}^{d}} \zeta_{R}\left(\left| x-z\right|\right) \operatorname{Im}\{\bar{w} \nabla w\}(t, x) d x}{\int_{\mathbb{R}^{d}} \zeta_{R}\left(\left| x-z\right|\right)|w(t, x)|^{2} d x} \quad \text{otherwise}.
\end{cases}
\end{equation*}
 With this choice of $\xi(t, z)$, we see from unpacking the definition of $\varphi_{R}$ and using the Fubini-Tonelli theorem to swap the order of the $d x d y$ and $d z$ integrations that under the Galilean transformation
\begin{align*}
& w \mapsto e^{-i x \cdot \xi(t, z)} w(t, x)=: v_{z}(t, x)=: v(t, x, z)  \\
& \operatorname{Term}_{1,1}+\operatorname{Term}_{2,1} \mapsto \frac{4}{\left\|\zeta_{R}(|\cdot|)\right\|_{L^{1}}} \int_{0}^{T} d t  \int_{\mathbb{R}^{3d}} d x d y d z \zeta_{R}\left(\left|x-z\right|\right) \zeta_{R}\left(\left|y-z\right|\right) \\
& \times\left|\left(\nabla_{x} v\right)(t, x, z)\right|^{2}|v(t, y,z)|^{2}.
\end{align*}
Our goal is to bound from below the expression appearing in the right-hand side. To accomplish this goal, we first introduce a bump function $\vartheta_{R}: \mathbb{R} \rightarrow[0, \infty)$ defined by the formula 
\begin{equation}\label{equ:varthetadef}
\vartheta_{R}(r):=\frac{2}{\eta_{2} R} \int_{\mathbb{R}} 1_{\left[-\left(1-4 \eta_{2} \right)R,\left(1-4 \eta_{2}\right) R\right]}(r-s) \chi\left(\frac{2 s}{\eta_{2} R}\right) d s.
\end{equation}

One  may easily check that $\vartheta_{R}$ is $C^{\infty}, 0 \leq \vartheta_{R} \leq 1$, and
\[
\vartheta_{R}(r)=\left\{\begin{array}{ll}
1, & |r| \leq\left(1-6 \eta_{2}\right) R \\
0, & |r| \geq\left(1-2 \eta_{2}\right) R,
\end{array} \quad \text { and } \quad\left\|\vartheta_{R}^{(n)}\right\|_{L^{\infty}} \lesssim_{n}\left(\eta_{2} R\right)^{-n}, \quad \forall n \in \mathbb{N}.\right.
\]
Since $\zeta_{R} \equiv 1$ on the support of $\vartheta_{R}$, we have the lower bound
\begin{equation*}
\int_{\mathbb{R}^{d}} \zeta_{R}\left(\left| x-z\right|\right)\left|\nabla_{x} v(t, x, z)\right|^{2} d x \geq \int_{\mathbb{R}^{d}} \vartheta_{R}^{2}\left(\left|x-z\right|\right)\left|\nabla_{x} v(t, x, z)\right|^{2} d x. 
\end{equation*}
Out of convenience, we introduce the additional notation
\begin{equation*}
\vartheta_{R, z}(t, x):=\vartheta_{R}\left(\left| x-z\right|\right), \quad(t, x, z) \in[0, \infty) \times \mathbb{R}^{d} \times \mathbb{R}^{d}. 
\end{equation*} Then we have already proved that
\begin{align}\label{F-est4}
 \eqref{mo1}+\eqref{mo2}
 &\geq\frac{4}{\|\zeta_R(|\cdot|)\|_{L^1}}
   \int_0^T\int_{\R^{3d}}\zeta_R(|y-z|)\notag\\
 &\qquad\times|\vartheta_{R,z}(t,x)\nabla_xv_z(t,x)|^2
           |v_z(t,y)|^2\,dx\,dy\,dz\,dt.
\end{align}

$\bullet$ Estimate of \eqref{mo3}: By the oddness of $\nabla V$ and
symmetry under swapping $x$ and $z$, we have
\begin{align*}
 &\int_{\R^{2d}}a_R(t,x-y)\cdot(\nabla V*|w|^2)(t,x)
                   |w(t,y)|^2|w(t,x)|^2\,dx\,dy\\
 &=\frac12\int_{\R^{3d}}
       \big(a_R(t,x-y)-a_R(t,z-y)\big)\cdot\nabla V(x-z)\\
 &\qquad\times|w(t,y)|^2|w(t,x)|^2|w(t,z)|^2\,dx\,dy\,dz.
\end{align*}
By Taylor's theorem, we have that for each component $j \in\{1,\cdots,d\}$,
\begin{equation*}
a_{R}^{j}(t, x-y)-a_{R}^{j}(t, z-y)=\nabla_{x} a_{R}^{j}(t, x-y) \cdot(x-z)-\sum_{|\alpha|=2}(z-x)^{\alpha} \mathcal{R}_{\alpha}^{j}(t, x, y, z),
\end{equation*}
where
\[
\mathcal{R}_{\alpha}^{j}(t, x, y, z):=\frac{2}{\alpha!} \int_{0}^{1}(1-\theta) \partial^{\alpha} a_{R}^{j}(t, x-y+\theta(z-x)) d \theta.
\]
Therefore,
\begin{align}
\eqref{mo3}= & 2\mu \int_0^T dt\int_{\mathbb{R}^{2d}} d x d y \int_{\mathbb{R}^{d}} d z \partial_{k} a_{R}^{j}(t, x-y) \frac{(x-z)_{k}(x-z)_{j}}{|x-z|^{4}}|w(t, y)|^{2}|w(t, x)|^{2}|w(t, z)|^{2} \label{eest1}\\
& -2\mu \sum_{|\alpha|=2}\int_0^Tdt \int_{\mathbb{R}^{2d}} d x d y \int_{\mathbb{R}^{d}} d z(z-x)^{\alpha} \mathcal{R}_{\alpha}^{j}(t, x, y, z) \frac{(x-z)_{j}}{|x-z|^{4}}|w(t, x)|^{2}|w(t, y)|^{2}|w(t, z)|^{2} \label{eest2}.
\end{align}
We first treat \eqref{eest1}. In view of the identity \eqref{lllk}, using Remark \ref{dera} we have
\begin{align*}
\partial_{k} a_{R}^{j}(t, x-y) \frac{(x-z)_{k}(x-z)_{j}}{|x-z|^{4}}
={}&\varphi_R(|x-y|)V(x-z)\\
&+(\psi_R-\varphi_R)(|x-y|)
\frac{|x-z|^2-\left((x-z)\cdot\frac{x-y}{|x-y|}\right)^2}{|x-z|^4},
\end{align*}
then by Lemma \ref{quasisolitonlemma1},
\begin{align}
& 2\mu \int_{0}^{T} \int_{\mathbb{R}^{3d}} \partial_{k} a_{R}^{j}(t, x-y) \frac{(x-z)_{k}(x-z)_{j}}{|x-z|^{4}}|w(t, y)|^{2}|w(t, x)|^{2}|w(t, z)|^{2} d x d y d z d t\notag \\
& = 2\mu \int_{0}^{T} \int_{\mathbb{R}^{2d}} \varphi_{R}\left(|x-y|\right)|w(t, y)|^{2}\left(|w|^{2}\left(V *|w|^{2}\right)\right)(t, x) d x d y d t\notag  \\
& -O(\int_{0}^{T} \int_{|x-y| \geq \eta_2R} |w(t, y)|^{2}\left(|w|^{2}\left(V *|w|^{2}\right)\right)(t, x) d x d y d t) \notag \\
& \quad- \eta_{2} O(\int_{0}^{T} \int_{\mathbb{R}^{2d}} |w(t, y)|^{2}\left(|w|^{2}\left(V *|w|^{2}\right)\right)(t, x) d x d y d t)\\
& = 2\mu\int_{0}^{T} \int_{\mathbb{R}^{2d}} \varphi_{R}\left(|x-y|\right)|w(t, y)|^{2}\left(|w|^{2}\left(V *|w|{ }^{2}\right)\right)(t, x) d x d y d t \notag \\
&-o_{\eta_2 R}(1)K-\eta_2O(K).\label{eest1q}
\end{align}
Next we treat \eqref{eest2}. A simple computation shows
\begin{align*}
\partial_{l} \partial_{k} a_{R}^{j}(t, y)= & \left(\partial_{r}^{2} \psi_{R}\right)(|y|) \frac{y_{l} y_{k} y_{j}}{|y|^{2}} \\
& +\left(\partial_{r} \psi_{R}\right)(|y|) \frac{\left(\delta_{k l} y_{j}+\delta_{j l} y_{k}\right)}{|y|}  \\
& -\left(\partial_{r} \psi_{R}\right)(|y|) \frac{y_{k} y_{j} y_{l}}{|y|^{3}} \\
& +\delta_{j k}\left(\partial_{r} \psi_{R}\right)(|y|) \frac{y_{l}}{|y|}.
\end{align*}
 Hence by $\|\nabla_x^2a_R\|_{L^\infty}\lesssim_d(\eta_2R)^{-1}$, Hardy-Littilewood-Sobolev's inequality, H\"older's inequality the fact that $\|u\|^q_{L_t^qL_x^r([0,T]\times\R^d)}\lesssim_{q}1+ \int_0^T N(t)^2 dt$ for any admissible pair $(q,r)$,
\begin{align}
& \left|2 \sum_{|\alpha|=2} \int_{0}^{T} \int_{\mathbb{R}^{3d}} \frac{(z-x)^{\alpha}(x-z)_{j}}{|x-z|^{4}} \mathcal{R}_{\alpha}^{j}(t, x, y, z)|w(t, x)|^{2}|w(t, y)|^{2}|w(t, z)|^{2} d x d y d z d t \right|\notag \\
 \lesssim& \frac{1}{\eta_2 R}\int_{0}^{T} \int_{\mathbb{R}^{2d}} |w(t, y)|^{2}\left(|w|^{2}\left(\tilde{V} *|w|^{2}\right)\right)(t, x) d x d y d t \notag \\
\lesssim& \frac{K}{\eta_2 R},\label{eest2q}
\end{align}
where $\tilde{V}(x-z):=|x-z|^{-1}$.

Collecting \eqref{eest1q} and \eqref{eest2q}, we obtain that
\begin{align}\label{interme}
    \eqref{mo3}\geq& 2\mu \int_{0}^{T} \int_{\mathbb{R}^{2d}} \varphi_{R}\left(|x-y|\right)|w(t, y)|^{2}\left(|w|^{2}\left(V *|w|^{2}\right)\right)(t, x) d x d y d t \notag\\
&-\eta_2O(K)-o_{\eta_2R}(K).
\end{align}

Next, we proceed to decompose the first term of RHS of \eqref{interme} as
\begin{align}
& 2\mu \int_{0}^{T} \int_{\mathbb{R}^{2d}} \varphi_{R}\left(|x-y|\right)|w(t, y)|^{2}\left(|w|^{2}\left(V *|w|^{2}\right)\right)(t, x) d x d y d t \notag\\
=&\frac{2\mu}{\left\|\zeta_{R}(|\cdot|)\right\|_{L^{1}}} \int_{0}^{T} \int_{\mathbb{R}^{3d}}  \zeta_{R}\left(\left| y-z\right|\right)|w(t, y)|^{2}\left|\left(\vartheta_{R,z} v_{z}\right)(t, x)\right|^{2}\notag\\
&\hspace{26ex}\times\left(V *\left|\vartheta_{R,z} v_{z}\right|^{2}\right)(t, x) d z d x d y d t\label{eeest1}\\
&+ \frac{2\mu}{\left\|\zeta_{R}(|\cdot|)\right\|_{L^{1}}} \int_{0}^{T} d t \int_{\mathbb{R}^{3d}} d z d x d y  \zeta_{R}\left(\left| y-z\right|\right)\left(\zeta_{R}\left(\left| x-z\right|\right)-\vartheta_{R}^{2}\left(\left|x-z\right|\right)\right)\notag \\
&\hspace{26ex}\times|w(t, y)|^{2}\left(|w|^{2}\left(V *|w|^{2}\right)\right)(t, x) \label{eeest0} \\
&+\frac{2\mu}{\left\|\zeta_{R}(|\cdot|)\right\|_{L^{1}}} \int_{0}^{T} \int_{|x-y|\geq\frac{\eta_2 R}{4}}  \zeta_{R}\left(\left| y-z\right|\right)|w(t, y)|^{2}\left|\left(\vartheta_{R,z} v_{z}\right)(t, x)\right|^{2}\notag\\
&\hspace{26ex}\times\left(V *[(1-\vartheta^2_{R,z} )\left|w\right|^{2}]\right)(t, x) d z d x d y d t\label{eeest2}\\
&+\frac{2\mu}{\left\|\zeta_{R}(|\cdot|)\right\|_{L^{1}}} \int_{0}^{T} \int_{|x-y|\leq\frac{\eta_2 R}{4}}  \zeta_{R}\left(\left| y-z\right|\right)|w(t, y)|^{2}\left|\left(\vartheta_{R,z} v_{z}\right)(t, x)\right|^{2}\notag\\
&\hspace{26ex}\times\left(V_{loc} *[(1-\vartheta^2_{R,z} )\left|w\right|^{2}]\right)(t, x) d z d x d y d t\label{eeest3}\\
&+\frac{2\mu}{\left\|\zeta_{R}(|\cdot|)\right\|_{L^{1}}} \int_{0}^{T} \int_{|x-y|\leq\frac{\eta_2 R}{4}}  \zeta_{R}\left(\left| y-z\right|\right)|w(t, y)|^{2}\left|\left(\vartheta_{R,z} v_{z}\right)(t, x)\right|^{2}\notag\\
&\hspace{26ex}\times\left(V_{glob} *[(1-\vartheta^2_{R,z} )\left|w\right|^{2}]\right)(t, x) d z d x d y d t,\label{eeest4}
\end{align}
where
\[
V_{loc}(x):=V(x) 1_{[0, \frac{\eta_2 R}{4}]}(|x|)\qtq{and}V_{glob}(x):=V(x) 1_{[\frac{\eta_2 R}{4},\infty)}(|x|).
\]
We proceed to estimate \eqref{eeest0}-\eqref{eeest4} individually.

$\bullet$ For \eqref{eeest0},  recalling  that $\zeta_R \equiv 1$ on the interval $[0, R]$ and $\zeta_R \equiv 0$ on the interval $\left[\left(1+4 \eta_2\right) R, \infty\right)$, and  $\vartheta_R \equiv 1$ on the interval $\left[0,\left(1-6 \eta_2\right) R\right]$ and $\vartheta_R \equiv 0$ on the interval $\left[\left(1-2 \eta_2\right) R, \infty\right)$, we see that
\[
 \operatorname{supp}\left(\zeta_R-\vartheta_R^2\right) \subset\left[\left(1-6 \eta_2\right) R,\left(1+4 \eta_2\right) R\right].
\]
Since $0 \leq \zeta_R, \vartheta_R \leq 1$ by construction, it then follows that
\begin{align*}
& \frac{1}{\left\|\zeta_R(|\cdot|)\right\|_{L^1}} \int_{\mathbb{R}^d} \zeta_R\left(\left|\frac{y}{R} -z\right|\right)\left(\zeta_R\left(\left|\frac{x}{R} -z\right|\right)-\vartheta_R^2\left(\left|\frac{x}{R}-z\right|\right)\right) d z \\
 \leq& \frac{1}{\left\|\zeta_R(|\cdot|)\right\|_{L^1}} \int_{\left(1-6 \eta_2\right) R \leq|z| \leq\left(1+4 \eta_2\right) R} d z \\
 \lesssim& \eta_2,
\end{align*}
where the ultimate equality follows from $\left\|\zeta_R(|\cdot|)\right\|_{L^1} \sim R^d$. So by the Fubini-Tonelli theorem and the preceding estimate,
\begin{align}
|\eqref{eeest0}|& \lesssim \eta_2 \int_0^T \int_{\mathbb{R}^{2d}} |w(t, y)|^2\left(|w|^2\left(V *|w|^2\right)\right)(t, x) d x d y d t\notag \\
& \lesssim \eta_2 K.\label{aeeest0}
\end{align}

$\bullet$ For \eqref{eeest2}, we observe that
\begin{equation*}
\left|\left(\vartheta_{R,  z} v_{z}\right)(t, x)\right|^{2}\left(V *\left(\left(1-\vartheta_{R,  z}^{2}\right)\left|v_{z}\right|^{2}\right)\right)(t, x) \lesssim|w(t, x)|^{2}\left(V *|w|^{2}\right)(t, x)  
\end{equation*}
and hence by Lemma \ref{quasisolitonlemma1}, we see that
\begin{align}
    |\eqref{eeest2}|& \lesssim \int_{0}^{T} \int_{|x-y|\geq\frac{\eta_{2} R}{4 }} |w(t, y)|^{2}|w(t, x)|^{2}\left(V *|w|^{2}\right)(t, x) d x d y d t\notag\\
& \leq \int_{0}^{T} \left(\int_{|y-x(t)|\geq\frac{\eta_{2} R}{8}}|w(t, y)|^{2} d y\right)| ||w(t)|^{2}\left(V *|w|^{2}\right)(t) \|_{L_{x}^{1}} d t\notag \\
& \quad+M(u) \int_{0}^{T} \left(\int_{|x-x(t)|\geq\frac{\eta_{2} R}{8 }}|w(t, x)|^{2}\left(V *|w|^{2}\right)(t, x) d x\right) d t\notag\\
&\lesssim o_{\eta_2 R}(K). \label{aeeest2}
\end{align}

$\bullet$ For \eqref{eeest3}, we first claim that
\begin{equation*}
1_{\leq \frac{\eta_{2} R}{4 }}(|x-y|)\left(V_{loc} *\left(\left(1-\vartheta_{R, z}^{2}\right)|w|^{2}\right)\right)(t, x) \neq 0 \Longrightarrow\left| y-z\right|>\left(1-7 \eta_{2}\right) R. 
\end{equation*}
Indeed, observe that
\begin{align*}
\left(V_{loc} *\left(\left(1-\vartheta_{R, z}^{2}\right)|w|^{2}\right)\right)(t, x)&=\int_{\left|z^{\prime}\right| \leq \frac{\eta_{2} R}{4 }} V\left(z^{\prime}\right)\left|w\left(t, x-z^{\prime}\right)\right|^{2}\notag\\
&\hspace{5ex}\times\left(1-\vartheta_{R}^{2}\left(\left|\left(x-z^{\prime}\right)-z\right|\right)\right) d z^{\prime}. 
\end{align*}

Since $\operatorname{supp}\left(1-\vartheta_{R}^{2}\right) \subset\left[\left(1-6 \eta_{2}\right) R, \infty\right)$, there must exist $z_0\in\R^d$, $|z_0|\leq \frac{\eta_2R}{4}$ such that
\begin{equation*}
\left|x-z_0-z\right| \geq\left(1-6 \eta_{2}\right) R, 
\end{equation*}
which by the reverse triangle inequality implies that
\begin{align*}
\left|y-z\right| & =\left|(y-x)+\left(x-z^{\prime}\right)+ z^{\prime}-z\right| \\
& \geq\left|x-z_0-z\right|-\left(|y-x|+\left|z^{\prime}\right|\right) \\
& \geq\left(1-6 \eta_{2}\right) R-\frac{\eta_{2} R}{2} \\
& >\left(1-7 \eta_{2}\right) R,
\end{align*}
for all $|x-y| \leq \frac{\eta_2 R}{4}$. Hence, we can use the support of $\zeta_{R}$, H\"older's inequality, Hardy-Littlewood-Sobolev's inequality to obtain
\begin{align}
    |\eqref{eeest3}|&\lesssim \frac{1}{\left\|\zeta_{R}(|\cdot|)\right\|_{L^{1}}} \int_{0}^{T} d t  \int_{\mathbb{R}^{3d}} d z d x d y 1_{\left(1-7 \eta_{2}\right) R \leq \cdot \leq\left(1+4 \eta_{2}\right) R}\left(\left|y-z\right|\right)\notag\\
&\hspace{18ex}\times|w(t, y)|^{2}\left(|w|^{2}\left(V *|w|^{2}\right)\right)(t, x) \notag\\
&\lesssim \eta_{2} \int_{0}^{T} \int_{\mathbb{R}^{2d}} |w(t, y)|^{2}\left(|w|^{2}\left(V *|w|^{2}\right)\right)(t, x) d x d y d t\notag\\
    &\lesssim \eta_2K.\label{aeeest3}
\end{align}

$\bullet$ For \eqref{eeest4}, we first observe from the support of $V_{glob}$ the simple estimate
\begin{align*}
\left(V_{glob} *\left(\left(1-\vartheta_{R,  z}^{2}\right)\left|v_{z}\right|^{2}\right)\right)(t, x)  \lesssim \int_{\left|z^{\prime}\right|>\frac{\eta_{2} R}{4}} \frac{\left|w\left(t, x-z^{\prime}\right)\right|^{2}}{\left|z^{\prime}\right|^{2}} d z^{\prime}\lesssim\left(\frac{1}{\eta_{2} R}\right)^{2} M(u). 
\end{align*}
Therefore,
\begin{align}
|\eqref{eeest4}| &\lesssim \frac{1}{\left\|\zeta_{R}(|\cdot|)\right\|_{L^{1}}\left(\eta_{2} R\right)^{2}} \int_{0}^{T} \int_{\mathbb{R}^{3d}}  \zeta_{R}\left(\left| y-z\right|\right)|w(t, y)|^{2}|w(t, x)|^{2} d z d x d y d t \notag\\
& \lesssim \frac{K}{(\eta_2 R)^2}\label{aeeest4}
\end{align}
Collecting estimates \eqref{aeeest0}-\eqref{aeeest4}, in view of \eqref{interme}-\eqref{eeest4}, we readily obtain that
\begin{align}
    \eqref{mo3}&\geq \frac{2\mu}{\left\|\zeta_{R}(|\cdot|)\right\|_{L^{1}}} \int_{0}^{T} \int_{\mathbb{R}^{3d}}  \zeta_{R}\left(\left| y-z\right|\right)|w(t, y)|^{2}\left|\left(\vartheta_{R,z} v_{z}\right)(t, x)\right|^{2}\notag\\
    &\hspace{12ex}\times \left(V *\left|\vartheta_{R,z} v_{z}\right|^{2}\right)(t, x) d z d x d y d t-\eta_2 K-o_{\eta_2R}(1)K.\label{F-est5}
\end{align}

In view of \eqref{F-est1}--\eqref{F-est5}, we obtain
\begin{align}\label{interme2}
 &\frac4{\|\zeta_R(|\cdot|)\|_{L^1}}
 \int_0^T\int_{\R^{3d}}\zeta_R(|y-z|)|w(t,y)|^2
 |\vartheta_{R,z}\nabla v_z(t,x)|^2\,dx\,dy\,dz\,dt\notag\\
 &+\frac{2\mu}{\|\zeta_R(|\cdot|)\|_{L^1}}
 \int_0^T\int_{\R^{3d}}\zeta_R(|y-z|)|w(t,y)|^2|\vartheta_{R,z}v_z(t,x)|^2\notag\\
 &\hspace{18ex}\times(V*|\vartheta_{R,z}v_z|^2)(t,x)\,dx\,dy\,dz\,dt\notag\\
 &\lesssim\eta_2K+o_{\eta_2R}(1)K+Ro(K).
\end{align}
Next, we use the identity
\begin{align*}
 |\vartheta_{R,z}v_z(t,x)|^2
 ={}&\left|\tilde{1}_{[0,1]}\left(\frac{|x-x(t)|}{C(\eta)}\right)
       \vartheta_{R,z}v_z(t,x)\right|^2\\
 &+\left[1-\tilde{1}_{[0,1]}\left(\frac{|x-x(t)|}{C(\eta)}\right)^2\right]
       |\vartheta_{R,z}v_z(t,x)|^2.
\end{align*}
By the symmetry and nonnegativity of $V$, together with
$|v_z|=|w|$ and $0\leq\vartheta_{R,z}\leq1$, we obtain
\begin{align*}
 0\leq{}&\int_{\R^d}|\vartheta_{R,z}v_z(t,x)|^2
                (V*|\vartheta_{R,z}v_z|^2)(t,x)\,dx\\
 &-\int_{\R^d}\left|\tilde{1}_{[0,1]}\left(\frac{|x-x(t)|}{C(\eta)}\right)
                   \vartheta_{R,z}v_z(t,x)\right|^2\\
 &\hspace{5ex}\times\left(V*\left|\tilde{1}_{[0,1]}
            \left(\frac{|\cdot-x(t)|}{C(\eta)}\right)
            \vartheta_{R,z}v_z\right|^2\right)(t,x)\,dx\\
 \leq{}&2\int_{\R^d}\left[1-\tilde{1}_{[0,1]}
            \left(\frac{|x-x(t)|}{C(\eta)}\right)^2\right]
            |\vartheta_{R,z}v_z(t,x)|^2\\
 &\hspace{5ex}\times(V*|\vartheta_{R,z}v_z|^2)(t,x)\,dx\\
 \leq{}&2\int_{|x-x(t)|\geq C(\eta)}
           |w(t,x)|^2(V*|w|^2)(t,x)\,dx.
\end{align*}
Moreover, by the Fubini--Tonelli theorem and conservation of mass,
\[
 \frac1{\|\zeta_R(|\cdot|)\|_{L^1}}
 \int_{\R^{2d}}\zeta_R(|y-z|)|w(t,y)|^2\,dy\,dz
 =\|w(t)\|_{L^2}^2\leq M(u).
\]
Multiplying the preceding difference by
$\|\zeta_R(|\cdot|)\|_{L^1}^{-1}\zeta_R(|y-z|)|w(t,y)|^2$
and integrating in $t,y,z$, we can apply \eqref{truncateu2}, with
$N=\tilde N=1$, to bound the resulting error by
\[
 2M(u)\int_0^T\int_{|x-x(t)|\geq C(\eta)}
 |w(t,x)|^2(V*|w|^2)(t,x)\,dx\,dt
 \lesssim_u o_{C(\eta)}(1)K.
\]
Combining this estimate with \eqref{interme2}, we conclude that
\begin{align}\label{finalz1}
 &\frac4{\|\zeta_R(|\cdot|)\|_{L^1}}
 \int_0^T\int_{\R^{3d}}\zeta_R(|y-z|)|w(t,y)|^2
 |\vartheta_{R,z}\nabla v_z(t,x)|^2\,dx\,dy\,dz\,dt\notag\\
 &+\frac{2\mu}{\|\zeta_R(|\cdot|)\|_{L^1}}
 \int_0^T\int_{\R^{3d}}\zeta_R(|y-z|)|w(t,y)|^2\notag\\
 &\hspace{5ex}\times\left|\tilde{1}_{[0,1]}\left(\frac{|x-x(t)|}{C(\eta)}\right)
                 \vartheta_{R,z}v_z(t,x)\right|^2\notag\\
 &\hspace{5ex}\times\left(V*\left|\tilde{1}_{[0,1]}
                 \left(\frac{|\cdot-x(t)|}{C(\eta)}\right)
                 \vartheta_{R,z}v_z\right|^2\right)(t,x)\,dx\,dy\,dz\,dt\notag\\
 &\lesssim\eta_2K+o_{\eta_2R}(1)K+o_{C(\eta)}(1)K+Ro(K).
\end{align}

Now we choose $x^{\ast}=x^{\ast}(t,z)$ with $|x^{\ast}-x(t)|\leq2C(\eta)$ such that
\[
 \vartheta_{R,z}(x^{\ast})
 =\inf_{|x-x(t)|\leq2C(\eta)}\vartheta_{R,z}(x).
\]
Then by the fundamental theorem of calculus, for $|x-x(t)|\leq2C(\eta)$,
\[
 \vartheta_{R,z}(x)=\vartheta_{R,z}(x^{\ast})
 +O\left(\frac{C(\eta)}{\eta_2R}\right).
\]
Since $0\leq\vartheta_{R,z}\leq1$, the error in replacing the four cutoff
factors in the potential term of \eqref{finalz1} by
$\vartheta_{R,z}(x^{\ast})$ is bounded, after integration in $t,y,z$, by
\[
 C\frac{C(\eta)}{\eta_2R}
 \int_0^T\int_{\R^d}|w(t,x)|^2(V*|w|^2)(t,x)\,dx\,dt
 \lesssim_u\frac{C(\eta)}{\eta_2R}K.
\]
Moreover, integration by parts gives
\begin{align*}
 &\int_{\R^d}\tilde{1}_{[0,1]}
       \left(\frac{|x-x(t)|}{C(\eta)}\right)^2|\nabla v_z(t,x)|^2\,dx\\
 &=\left\|\nabla\left[\tilde{1}_{[0,1]}
       \left(\frac{|\cdot-x(t)|}{C(\eta)}\right)v_z\right]\right\|_{L_x^2}^2\\
 &\quad+\int_{\R^d}\tilde{1}_{[0,1]}
       \left(\frac{|x-x(t)|}{C(\eta)}\right)
       \Delta_x\left[\tilde{1}_{[0,1]}
       \left(\frac{|x-x(t)|}{C(\eta)}\right)\right]|v_z(t,x)|^2\,dx.
\end{align*}
The last integral is bounded in absolute value by $CC(\eta)^{-2}$.
Using $|\vartheta_{R,z}(x^{\ast})|^2\geq|\vartheta_{R,z}(x^{\ast})|^4$
and the definition of $E$, we therefore obtain
\begin{align}\label{finalz2}
 \text{left-hand side of }\eqref{finalz1}
 &\geq\frac8{\|\zeta_R(|\cdot|)\|_{L^1}}
 \int_0^T\int_{\R^{2d}}\zeta_R(|y-z|)|w(t,y)|^2
 |\vartheta_{R,z}(x^{\ast})|^4\notag\\
 &\qquad\times E\left(\tilde{1}_{[0,1]}
       \left(\frac{|\cdot-x(t)|}{C(\eta)}\right)v_z\right)\,dy\,dz\,dt\notag\\
 &\quad-C\left(\frac{C(\eta)}{\eta_2R}
                         +\frac1{C(\eta)^2}\right)K.
\end{align}
In view of \eqref{interme2}, \eqref{finalz1} and \eqref{finalz2}, we have proved that
\begin{align}\label{mmmn}
 &\frac1{\|\zeta_R(|\cdot|)\|_{L^1}}
 \int_0^T\int_{\R^{2d}}\zeta_R(|y-z|)|w(t,y)|^2
 |\vartheta_{R,z}(x^{\ast})|^4\notag\\
 &\qquad\times E\left(\tilde{1}_{[0,1]}
       \left(\frac{|\cdot-x(t)|}{C(\eta)}\right)v_z\right)\,dy\,dz\,dt\notag\\
 &\lesssim_u\left(\eta_2+o_{\eta_2R}(1)+o_{C(\eta)}(1)
       +\frac{C(\eta)}{\eta_2R}+\frac1{C(\eta)^2}\right)K+Ro(K).
\end{align}
\begin{proof}[Proof of Theorem \ref{quasi} when $N(t)\equiv1$]
For $|z-x(t)|\leq R/2$ and $R\gg C(\eta)/\eta_2$, we have
$\vartheta_{R,z}(x^{\ast})=1$. The sharp Gagliardo--Nirenberg inequality
in the focusing case, and the Sobolev interpolation inequality in either case, give
\[
 E\left(\tilde{1}_{[0,1]}
       \left(\frac{|\cdot-x(t)|}{C(\eta)}\right)v_z\right)
 \gtrsim_u\int_{\R^d}\left|\tilde{1}_{[0,1]}
       \left(\frac{|x-x(t)|}{C(\eta)}\right)w(t,x)\right|^{\frac{2(d+2)}d}\,dx.
\]
Hence by \eqref{localized} and Lemma \ref{quasisolitonlemma2},
\begin{align}\label{zvc}
 K\lesssim_u{}&\frac1{\|\zeta_R(|\cdot|)\|_{L^1}}
 \int_0^T\int_{\R^{2d}}\zeta_R(|y-z|)|w(t,y)|^2
 |\vartheta_{R,z}(x^{\ast})|^4\notag\\
 &\qquad\times E\left(\tilde{1}_{[0,1]}
       \left(\frac{|\cdot-x(t)|}{C(\eta)}\right)v_z\right)\,dy\,dz\,dt.
\end{align}
Choose $\eta$ so that $C(\eta)$ is sufficiently large, then choose $\eta_2$
sufficiently small and $R\gg C(\eta)/\eta_2$. Keeping these parameters
fixed and letting $K\to\infty$, \eqref{mmmn} contradicts \eqref{zvc}.
\end{proof}

\subsection{$N(t)$ varies}
In this general case, let $\tilde N(t)\leq N(t)$ be given by the smoothing
algorithm described below, and let $[0,T]$ be a union of complete small
intervals $J_l$. We take
\[
 w(t,x):=P_{\leq\eta_{\ast}^{-1}K}u(t,x),\qquad
 K:=\int_0^TN(t)^3\,dt,
\]
and
\[
 a_R(t,x-y):=\psi_R(\tilde N(t)|x-y|)\tilde N(t)(x-y).
\]
We define $M_R$ as before:
\[
 M_R(t):=2\int_{\R^{2d}}a_R(t,x-y)\cdot|w(t,y)|^2
                  \operatorname{Im}\{\bar w\nabla w\}(t,x)\,dx\,dy.
\]
For use in this subsection, we set
\[
 \vartheta_{R,z}(t,x):=\vartheta_R(\tilde N(t)|x-z|),\qquad
 v_z(t,x):=e^{-ix\cdot\xi(t,z)}w(t,x),
\]
where
\[
 \xi(t,z):=
 \begin{cases}
 \displaystyle\frac{\int_{\R^d}\zeta_R(\tilde N(t)|x-z|)
                  \operatorname{Im}\{\bar w\nabla w\}(t,x)\,dx}
                  {\int_{\R^d}\zeta_R(\tilde N(t)|x-z|)|w(t,x)|^2\,dx},
       &\displaystyle\int_{\R^d}\zeta_R(\tilde N(t)|x-z|)|w(t,x)|^2\,dx>0,\\
 0, &\text{otherwise}.
 \end{cases}
\]
By a change of variables in the definition of $\varphi_R$,
\begin{align*}
 \varphi_R(\tilde N(t)|x-y|)
 ={}&\frac{\tilde N(t)^d}{\|\zeta_R(|\cdot|)\|_{L^1}}\\
 &\quad\times\int_{\R^d}\zeta_R(\tilde N(t)|x-z|)
                          \zeta_R(\tilde N(t)|y-z|)\,dz.
\end{align*}
In particular,
\[
 \int_{\R^d}\zeta_R(\tilde N(t)|x-z|)
             \operatorname{Im}\{\bar v_z\nabla v_z\}(t,x)\,dx=0.
\]
Choose $x^{\ast}=x^{\ast}(t,z)$ such that
\[
 |x^{\ast}-x(t)|\leq\frac{2C(\eta)}{\tilde N(t)},\qquad
 \vartheta_{R,z}(t,x^{\ast})
 =\inf_{|x-x(t)|\leq2C(\eta)/\tilde N(t)}\vartheta_{R,z}(t,x).
\]
Then on this ball,
\[
 \vartheta_{R,z}(t,x)=\vartheta_{R,z}(t,x^{\ast})
                      +O\left(\frac{C(\eta)}{\eta_2R}\right).
\]
Differentiating $M_R$, we have
\begin{align}
 M_R(T)-M_R(0)
 ={}&\eqref{mo1}+\eqref{mo2}+\eqref{mo3}+\eqref{mo4}
       +\eqref{mo6}+\eqref{mo7}+\eqref{mo8}\label{mo50}\\
 &+2\int_0^T\int_{\R^{2d}}(\partial_ta_R)(t,x-y)\cdot|w(t,y)|^2
       \operatorname{Im}\{\bar w\nabla w\}(t,x)\,dx\,dy\,dt.\label{mo5}
\end{align}
By \eqref{kineticb}, Proposition \ref{reminder} and \eqref{arbound},
\begin{equation}\label{finala0}
 |M_R(T)|+|M_R(0)|+|\eqref{mo6}+\eqref{mo7}+\eqref{mo8}|
 \lesssim_u Ro(K).
\end{equation}

We first estimate the terms \eqref{mo1}--\eqref{mo4}. The derivative
estimates following Remark \ref{dera} give
\[
 \|\nabla_x^2a_R(t)\|_{L^\infty}
 \lesssim_d\frac{\tilde N(t)^2}{\eta_2R},\qquad
 \|\Delta_x\nabla_x\cdot a_R(t)\|_{L^\infty}
 \lesssim_d\frac{\tilde N(t)^3}{(\eta_2R)^2}.
\]
Also,
\[
 0\leq(\psi_R-\varphi_R)(\tilde N(t)|x-y|)
 \lesssim_d\eta_2+
 1_{\{|x-y|\geq\eta_2R/\tilde N(t)\}}.
\]
On each $J_l$, the local Strichartz estimate and the Hardy--Littlewood--Sobolev
inequality yield
\[
 \int_{J_l}\int_{\R^d}|w(t,x)|^2(V*|w|^2)(t,x)\,dx\,dt
 \lesssim_u1.
\]
By Cauchy--Schwarz, first in the spatial variables and then in time,
\[
 \int_{J_l}\int_{\R^d}|w(t,x)|^2(\tilde V*|w|^2)(t,x)\,dx\,dt
 \lesssim_u|J_l|^{1/2}\lesssim_u N(J_l)^{-1},
\]
where $\tilde V(x)=|x|^{-1}$ is as in \eqref{eest2q}. Using the local
comparability of $\tilde N$ on $J_l$ and $\tilde N\leq N$, we obtain
\begin{align*}
 &\frac1{\eta_2R}\int_0^T\tilde N(t)^2
       \int_{\R^d}|w(t,x)|^2(\tilde V*|w|^2)(t,x)\,dx\,dt\\
 &\lesssim_u\frac1{\eta_2R}\sum_{J_l\subset[0,T]}
                  \frac{(\sup_{J_l}\tilde N)^2}{N(J_l)}
 \lesssim_u\frac1{\eta_2R}\int_0^T\tilde N(t)N(t)^2\,dt.
\end{align*}
Thus the Taylor remainder has the same relative bound as in \eqref{eest2q}.
Similarly, \eqref{mo4} is bounded in absolute value by
\[
 \frac{C}{(\eta_2R)^2}\int_0^T\tilde N(t)^3\,dt
 \leq\frac{C}{(\eta_2R)^2}\int_0^T\tilde N(t)N(t)^2\,dt.
\]
Since $\tilde N\leq N$, Lemma \ref{quasisolitonlemma1} applies on the regions
$|x-y|\geq\eta_2R/\tilde N(t)$ and
$|x-x(t)|\geq C(\eta)/\tilde N(t)$. Consequently, the matrix remainder in
\eqref{eest1q} and the potential localization errors are bounded by
\[
 C\left(\eta_2+o_{\eta_2R}(1)+o_{C(\eta)}(1)
                  +\frac1{(\eta_2R)^2}\right)
       \int_0^T\tilde N(t)N(t)^2\,dt.
\]
The argument of \eqref{finalz2}, with the cutoff
$\tilde{1}_{[0,1]}(\tilde N(t)|x-x(t)|/C(\eta))$, gives the additional error
\[
 C\left(\frac{C(\eta)}{\eta_2R}+\frac1{C(\eta)^2}\right)
       \int_0^T\tilde N(t)N(t)^2\,dt.
\]
Here the second term follows from the cutoff Laplacian bound
$C\tilde N(t)^2/C(\eta)^2$ and $\tilde N^3\leq\tilde NN^2$.

We also retain part of the localized kinetic term. Indeed, integration by parts gives
\begin{align*}
 \|\nabla(\vartheta_{R,z}v_z)\|_{L_x^2}^2
 ={}&\int_{\R^d}\vartheta_{R,z}(t,x)^2|\nabla v_z(t,x)|^2\,dx\\
 &-\int_{\R^d}\vartheta_{R,z}(t,x)\Delta_x\vartheta_{R,z}(t,x)
                       |v_z(t,x)|^2\,dx\\
 \leq{}&\int_{\R^d}\zeta_R(\tilde N(t)|x-z|)|\nabla v_z(t,x)|^2\,dx
                +\frac{C\tilde N(t)^2}{(\eta_2R)^2}.
\end{align*}
For $\mu=-1$, the sharp Gagliardo--Nirenberg inequality therefore implies
\begin{align*}
 &4\int_{\R^d}\zeta_R(\tilde N(t)|x-z|)|\nabla v_z(t,x)|^2\,dx\\
 &\quad-2\int_{\R^d}|\vartheta_{R,z}v_z(t,x)|^2
                     (V*|\vartheta_{R,z}v_z|^2)(t,x)\,dx\\
 &\geq4\left(1-\frac{M(u)}{M(Q)}\right)
       \int_{\R^d}\zeta_R(\tilde N(t)|x-z|)|\nabla v_z(t,x)|^2\,dx
       -\frac{C\tilde N(t)^2}{(\eta_2R)^2}.
\end{align*}
For $\mu=1$, the potential term is nonnegative. Averaging this lower bound
with the lower bound obtained from \eqref{finalz2}, we find a constant
$\delta_0=\delta_0(u)>0$ such that
\begin{align}\label{mmmn2}
 &\eqref{mo1}+\eqref{mo2}+\eqref{mo3}+\eqref{mo4}\notag\\
 &\geq\frac{\delta_0}{\|\zeta_R(|\cdot|)\|_{L^1}}
       \int_0^T\tilde N(t)^{d+1}\int_{\R^{2d}}
       \zeta_R(\tilde N(t)|y-z|)|w(t,y)|^2
       |\vartheta_{R,z}(t,x^{\ast})|^4\notag\\
 &\qquad\times E\left(\tilde{1}_{[0,1]}
       \left(\frac{\tilde N(t)|\cdot-x(t)|}{C(\eta)}\right)v_z\right)
                 \,dy\,dz\,dt\notag\\
 &\quad+\frac{\delta_0}{\|\zeta_R(|\cdot|)\|_{L^1}}
       \int_0^T\tilde N(t)^{d+1}\int_{\R^{3d}}
       \zeta_R(\tilde N(t)|y-z|)\zeta_R(\tilde N(t)|x-z|)\notag\\
 &\qquad\times |w(t,y)|^2|\nabla v_z(t,x)|^2\,dx\,dy\,dz\,dt\notag\\
 &\quad-C\left(\eta_2+o_{\eta_2R}(1)+o_{C(\eta)}(1)
       +\frac1{\eta_2R}+\frac1{(\eta_2R)^2}
       +\frac{C(\eta)}{\eta_2R}+\frac1{C(\eta)^2}\right)
       \int_0^T\tilde N(t)N(t)^2\,dt.
\end{align}

It remains to consider \eqref{mo5}. By \eqref{equ:idenpsi},
\[
 \partial_ta_R(t,x-y)=\tilde N'(t)(x-y)
                            \varphi_R(\tilde N(t)|x-y|).
\]
Using the convolution identity above and the choice of $\xi(t,z)$, we have
\begin{align*}
 \eqref{mo5}
 ={}&\frac2{\|\zeta_R(|\cdot|)\|_{L^1}}
       \int_0^T\tilde N'(t)\tilde N(t)^d
       \int_{\R^d}\left(\int_{\R^d}
          \zeta_R(\tilde N(t)|y-z|)|w(t,y)|^2\,dy\right)\\
 &\quad\times\left(\int_{\R^d}\zeta_R(\tilde N(t)|x-z|)
       \operatorname{Im}\{\bar v_z\nabla v_z\}(t,x)\cdot(x-z)\,dx\right)
       \,dz\,dt.
\end{align*}
By Cauchy--Schwarz and Young's inequality,
\begin{align}
 |\eqref{mo5}|
 &\leq\frac{C\delta}{\|\zeta_R(|\cdot|)\|_{L^1}}
       \int_0^T\tilde N(t)^{d+1}\int_{\R^{3d}}
       \zeta_R(\tilde N(t)|y-z|)\zeta_R(\tilde N(t)|x-z|)\notag\\
 &\qquad\times |w(t,y)|^2|\nabla v_z(t,x)|^2\,dx\,dy\,dz\,dt\label{bbv1}\\
 &\quad+\frac{C}{\delta\|\zeta_R(|\cdot|)\|_{L^1}}
       \int_0^T|\tilde N'(t)|^2\tilde N(t)^{d-1}\int_{\R^{3d}}
       \zeta_R(\tilde N(t)|y-z|)\zeta_R(\tilde N(t)|x-z|)\notag\\
 &\qquad\times |w(t,y)|^2|w(t,x)|^2|x-z|^2\,dx\,dy\,dz\,dt.\label{bbv2}
\end{align}
Choose $\delta>0$ such that $C\delta\leq\delta_0/2$. Subtracting
\eqref{bbv1} from the second positive term in \eqref{mmmn2}, we obtain
\begin{align}\label{finala}
 &\frac{\delta_0-C\delta}{\|\zeta_R(|\cdot|)\|_{L^1}}
       \int_0^T\tilde N(t)^{d+1}\int_{\R^{3d}}
       \zeta_R(\tilde N(t)|y-z|)\zeta_R(\tilde N(t)|x-z|)\notag\\
 &\qquad\times |w(t,y)|^2|\nabla v_z(t,x)|^2\,dx\,dy\,dz\,dt
       \geq0.
\end{align}
On the support of $\zeta_R(\tilde N(t)|x-z|)$,
$|x-z|\leq(1+4\eta_2)R/\tilde N(t)$. Moreover,
\[
 \int_{\R^d}\zeta_R(\tilde N(t)|y-z|)\,dz
 =\tilde N(t)^{-d}\|\zeta_R(|\cdot|)\|_{L^1}.
\]
Thus conservation of mass yields
\begin{equation}\label{finalb}
 |\eqref{bbv2}|\leq\frac{C_1R^2}{\delta}
               \int_0^T\frac{|\tilde N'(t)|^2}{\tilde N(t)^3}\,dt.
\end{equation}

We now use the smoothing algorithm in \cite[Section 4]{Dodson-Adv}.
For each fixed $n\geq1$, we can choose $\tilde N$ and a sequence of endpoints
$T\to\infty$ of small intervals such that
\begin{gather*}
 \frac{N(t)}{C(n)}\leq\tilde N(t)\leq N(t),\qquad
 |\tilde N'(t)|\leq C_2\tilde N(t)^3,\\
 \int_0^T|\tilde N'(t)|\,dt
 \leq\frac{C_3}{n}\int_0^T\int_{\R^d}
          \tilde N(t)|u(t,x)|^{\frac{2(d+2)}d}\,dx\,dt.
\end{gather*}
The constants $C_2,C_3$ are independent of $n,T$. In particular,
$|\tilde N'|/\tilde N\lesssim_u N^2$, which gives local comparability
of $\tilde N$ on each $J_l$. Summing the local spacetime estimates over
these intervals, we obtain
\begin{align*}
 \int_0^T\frac{|\tilde N'(t)|^2}{\tilde N(t)^3}\,dt
 &\leq C_2\int_0^T|\tilde N'(t)|\,dt\\
 &\lesssim_u\frac1n\int_0^T\tilde N(t)N(t)^2\,dt.
\end{align*}
Combining \eqref{mo50}--\eqref{finalb}, we conclude that
\begin{align}\label{mmmn20}
 &\frac{\delta_0}{\|\zeta_R(|\cdot|)\|_{L^1}}
       \int_0^T\tilde N(t)^{d+1}\int_{\R^{2d}}
       \zeta_R(\tilde N(t)|y-z|)|w(t,y)|^2
       |\vartheta_{R,z}(t,x^{\ast})|^4\notag\\
 &\qquad\times E\left(\tilde{1}_{[0,1]}
       \left(\frac{\tilde N(t)|\cdot-x(t)|}{C(\eta)}\right)v_z\right)
                 \,dy\,dz\,dt\notag\\
 &\lesssim_u\left(\eta_2+o_{\eta_2R}(1)+o_{C(\eta)}(1)
       +\frac1{\eta_2R}+\frac1{(\eta_2R)^2}
       +\frac{C(\eta)}{\eta_2R}+\frac1{C(\eta)^2}
       +\frac{R^2}{\delta n}\right)
       \int_0^T\tilde N(t)N(t)^2\,dt+Ro(K).
\end{align}
For $|z-x(t)|\leq R/(2\tilde N(t))$ and $R\gg C(\eta)/\eta_2$,
$\vartheta_{R,z}(t,x^{\ast})=1$. By \eqref{truncateu} and mass conservation,
\[
 \int_{\R^d}\zeta_R(\tilde N(t)|y-z|)|w(t,y)|^2\,dy\gtrsim_u1
\]
on this region for sufficiently large $K,R$. The volume of the $z$-ball
is comparable to $R^d\tilde N(t)^{-d}$. Therefore, as in \eqref{zvc}, the
sharp Gagliardo--Nirenberg inequality and Sobolev interpolation give
\begin{align}\label{finalm}
 \text{left-hand side of }\eqref{mmmn20}
 &\gtrsim_u\int_0^T\tilde N(t)
       \int_{|x-x(t)|\leq C(\eta)/\tilde N(t)}
       |w(t,x)|^{\frac{2(d+2)}d}\,dx\,dt\notag\\
 &\gtrsim_u\int_0^T\tilde N(t)N(t)^2\,dt,
\end{align}
where the last inequality follows from Lemma \ref{quasisolitonlemma2}
and $\tilde N\leq N$.

\begin{proof}[Proof of Theorem \ref{quasi} in the general case]
Fix $\delta>0$ as in \eqref{finala}. Choose $\eta$ so that $C(\eta)$ is
sufficiently large, then choose $\eta_2$ sufficiently small,
$R\gg C(\eta)/\eta_2$, and finally $n\gg_u R^2/\delta$.
The terms multiplying $\int_0^T\tilde N(t)N(t)^2\,dt$ in \eqref{mmmn20}
are then absorbed into \eqref{finalm}. Since $\tilde N\geq N/C(n)$, we obtain
\begin{equation}\label{finaln}
 \frac{K}{C(n)}\leq\int_0^T\tilde N(t)N(t)^2\,dt
                    \lesssim_u Ro(K).
\end{equation}
Keeping the parameters fixed and letting $K\to\infty$ along the chosen
sequence, we arrive at a contradiction. This completes the proof.
\end{proof}

\appendix 

\section{Analysis at the minimal blow-up mass}\label{app}

As discussed in the introduction, the uniqueness of the ground state is known only for $d=4$ (and is otherwise unknown). 
Therefore, in this appendix we assume uniqueness for dimensions other than $d=4$.

The argument in the proof of Theorem \ref{quasi} can also be used to study the threshold dynamics, 
leading to the following theorem.

\begin{theorem}
Let $u$ be an almost-periodic solution on $[0,\infty)$  with frequency function $N(t)\leq 1$. 
If $\mu=-1$ and $M(u)=M(Q)$, then there exist a sequence $t_n\geq 0$ and parameters 
$(\gamma_n,x_n,\xi_n,N_n)\in\mathbb{R}\times\mathbb{R}^d\times\mathbb{R}^d\times\mathbb{R}_+$ such that
\begin{equation}\label{sequence}
\|e^{i\gamma_n}e^{ix\cdot\xi_n}N_n^{\frac{d}{2}}u(t_n,N_n(x-x_n))-Q\|_{L^2}\to 0.
\end{equation}
\end{theorem}

\begin{remark}
This theorem will play a key role in a future paper, where we establish the complete classification of minimal-mass blow-up solutions to \eqref{NLH} with $\mu=-1$.
\end{remark}

\begin{remark}
Analogous results are known for the mass-critical NLS; we refer the interested reader to \cite{Dodson-SIAM,Dodson-NA,F}.
\end{remark}

\begin{proof}
If $\int_0^\infty N(t)^3\,dt<\infty$, let
$\xi_0=\lim_{t\to\infty}\xi(t)$. The argument in the proof of
Theorem \ref{rapidcascade} gives $E(e^{-ix\cdot\xi_0}u_0)=0$.
By \cite[Proposition 3.2]{MiaoXuZhaoblowup}, we have
\[
 u(t,x)=\lambda_0^{d/2}e^{i\gamma_0-it|\xi_0|^2}e^{ix\cdot\xi_0}
       e^{i\lambda_0^2t}Q(\lambda_0(x-x_0-2\xi_0t))
\]
for some $\lambda_0>0$, $\gamma_0\in\R$ and $x_0,\xi_0\in\R^d$.
In this case $N(t)\sim\lambda_0$, which contradicts
$\int_0^\infty N(t)^3\,dt<\infty$.

If $\int_0^\infty N(t)^3\,dt=\infty$, suppose that \eqref{sequence}
fails. Then there exists $\eps>0$ such that
\[
 \inf_{\substack{t\geq0,\ \gamma_0\in\R,\ \lambda_0>0\\
                  x_0,\xi_0\in\R^d}}
 \left\|e^{i\gamma_0}e^{ix\cdot\xi_0}\lambda_0^{d/2}
 u\bigl(t,\lambda_0(x-x_0)\bigr)-Q\right\|_{L^2}\geq4\eps.
\]
As in Section \ref{quasisec}, choose the parameters satisfying
\eqref{xxibound}, with $\xi(0)=0$, and set
\[
 K=\int_0^TN(t)^3\,dt,\qquad w=P_{\leq\eta_{\ast}^{-1}K}u.
\]
For $\eta_{\ast}$ sufficiently small, \eqref{xxibound}, $N(t)\leq1$
and \eqref{localized} imply
\begin{equation}\label{appd5}
 \lim_{T\to\infty}\sup_{0\leq t\leq T}
 \|u(t)-w(t)\|_{L^2}
 =\lim_{T\to\infty}\sup_{0\leq t\leq T}
 \|P_{\geq\eta_{\ast}^{-1}K}u(t)\|_{L^2}=0,
\end{equation}
where $P_{\geq\eta_{\ast}^{-1}K}=1-P_{\leq\eta_{\ast}^{-1}K}$.
Let $\tilde N,\vartheta_{R,z},\xi(t,z)$ and $v_z$ be as in the general
case of Section \ref{quasisec}.

We claim that there exists $c_0\in(0,1/2]$ such that, for all
sufficiently large $K$, uniformly in $t,z$ and the spatial cutoff parameters,
\begin{equation}\label{appd3}
\begin{split}
 E(\vartheta_{R,z}v_z)
 &\geq c_0\|\nabla(\vartheta_{R,z}v_z)\|_{L^2}^2,\\
 E\left(\tilde{1}_{[0,1]}
       \left(\frac{\tilde N(t)|\cdot-x(t)|}{C(\eta)}\right)v_z\right)
 &\geq c_0\left\|\nabla\left(\tilde{1}_{[0,1]}
       \left(\frac{\tilde N(t)|\cdot-x(t)|}{C(\eta)}\right)v_z\right)
       \right\|_{L^2}^2.
\end{split}
\end{equation}
Otherwise, along a sequence of intervals with $K\to\infty$, there exist
times $t_n$, points $z_n$ and nonzero functions $w_n$ of one of the two
forms in \eqref{appd3} such that
\begin{equation}\label{appd0}
 M(w_n)\leq M(Q),\qquad
 0\leq E(w_n)\leq\frac1n\|\nabla w_n\|_{L^2}^2.
\end{equation}
The sharp Gagliardo--Nirenberg inequality gives
\[
 \frac{E(w_n)}{\|\nabla w_n\|_{L^2}^2}
 \geq\frac12\left(1-\frac{M(w_n)}{M(Q)}\right),
\]
so that
\begin{equation}\label{appd00}
 \|w_n\|_{L^2}\longrightarrow\|Q\|_{L^2}.
\end{equation}
Define
\[
 \tilde w_n(x):=
 \left(\frac{\|\nabla Q\|_{L^2}}{\|\nabla w_n\|_{L^2}}\right)^{d/2}
 w_n\left(\frac{\|\nabla Q\|_{L^2}}{\|\nabla w_n\|_{L^2}}x\right).
\]
Then
\[
 \|\tilde w_n\|_{L^2}\to\|Q\|_{L^2},\qquad
 \|\nabla\tilde w_n\|_{L^2}=\|\nabla Q\|_{L^2},\qquad
 E(\tilde w_n)\to0.
\]
Following the proofs of Propositions 3.1 and 3.3 in \cite{MiaoXuZhaoblowup},\footnote{Miao, Xu
and Zhao treat the case $d=4$. Under the uniqueness assumption for the
ground state, their argument extends readily to every dimension $d\geq3$.}
after passing to a subsequence, there exist $\theta\in\R$ and $x_n\in\R^d$
such that
\begin{equation}\label{finalblow}
 \|\tilde w_n(\cdot-x_n)-e^{i\theta}Q\|_{H^1}\longrightarrow0.
\end{equation}
On the other hand, since both cutoff functions take values in $[0,1]$,
\eqref{appd00} implies
\begin{equation}\label{appd4}
 \|v_{z_n}(t_n)-w_n\|_{L^2}^2
 \leq M(w(t_n))-M(w_n)\leq M(Q)-M(w_n)\longrightarrow0.
\end{equation}
Together with \eqref{appd5}, this yields
\begin{align*}
 \|u(t_n)-e^{ix\cdot\xi(t_n,z_n)}w_n\|_{L^2}
 &\leq\|u(t_n)-w(t_n)\|_{L^2}
       +\|v_{z_n}(t_n)-w_n\|_{L^2}\\
 &\longrightarrow0.
\end{align*}
Combining this with \eqref{finalblow}, we obtain \eqref{sequence},
a contradiction. This proves \eqref{appd3}.

By the first inequality in \eqref{appd3},
\[
 \int_{\R^d}|\vartheta_{R,z}v_z|^2
                  (V*|\vartheta_{R,z}v_z|^2)\,dx
 \leq(2-4c_0)\|\nabla(\vartheta_{R,z}v_z)\|_{L^2}^2.
\]
Therefore, integrating by parts as in the estimate preceding \eqref{mmmn2},
we have
\begin{align*}
 &4\int_{\R^d}\zeta_R(\tilde N(t)|x-z|)|\nabla v_z(t,x)|^2\,dx\\
 &\quad-2\int_{\R^d}|\vartheta_{R,z}v_z(t,x)|^2
                         (V*|\vartheta_{R,z}v_z|^2)(t,x)\,dx\\
 &\geq8c_0\int_{\R^d}\zeta_R(\tilde N(t)|x-z|)
                         |\nabla v_z(t,x)|^2\,dx
          -\frac{C\tilde N(t)^2}{(\eta_2R)^2}.
\end{align*}
Combining this estimate with \eqref{finalz2}, we obtain \eqref{mmmn2}
with $\delta_0=\delta_0(u,c_0)>0$. The second inequality in \eqref{appd3}
and Sobolev interpolation give \eqref{finalm}. Choose $\delta>0$ so that
$C\delta\leq\delta_0/2$. Then \eqref{bbv1} is absorbed as in
\eqref{finala}, while \eqref{bbv2} is estimated by \eqref{finalb}.
Proceeding as in \eqref{mmmn20}--\eqref{finaln}, we choose $C(\eta)$
sufficiently large, then $\eta_2$ sufficiently small,
$R\gg C(\eta)/\eta_2$ and $n\gg_uR^2/\delta$, to obtain
\[
 \frac{K}{C(n)}
 \leq\int_0^T\tilde N(t)N(t)^2\,dt\lesssim_u Ro(K).
\]
Keeping the parameters fixed and letting $K\to\infty$ along the
endpoints supplied by the smoothing algorithm, we arrive at a contradiction.
The proof is complete.
\end{proof}

\end{document}